\documentclass[11pt,oneside,reqno]{article}

\usepackage{amsmath,amsthm,amssymb,color,bm,graphicx}

\usepackage{array}
\usepackage[margin=2.5cm,a4paper]{geometry}
\usepackage{bbm}
\usepackage{paralist}
\usepackage{mathrsfs}
\usepackage[breaklinks=true]{hyperref}
\hypersetup{
colorlinks=true,
linkcolor=black,
urlcolor=blue,
citecolor=black
}

\usepackage[square]{natbib} %[authoryear,longnamesfirst]{natbib}

\renewcommand{\cite}{\citep*}

\numberwithin{equation}{section}
\allowdisplaybreaks[4]

\theoremstyle{plain}
\newtheorem{theorem}{Theorem}[section]
\newtheorem{proposition}[theorem]{Proposition}
\newtheorem{lemma}[theorem]{Lemma}
\newtheorem{corollary}[theorem]{Corollary}

\theoremstyle{definition}
\newtheorem{definition}[theorem]{Definition}

\newtheorem{remark}[theorem]{Remark}

\makeatletter

\renewcommand{\phi}{\varphi}
\newcommand{\eps}{\varepsilon}
\newcommand{\eq}{\eqref}

\newcommand{\dtv}{\mathop{d_{\mathrm{TV}}}}

\newcommand{\bigo}{\mathrm{O}}
\newcommand{\lito}{\mathrm{o}}
\def\tsfrac#1#2{{\textstyle\frac{#1}{#2}}}

\newcommand{\Bi}{\mathop{\mathrm{Bi}}}

\newcommand{\MN}{\mathop{\mathrm{MN}}}

\def\E{\mathbbm{E}}

\newcommand{\law}{\mathscr{L}}
\newcommand{\eqlaw}{\stackrel{d}{=}}

\newcounter{ctr}\loop\stepcounter{ctr}\edef\X{\@Alph\c@ctr}%
	\expandafter\edef\csname s\X\endcsname{\noexpand\mathscr{\X}}
	\expandafter\edef\csname c\X\endcsname{\noexpand\mathcal{\X}}
	\expandafter\edef\csname b\X\endcsname{\noexpand\boldsymbol{\X}}
	\expandafter\edef\csname I\X\endcsname{\noexpand\mathbbm{\X}}
	\expandafter\edef\csname r\X\endcsname{\noexpand\mathrm{\X}}
\ifnum\thectr<26\repeat

\def\be#1{\begin{equation*}#1\end{equation*}}
\def\ben#1{\begin{equation}#1\end{equation}}
\def\bes#1{\begin{equation*}\begin{split}#1\end{split}\end{equation*}}
\def\besn#1{\begin{equation}\begin{split}#1\end{split}\end{equation}}

\def\ba#1{\begin{align*}#1\end{align*}}
\def\ban#1{\begin{align}#1\end{align}}

\def\given{\typeout{Command 'given' should only be used within bracket command}}
\newcounter{@bracketlevel}
\def\@bracketfactory#1#2#3#4#5#6{
\expandafter\def\csname#1\endcsname##1{%
\addtocounter{@bracketlevel}{1}%
\global\expandafter\let\csname @middummy\alph{@bracketlevel}\endcsname\given%
\global\def\given{\mskip#5\csname#4\endcsname\vert\mskip#6}\csname#4l\endcsname#2##1\csname#4r\endcsname#3%
\global\expandafter\let\expandafter\given\csname @middummy\alph{@bracketlevel}\endcsname
\addtocounter{@bracketlevel}{-1}}%
}
\def\bracketfactory#1#2#3{%
\@bracketfactory{#1}{#2}{#3}{relax}{1mu plus 0.25mu minus 0.25mu}{0.6mu plus 0.15mu minus 0.15mu}
\@bracketfactory{b#1}{#2}{#3}{big}{1mu plus 0.25mu minus 0.25mu}{0.6mu plus 0.15mu minus 0.15mu}
\@bracketfactory{bb#1}{#2}{#3}{Big}{2.4mu plus 0.8mu minus 0.8mu}{1.8mu plus 0.6mu minus 0.6mu}
\@bracketfactory{bbb#1}{#2}{#3}{bigg}{3.2mu plus 1mu minus 1mu}{2.4mu plus 0.75mu minus 0.75mu}
\@bracketfactory{bbbb#1}{#2}{#3}{Bigg}{4mu plus 1mu minus 1mu}{3mu plus 0.75mu minus 0.75mu}
}
\bracketfactory{clc}{\lbrace}{\rbrace}
\bracketfactory{clr}{(}{)}
\bracketfactory{cls}{[}{]}

\def\bbklr#1{\Bigl(#1\Bigr)}

\def\norm#1{\Vert#1\Vert}
\def\bnorm#1{\bigl\Vert#1\bigr\Vert}
\def\bbnorm#1{\Bigl\Vert#1\Bigr\Vert}
\def\bbbnorm#1{\biggl\Vert#1\biggr\Vert}

\def\abs#1{\vert#1\vert}
\def\babs#1{\bigl\vert#1\bigr\vert}
\def\bbabs#1{\Bigl\vert#1\Bigr\vert}
\def\bbbabs#1{\biggl\vert#1\biggr\vert}

\def\bfloor#1{{\bigl\lfloor#1\bigr\rfloor}}

\def\ceil#1{{\lceil#1\rceil}}
\def\bceil#1{{\bigl\lceil#1\bigr\rceil}}
\def\bbceil#1{{\Bigl\lceil#1\Bigr\rceil}}
\def\bbbceil#1{{\biggl\lceil#1\biggr\rceil}}

\def\angle#1{{\langle#1\rangle}}
\def\bangle#1{{\bigl\langle#1\bigr\rangle}}
\def\bbangle#1{{\Bigl\langle#1\Bigr\rangle}}

\renewcommand\section{\@startsection {section}{1}{\z@}%
{-3.5ex \@plus -1ex \@minus -.2ex}%
{1.3ex \@plus.2ex}%
{\center\small\sc\mathversion{bold}\MakeUppercase}}

\def\subsection#1{\@startsection {subsection}{2}{0pt}%
{-3.5ex \@plus -1ex \@minus -.2ex}%
{1ex \@plus.2ex}%
{\bf\mathversion{bold}}{#1}}

\def\subsubsection#1{\@startsection{subsubsection}{3}{0pt}%
{\medskipamount}%
{-10pt}%
{\normalsize\itshape}{\kern-2.2ex. #1.}}

\def\blfootnote{\xdef\@thefnmark{}\@footnotetext}

\makeatother

\def\t#1{^{[#1]}}
\def\s#1{^{(#1)}}

\newcommand\tx{\textnormal{\textbf{x}}}

\newcommand{\PD}{\mathrm{PD}}
\newcommand{\DP}{\mathrm{DP}}

\def\BC{\mathrm{BC}}

\def\wh{\widehat}

\def\d{\delta}

\def\Lip{\mathrm{L}}
\def\Ind{\mathrm{\bI}}

\def\Delsthree#1#2#3{\cT_{#1,#2,#3}}

\def\wt{\widetilde}
\def\stleq{\leq_{\mathrm{st}}}

\begin{document}

\title{\sc\bf\large\MakeUppercase{Stein's method for the Poisson-Dirichlet distribution and the Ewens Sampling Formula, with applications to 
Wright-Fisher models
%via Fleming-Viot processes and stationary Wright-Fisher approximation
}}
\author{\sc Han~L.~Gan and Nathan Ross}
\date{\it Northwestern University and University of Melbourne}
\maketitle

\begin{abstract} 
We provide a general theorem bounding the error in the approximation of a random measure of interest--for example, the
empirical population measure of types in a Wright-Fisher model--and a Dirichlet process, which is a measure having Poisson-Dirichlet distributed atoms with i.i.d.\ labels from a diffuse distribution. The
implicit metric of the approximation theorem captures the sizes and locations of the masses, and so also yields 
bounds on the approximation between the masses of the measure of interest and the Poisson-Dirichlet distribution. 
We apply the result to bound the error in the approximation of the stationary distribution of types in the 
finite Wright-Fisher model with infinite-alleles mutation structure (not necessarily parent independent)
by the Poisson-Dirichlet distribution. An important consequence of our result is an explicit upper bound
on the total variation distance between the random partition generated by sampling from a finite Wright-Fisher 
stationary distribution, and the Ewens Sampling Formula. 
The bound is small if the sample size $n$ is much smaller than $N^{1/6}\log(N)^{-1/2}$,
where $N$ is the total population size. Our analysis requires a result of separate interest, giving
an explicit bound on the second moment of the number of types 
of a finite Wright-Fisher stationary distribution.
The general approximation result follows from a new development of Stein's method for the Dirichlet process, which follows by viewing the Dirichlet process as the stationary distribution of a Fleming-Viot process, and then applying Barbour's generator approach. 
\end{abstract}

% \noindent\textbf{Keywords: } 

\section{Introduction and results}
The one parameter Poisson-Dirichlet (PD) distribution is a probability measure on 
the infinite dimensional ordered simplex
\be{
\nabla_\infty:=\bclc{(p_1,p_2,\ldots): p_1\geq p_2 \geq \cdots, \sum_{i=1}^\infty p_i=1},
}
which is fundamental to combinatorial probability, population genetics and  Bayesian nonparametrics; see, for example, \cite{Arratia2003}, \cite{Pitman2006}, \cite{Ewens2004}, \cite{Ghosal2010}, \cite{Feng2010}. 
It can be defined in a number of ways, first in \cite{Kingman1975}, as a limit of symmetric Dirichlet distributions with components in decreasing order, and also as the distribution of the ordered and normalized jump-sizes in a gamma subordinator, as follows.
\begin{definition}[Poisson-Dirichlet distribution]\label{def:pd}
Let $\theta>0$, $\gamma_1>\gamma_2>\cdots$ be the points of a Poisson process on $\IR^+$ with intensity $\theta x^{-1} e^{-x}$, $x>0$, and $\gamma=\sum_{i=1}^\infty \gamma_i$. 
Define the probability measure on $\nabla_\infty$, by
\be{
\PD(\theta):=\law\bclr{\gamma_1/\gamma,\gamma_2/\gamma,\ldots}. 
} 
\end{definition}
Most relevant to our study, $\PD(\theta)$ is the distributional limit of a wide array of (scaled) random integer partitions, such as the cycle structure of random permutations, or the partition of types in a finite population model; see the references above for plenty of examples. In this paper, we derive a general result that bounds the error in the approximation of the normalized block sizes of a random partition by the Poisson-Dirichlet distribution, and apply it to the  frequencies of types in stationary distributions of some Wright-Fisher models.
%, and to the  normalized block sizes of the partition generated by the Ewens sampling formula. 

Working with non-increasing sequences presents technical difficulties, and so it is standard to use other orderings of the coordinates of~$\PD(\theta)$; for example, size-bias ordering gives the Griffiths-Engen-McCloskey distribution. Another approach, which is standard in population genetics, e.g., see \cite{Ethier1993}, is to embed the~$\PD(\theta)$ into a random measure. We refer to this representation as the Dirichlet process, following the Bayesian nonparametrics literature. Let $\delta_x(\cdot)$ denote both the Dirac measure and function.
\begin{definition}[Dirichlet process]
Let $\theta>0$ and let $\pi$ be a probability measure on a compact metric space~$E$.
Let $(P_1,P_2,\ldots)\sim \PD(\theta)$ be independent of 
$\xi_1,\xi_2,\ldots$, which are i.i.d.\ with distribution~$\pi$.
Define the probability measure on $M_1 := M_1(E)$, the space of probability measures on~$E$,  by 
\ben{\label{mucomp}
\DP(\theta, \pi) := \law\bbclr{\sum_{i=1}^\infty P_i \delta_{\xi_i}}.
}
\end{definition}
For a random probability measure~$W$, and $Z \sim \DP(\theta,\pi)$, our main result is a bound on $\abs{\IE H(W) - \IE H(Z)}$ for certain
test functions~$H$. To describe these test functions, we need some notation and definitions. Let $\rC(E^k)$  be the set of continuous functions from $E^k\to\IR$. For
 $\mu\in M_1$, we denote the expectation of the function $\phi\in\rC(E)$ with respect to $\mu$ by $\angle{\phi, \mu} := \int_{E}\phi(x) \mu(dx)$. 
Letting 
$\BC^{2,1}(\IR^k)$ be the set of functions with two bounded and continuous derivatives with the second derivative Lipschitz, we define the first set of test functions as
\be{
\cH_1=\bclc{F(\mu):=f(\angle{\phi_1,\mu},\ldots, \angle{\phi_k,\mu}):
k \in \IN, f\in \BC^{2,1}(\IR^k), \phi_i\in \rC(E), i=1,\ldots, k}.
}
While convergence of expectations of functions in $\cH_1$ implies weak convergence (with respect to the Prokhorov metric) of random measures, it 
does not imply convergence of sizes of masses (e.g., consider the convergence of the sequence of probability measures putting mass~$1/2$ on both~$2^{-n}$ and~$0$, as $n\to\infty$), or convergence of sampling formulas, which only depend on the sizes of the masses and not their labels;
 see \cite{Ethier1994} for further discussion. 
%By sampling formula we mean the distribution of the integer partition of $n$ generated by $n$ i.i.d.\ samples from a discrete distribution, where the number of blocks of the partition is the number of distinct values appearing in the sample, and the size of a block is the number of times that value appears in the sample. Sampling formulas play a prominent role in population genetics, where the values of the sample 
Thus, the test functions $\cH_1$ are natural for the Dirichlet process--our chosen encoding for the Poisson-Dirichlet distribution--but they do not capture 
approximation by~$\PD(\theta)$, or approximation of sampling formulas, 
which are key statistics in population genetics and Bayesian nonparametrics.
To address this issue, we also consider the following family of test functions:
\be{
\cH_2 = \bclc{F(\mu) := \angle{\phi, \mu^k}: k \in \IN, \phi \in \rC(E^k)}.
}
As discussed further in Section~\ref{sec:testfun}, convergence of expectations for functions in $\cH_2$ does not imply convergence of sampling formulas in general, but it is sufficient under the assumption that the labels are independent of the masses, and i.i.d.\ distributed according to a diffuse measure~$\pi$. 

We now state our general approximation theorem.
For a real-valued function $\psi$, define the supremum norm $\norm{\psi}_\infty=\sup_{x} \abs{\psi(x)}$, and for a signed measure~$\mu$ on~$E$, define the variation norm 
\be{
\norm{\mu}=\mathop{\sup_{\phi:E\to\IR}}_{\norm{\phi}_\infty=1}\int \phi d\mu.
}

\begin{theorem}\label{THM3}
Let~$\theta > 0$ and $(W,W')$ be a pair of random atomic probability measures in $M_1$ such that 
$\law(W)=\law(W')$, and 
$\pi\in M_1$ diffuse.
Let~$\lambda\in \IR$ and let~$R$ be the random signed measure satisfying 
\ben{
\IE [W'-W| W]=-\lambda \theta(W-\pi)+ R. \label{7}
}
Then for any $H \in \cH_1\cup \cH_2$ and $Z\sim\DP(\theta, \pi)$,
\ben{\label{8}
\abs{\IE H(W)-\IE H(Z)}\leq \frac{1}{2}\bbclc{D_1(H;\theta) A_1 + D_2(H; \theta) A_2 +D_3(H; \theta) A_3},
}
where $D_i(H;\theta)$ are the explicit constants in terms of~$\theta$ and properties of~$H$, given in Definition~\ref{DEF:D}, and
\ba{
A_1&:=\frac{1}{\lambda}\IE\norm{R}, \\
A_2&:= \IE \bbcls{\bnorm{W^{*2}-W^2-\tsfrac{1}{2\lambda}\IE[(W'-W)^2|W]}}, \\
A_3 &:= \frac{1}{6\lambda} \times 
\begin{cases} 
\IE\bbcls{ \|W' - W \| \|(W'-W)^2\|}, & H \in \cH_1, \\
\IE\bbcls{\| (W' - W)^3\|}, & H \in \cH_2,
\end{cases}
}
where $W^{*2}(dx,dy):=W(dx)\delta_x(dy)$ is the measure on $(E^2,\cB(E^2))$ defined by $W^{*2}(B)=W(\{x\in E: (x,x)\in B\})$.
\end{theorem}

To interpret the terms appearing in the bound, $D_i(H;\theta)$, $i=1,2,3$, are bounds on certain ``derivatives" of Stein solutions (see Section~\ref{sec:smdp} for more details) which are derived from the generator approach of~\cite{Barbour1988, Barbour1990, Gotze1991}; see also \cite{Reinert2005}. In particular, we characterize the Dirichlet process as the stationary distribution of a particular Fleming-Viot (FV) process; see \cite{Fleming1979}, \cite{Ethier1993}, and 
\cite[Section~5.2]{Feng2010}; the bounds are then obtained using couplings built from two probabilistic descriptions of the transition semigroup of the process, given in \cite{Ethier1993a} and \cite{Dawson1982}. In broad strokes, this is the same approach used in other developments of Stein's method for processes, but there are a number  technical and conceptual difficulties in our setting; e.g., coupling arguments require non-standard definitions for directional derivatives; which require careful treatment 
and new ideas to overcome. Deriving these bounds is a significant contribution of the paper.

The other terms in the bound  are interpreted in the context of the generator approach: Given~$W$, we think of $W'$ as a step in a Markov chain with dynamics that approximate those of the FV process at small times. The term $A_1$ measures the difference between the ``drift'' components of the Markov chain and the FV process, the term $A_2$ gives the error in the ``diffusion" coefficient of the two processes, and $A_3$ controls the contribution of higher order terms. 

\begin{remark}
The proof of Theorem~\ref{THM3} uses  Stein's method for distributional approximation, which was first developed in  \cite{Stein1972, Stein1986}, and has now found a large number of applications in probability and statistics; 
see \cite{Barbour1992}, \cite{Chen2011}, and \cite{Ross2011}
for various introductions, and \cite{Chatterjee2014} for a recent literature survey. 
In particular, the theorem follows a long line of Stein's method ``exchangeable pairs" 
approximation theorems for other distributions; e.g., normal \cite[Theorem~1.1]{Rinott1997}; multivariate normal \cite[Theorem~2.3]{Chatterjee2008} \cite[Theorem~2.1]{Reinert2009}; 
Poisson~\cite[Proposition~3]{Chatterjee2005}; 
translated Poisson~\cite[Theorem~3.1]{Rollin2007};
exponential \cite[Theorem~1.1]{Chatterjee2011}, \cite[Theorem~1.1]{Fulman2013}; beta \cite[Theorem~4.4]{Dobler2015}; Dirichlet \cite[Theorem~3]{Gan2017}; limits in Curie-Weiss models \cite[Theorem~1.1]{Chatterjee2011a}.
However note that, following \cite{Rollin2008}, our result does not require the coupling $(W,W')$ to be exchangeable, only that it has equal marginals. Also note that the references above are for finite dimensional distributions, while ours is an infinite dimensional, or process level, result. Stein's method for process approximation is much less developed, with the seminal ideas going back to \cite{Barbour1988, Barbour1990}. Since then, there has been good progress for Poisson process approximation; e.g., \cite{Chen2004}; but little else outside of some very recent activity, for certain diffusions and Gaussian processes; see 
\cite{Kasprzak2017a,Kasprzak2017c,Kasprzak2020} and \cite{Bourguin2019}.
\end{remark}
\begin{remark}
The assumption that~$\pi$ is diffuse is necessary for the masses of the random measure $\DP(\theta,\pi)$ to have the $\PD(\theta)$ distribution. If~$\pi$ has finite support, then the Poisson process representation of Definition~\ref{def:pd} and Poisson thinning implies that the resulting ordered masses have the ordered Dirichlet distribution, and an analog of Theorem~\ref{THM3} can be obtained from \cite[Theorem~3]{Gan2017}. If the support of~$\pi$ is countable, then approximation could perhaps be reduced to the Dirichlet case, by first combining the small masses of~$\pi$ (at some cost to the error). More complicated distributions for~$\pi$ would require at least a mixture of the ideas above, but since our main focus is $\PD(\theta)$ approximation, we will always assume that~$\pi$ is diffuse.
\end{remark}
 
We apply Theorem~\ref{THM3} in two settings. First, at the end of this section, we work out the bounds of the theorem for the random measure derived from i.i.d.\ $\pi$-distributed labeling of the (normalized) masses from the Ewens Sampling Formula (ESF). Since the ESF is the exact sampling distribution of $\PD(\theta)$, similar bounds can be obtained from direct couplings, but the example is computationally illustrative and serves as a proof of concept.
The second setting is a more significant application, bounding the error when approximating stationary distributions of discrete Wright-Fisher models with infinite alleles mutation structure by the Dirichlet process.

\subsection{The infinite alleles Wright-Fisher model}
The Wright-Fisher model, first defined in \cite{Wright1949}, is one of the most useful and well-studied models of a genealogy in population genetics. In the model, generations have a fixed population size $N$, and each generation dies and gives birth to a new generation at the same time (non-overlapping, discrete generations). Specifically, the genealogy evolves by each child in the next generation choosing a parent uniformly at random. On top of the genealogy, we give each individual a genetic ``type" or label, which we encode as an element from the compact metric space~$E$. 
The dynamics of the process are Markovian in the generations, with parameters $p:E\to[0,1]$ and $\{\kappa_x\}_{x\in E}\subseteq M_1$ (these can be combined, but it is clarifying in our results to separate them). Given the genealogy, if a child's parent is of type $x$, then the child is of type $x$ with probability $1-p(x)$,  and otherwise mutates to a type distributed as $\kappa_x$, with choices independent between children. 
We study the empirical  probability distribution of  types in a given generation when the process is stationary. These empirical measures determine the sampling distribution, which forms the basis for statistical inference under this model. Unfortunately these stationary distributions are intractable, even in the parent independent mutation (PIM) case where $p(x)$ and $\kappa_x$ do not depend on~$x$. The standard approach under PIM 
is to approximate the stationary distribution by the limiting distribution as $N\to\infty$: assuming $2Np\to\theta$ and $\kappa=\pi$, the limiting stationary distribution is exactly~$\DP(\theta,\pi)$ and then the sampling distribution is given by the ESF.
We use Theorem~\ref{THM3} to quantify this convergence in the following general approximation result.

\begin{theorem}\label{THM:WF}
For a fixed $N$, $p:E\to[0,1]$, and $\clc{\kappa_{x}}_{x\in E}\subseteq M_1$, let $W_N$ be distributed as a stationary empirical probability measure associated with the Wright-Fisher Markov chain described above, and let $K_N$ be the number of distinct types of $W_N$. For any $\theta > 0$, and diffuse measure $\pi\in M_1$, let $Z \sim \DP(\theta, \pi)$. Then for any $H \in \cH_1 \cup \cH_2$, 
\ba{
\abs{ \E H(W_N) - \E H(Z) } \leq \frac12 \left\{D_1(H;\theta) A_1 + D_2(H;\theta) A_2 + D_3(H;\theta) A_3\right\},
}
where $D_i(H;\theta)$, $i=1,2,3$, are given in Definition~\ref{DEF:D}, and
\ba{
A_1 &=4N \sup_{x\in  E} \babs{ p(x) - \tsfrac{\theta}{2N}} + \theta \sup_{x\in E} \norm{\kappa_x - \pi},\\
A_2 &= 4 \norm{p}_\infty \bclr{N \norm{p}_\infty + 3  }, \\
A_3 &= \frac{\IE\bcls{K_N^{3/2}}}{3N^{1/2}}\bbclr{\sqrt{2}+(2) 14^{1/3} \bcls{\clr{N\norm{p}_\infty}^3 + N \norm{p}_{\infty}}^{1/3}}^{3}.
}
%where $C_{N,p}:=14 \bcls{\clr{N\norm{p}_\infty}^3 + N \norm{p}_{\infty}}$.
\end{theorem} 

\begin{remark}
The bound is completely explicit except for $\IE\bcls{K_N^{3/2}}$, which can be handled in two ways. First, Theorem~\ref{thm:KN2mombd} below implies that the term is bounded of order 
\be{
\log(N)^{3/2}\bclr{ 1+N \norm{p}_{\infty}+\clr{N \norm{p}_{\infty}}^2}^{3/4},
}
where the constant is quite large, but explicit and improvable for given $N$ and $\norm{p}_{\infty}$ (by following the proof of Theorem~\ref{thm:KN2mombd}). 
Thus, in the regime of convergence to $\DP(\theta,\pi)$ where (necessarily under PIM) $\sup_{x\in E}\abs{2Np(x)- \theta}\to 0$ and $\sup_{x\in E} \norm{\kappa_x - \pi}\to 0$,  the bound is of order
\be{
\sup_{x\in E}\abs{2Np(x)- \theta}+\sup_{x\in E} \norm{\kappa_x - \pi} +\frac{\log(N)^{3/2}}{N^{1/2}}.
}
Second, for any fixed value of $N$ and $\norm{p}_\infty$, the recursive formulas for the distribution of the number of types for the PIM case with mutation rate $\norm{p}_\infty$ (our proof shows that the number of types in this case dominates the general case) are available from \cite{Lessard2007, Lessard2010}, and these can be used to numerically compute $\IE[K_N^{3/2}]$. The recursive formulas do not seem usable to obtain an explicit analytic bound like that in Theorem~\ref{thm:KN2mombd}.
\end{remark}

\begin{remark}
Theorem~\ref{THM:WF}~applies only if a stationary distribution exists, which needs to be verified separately. Conditions for the existence 
of a stationary distribution can be adapted from the general theory given in, e.g., \cite{Meyn1993}. For example, if there is a probability
measure $\nu\in M_1$ and $\eta >0$ such that $p(x)\kappa_x(A)\geq \eta \nu(A)$ for all Borel sets $A\subseteq E$ and $x\in E$, then the chain is Harris recurrent since 1) there is a uniformly bounded away from zero chance at every step to move to the set $B$ of empirical measures with support a subset of the support of $\nu$ (i.e., all children are mutants with type sampled from~$\nu$), and 2)
the chain regenerates since it is possible from any state to mix according to the PIM-$\nu$ dynamics to the PIM-$\nu$ stationary distribution. Since the time between visits to $B$ has finite mean, the chain has a stationary distribution; see \cite[Section~2]{Baxendale2011} for a clear exposition of the argument we are using.
Alternatively, if $p$ is continuous and $\kappa_x$ is continuous with respect to weak convergence as a function of $x$, then the chain is Feller on
a compact state space, and thus has a stationary distribution.
We omit the details for the sake of space, and leave it to practitioners to ensure a stationary distribution exists before applying the theorem.
\end{remark}

As previously mentioned, perhaps the most useful convergence in the population genetics setting is for sampling formulas. 
Assuming a random measure $W$ is purely atomic  (as it is in our application), then we are interested in the distribution of the partition induced by a simple random sample from~$W$. Specifically, denoting the sample by $(y_1,\ldots, y_n)$, we write $\cS_n(W)$ for the probability measure on set partitions of $\{1,\ldots,n\}$ induced by the relation $i\sim j \iff y_i=y_j$. As is well-known, the law of $\cS_n(Z)$ for $Z\sim\DP(\theta,\pi)$ can be read from the ESF; see Section~\ref{sec:ESF} below.
We have the following corollary for the PIM Wright-Fisher model, which follows easily from Theorem~\ref{THM:WF} and Lemma~\ref{lem:smoothsamp}, stated and proved below.

\begin{corollary}\label{cor:wfsamp}
For a fixed $N$, $p(\cdot)\equiv \theta/(2N)\in(0,1)$ and $\kappa_{x}:=\pi$ for $x\in E$, where $\pi\in M_1$ is diffuse, let $W_N$ be distributed as a stationary empirical probability measure associated with the Wright-Fisher Markov chain described above. Then for $Z\sim\DP(\theta,\pi)$, 
\bes{
\dtv\bclr{\cS_n(W_N), \cS_n(Z)} &\leq  \bbbclr{\frac{2n}{\theta} + \frac{2n(n-1)}{\theta+1}}A_2 \\
	&\qquad +\bbbclr{\frac{2 n}{\theta} + \frac{6 n(n-1)}{\theta+1} +
\frac{8 n(n-1)(n-2)}{3(\theta+2)}}A_3,
}
where $A_2$ and $A_3$ are defined in Theorem~\ref{THM:WF}, with $A_2=\bigo(N^{-1})$ and $A_3=\bigo(N^{-1/2}\log(N)^{3/2})$.
%and
%\ba{
%%\wh D_1(n;\theta) &:= 4n \theta^{-1} ,\\
%\wh D_2(n;\theta)&:= \frac{4n}{\theta} + \frac{4n(n-1)}{\theta+1},\\
%\wh D_3(n;\theta) &:= \frac{4 n}{\theta} + \frac{12 n(n-1)}{\theta+1} +
%\frac{16 n(n-1)(n-2)}{3(\theta+2)}.
%}
\end{corollary}

The corollary implies that if $n\ll N^{1/6}\log(N)^{-1/2}$, then the ESF probabilities are good approximations for those of an $n$-sample from  the PIM Wright-Fisher model with $N$ individuals. This kind of approximation result has important implications for modern genetic studies, since cheaper and faster sequencing has made sample sizes large relative to effective populations size; see, e.g., \cite{Bhaskar2014} and \cite{Fu2006}. The asymptotic range of convergence of the sample size $n\ll N^{1/6}\log(N)^{-1/2}$ is likely not optimal (even without the log factor). The powers of $n$ appearing in the bound of the corollary stem from the number of derivatives of  the ``test'' functions~$H$ appearing in Definition~\ref{DEF:D}, and this may be improvable, leading to larger range of convergence with our methods. The proof of Lemma~\ref{lem:numedges} suggests that the genealogy of an $n$-sample in the PIM Wright-Fisher model is close in some sense to a discretized version of the coalescent when $n\ll \sqrt{N}$, but translating this heuristic into a bound for the total variation distance of the corollary is not straightforward. 
It is an interesting open problem to determine a sharp asymptotic threshold of the growth of $n$ with $N$, for convergence to zero of this total variation distance.

Theorem~\ref{THM:WF} and Corollary~\ref{cor:wfsamp} require a bound on $\IE\bcls{K_N^{3/2}}$ (or any higher moment via Jensen's inequality). Though simpler than the full stationary distribution, even this quantity is difficult to handle, and there do not appear to be any rigorous bounds or even asymptotics in the literature. Heuristically, in the PIM case with mutation probability of order $1/N$, $K_N$ should be of order $\log(N)$, since, looking backward in time, the genealogy should quickly collapse to a small power of $N$ without accumulating too many types, and then we can understand the number of types via the ESF for a sample of the small power of $N$ size, which predicts $\log(n)$ types in a sample of size $n$. Making this precise is difficult, but we are able to show the following result, whose proof verifies the heuristic.

\begin{theorem}\label{thm:KN2mombd}
Let $N\geq 3$ and $K_N$ be the stationary number of types in an infinite alleles Wright Fisher model with maximum probability of mutation $\norm{p}_\infty$. Then 
\be{
\IE \bcls{K_N^2} \leq \log(N)^2 \bclr{4+ \bclr{12\times 10^3} N\norm{p}_\infty+\bclr{ 6\times 10^6} (N\norm{p}_\infty)^2}. 
}
\end{theorem}

\begin{remark}
Regarding related work, even in the case of parent independent mutation (PIM), which we do not assume, the result is new. However, there is some related work which we now discuss. If $p(x)=\theta/2N$, then \cite[(3.92)]{Ewens2004} shows that $\IE[K_N] = \bigo(\log(N))$, with no explicit constant. The approach there is to relate the expectation to the expected fixation time of a single allele via ergodic considerations; see \cite[(3.91)]{Ewens2004}; and then approximate this quantity by the limiting Wright-Fisher diffusion approximation. Without good control of the errors, it is not possible to extract explicit constants using this method, and moreover, it is unclear how to extend it to the second moment.

Still assuming $p(x)=\theta/2N$, \cite[p.~434, Case 2]{Karlin1967} argue that $\IE[K_N]=\bigo(\log(N))$
by first discretizing the type space into $r$ types, and then approximating the discrete Wright Fisher stationary distribution by an appropriate Dirichlet distribution. However, the first approximation is good for $r$ large and~$N$ fixed, while the second is good for fixed~$r$ and large~$N$, and they give no argument to control the errors in these approximations. Thus the argument is incomplete (and we do not see a simple way to close the gap), so their result is taken as a heuristic, which our lemma confirms.

%For a sample of size $n=\lito(\sqrt{N})$, the Wright-Fisher genealogy viewed backwards in time is well-approximated by Kingman's coalescent; \cite{Kingman1982,Kingman1982a}; and then the distribution of the number of types in the sample follows from Ewen's sampling formula, which gives that the number of types in the sample is of the order~$\log(n)$. The approximation is not valid for larger values of $n$, see, e.g., \cite{Bhaskar2014} and \cite{Fu2006}, but does give a lower bound for the order of the number of types in the entire population. In particular Theorem~\ref{thm:KN2mombd} validates the heuristic that the full genealogy collapses very quickly to moderate values; this is made more precise below.
\end{remark}

\begin{remark}
The constants in the upper bounds of  Theorem~\ref{thm:KN2mombd} can be greatly improved 
 by assuming larger minimal values of $N$, or even just being more careful in the proofs, where we strive for clarity over obtaining good constants. 
\end{remark}

\begin{remark}
Applying Theorem~\ref{THM3} to the Cannings genealogy (where the offspring distribution is only assumed to be exchangeable) with PIM would yield an analog of Theorem~\ref{THM:WF}; cf.\ \cite[Theorem~2]{Gan2017}. However, it is not at all clear how to obtain an analog of Theorem~\ref{thm:KN2mombd}, bounding the relevant moments of $K_N$.
\end{remark}

The organization of the paper is as follows. We conclude the introduction with a simple application of Theorem~\ref{THM3} for the Ewens sampling formula, followed by a discussion of the implicit metrics of the test functions we use. In Section~\ref{sec:smdp} we develop Stein's method for the Dirichlet process, proving Theorem~\ref{THM3}, and then in Section~\ref{sec:wf} we prove the Wright-Fisher approximation Theorem~\ref{THM:WF} and Theorem~\ref{thm:KN2mombd}, bounding the second moment for the number of types. 

\subsection{Example: The Ewens Sampling Formula}\label{sec:ESF}
For $\theta>0$, the ESF$(\theta)$ distribution is a probability distribution on integer partitions arising through the sampling distribution of~$\PD(\theta)$, and so occurs in all the contexts and references mentioned in the first paragraph of this article.
By sampling distribution we mean that if~$n$  samples $(y_1,\ldots,y_n)$ are drawn i.i.d.\ from $\DP(\theta,\pi)$ for $\pi$ diffuse, then the 
distribution of the partition of $\{1,\ldots,n\}$ generated by the equivalence relation $i\sim j \iff y_i=y_j$ can be read from the ESF$(\theta)$, which can be described in terms of the number of labels in the sample appearing~$i$ times, $i=1,\ldots,n$; see, e.g., \cite[(4.5)]{Arratia2003} or \cite[(2.20)]{Pitman2006}. For our purposes, we retain the labels of the sample and view the ESF as giving the distribution of masses of an empirical measure 
in $M_1$. In particular, sampling sequentially, and letting $x_i\in E$, $i=1,2,\ldots,$ be the label of the $i$th unique label to appear in the sampling sequence, 
we define
\ben{\label{eq:dpsampcrp}
W_n:=\frac{1}{n}\sum_{i=1}^{n}  \delta_{x_{\sigma_i}},
}
where
$x_{\sigma_i}=y_i$ is the label of the $i$th sample (note that it is an eventuality that $\sigma_i=\sigma_j$ for some $i\not=j$).
As is well-known, e.g., from \cite[Section~3.2]{Pitman2006}, the dynamics of the un-normalized masses of this process follow a one parameter Chinese Restaurant Process, and in particular, given $W_{i-1}$, the  distribution of~$x_{\sigma_{i}}$ is 
\ben{\label{eq:crpsamp}
\frac{i-1}{i+\theta-1} W_{i-1}+ \frac{\theta}{i+\theta-1} \pi.
}
We now state the following result.

%distribut, either we draw an existing label, i.e.,$N_{m+1,i}= N_{m,i}+1$, for some $i$, and all else stays the same, or a new label is chosen and $K_{m+1}=K_m+1$ and $N_{K_{m+1}}=1$ and all else stays the same. The probability of the first event is defined to be
%\be{
%\IP\bbclr{N_{m+1,i}= N_{m,i}+1,  \{N_{m+1,j}= N_{m,j}, j\not=i \}| (N_{m,i})_{i=1}^{K_m} }=\frac{N_{m,i}}{m+\theta},
%}
%and the probability of the second event is 
%\be{
%\IP\bbclr{K_{m+1}= K_{m}+1, N_{K_{m}+1}=1,  \{N_{m+1,j}= N_{m,j},j=1,\ldots,K_m\} | (N_{m,i})_{i=1}^{K_m} } = \frac{\theta}{m+\theta}.
%}

%obtaining a sample with $a_i$ labels appearing~$i$ times, $i=1,\ldots,n$, with $\sum_{i=1}^{n}i a_i=n$ is given by the ESF formula from, e.g., \cite[(4.5)]{Arratia2003} or \cite[(2.20]{Pitman2006}:
%\be{
%\frac{n! }{\theta(\theta+1)\cdots (\theta + n-1)} \prod_{i=1}^n \bbbclr{\frac{\theta}{i}}^{a_i} \frac{1}{a_i !}.
%}
%Alternativ

\begin{proposition}\label{CRP}
Let $\theta>0$ and $W_n$ be distributed as the empirical measure of an $n$-sample from~$\DP(\theta,\pi)$ for~$\pi$ diffuse, as defined at~\eq{eq:dpsampcrp}, and $Z \sim \DP(\theta, \pi)$. Then for any $H \in \cH_1 \cup \cH_2$,
\ba{
\babs{ \E H(W_n) - \E H(Z) } \leq D_2(H;\theta) \frac{\theta}{n} +D_3(H;\theta)\frac{2(n+\theta-1)}{3n^2},
}
where $D_i(H;\theta)$, $i=2,3$, are given in Definition~\ref{DEF:D}.
\end{proposition}

\begin{proof}
We apply Theorem~\ref{THM3} and construct $W_n'$ from the sampling process by resampling the last step. 
Recalling  $\sigma_i$ is the label of the $i$th sample, denote $Y_i:= \delta_{x_{\sigma_i}}$, so that $W_n = \frac{1}{n} \sum_{i=1}^n Y_i$.  
Let $Y_n'$ be conditionally independent of $Y_n$ given $(Y_1,\ldots,Y_{n-1})$, and have the same conditional distribution. 
Setting $W_n' = W_n - \frac{Y_n}{n} + \frac{Y_n'}{n}$, it is straightforward to see that $W_n'\eqlaw W_n$. To check the linearity condition~\eq{7}, we use~\eq{eq:crpsamp} to easily compute
\besn{\label{eq:prcrp}
\E[ Y_n'| Y_1, \ldots, Y_n] &= \frac{1}{n+\theta -1} \left[ \theta \pi + (n-1)W_{n-1}\right]\\
	&=\frac{1}{n+\theta -1} \left[ \theta \pi+ nW_n - Y_n\right].
}
Now, we claim that the vector $(Y_1,\ldots,Y_n)$ is exchangeable. To see this, consider the
%for any permutation~$\tau$ of $\{1,\ldots,n\}$, 
%\be{
%(Y_{\tau(1)},\ldots, Y_{\tau(n)})=(\delta_{x_{\sigma_{\tau(i)}}}
%} since 
the random partition of $\{1,\ldots,n\}$ generated by the equivalence relation $i\sim j\iff \sigma_i=\sigma_j$; this partition is created by grouping the indices of the $Y_i$ that put mass on the same value of $x_i$. From the description of the dynamics above, it is clear that this partition is distributed as a one parameter Chinese Restaurant Process run to step~$n$, and thus it is exchangeable; see \cite[Section~3.2]{Pitman2006}. 
The exchangeability of $(Y_1,\ldots,Y_n)$ now follows from\cite[Proof of (11.9)]{Aldous1985}, noting that the $x_i$ are i.i.d.
Because of exchangeability, $\IE[Y_i | \sum_{j=1}^n Y_j]$ is the same for all~$i=1,\ldots,n$, and so 
must equal $W_n$. In particular, $\IE[Y_n|W_n]=W_n$, and using this in~\eq{eq:prcrp} implies
\ba{
\E[Y_n' | W_n] = \frac{\theta}{n+\theta -1} \pi + \frac{n-1}{n+\theta -1}W_n,
}
and so
\ba{
\E[W_n' - W_n | W_n] &=\frac{1}{n} \E[ Y_n' - Y_n | W_n]\\
	&= - \frac{1}{n(n+\theta-1)} \theta\clr{W_n- \pi}.
}
We therefore apply Theorem~\ref{THM3} with $R = 0$ and $\lambda = \frac{1}{n(n+\theta -1)}$.

Since $R=0$, $A_1$ from Theorem~\ref{THM3} is zero. To bound~$A_2$ from Theorem~\ref{THM3}, we first compute
\bes{
\IE\bcls{(Y_n')^2| (Y_1,\ldots,Y_n)}
	&= \frac{1}{n+\theta -1} \left[ \theta \pi^{*2} + (n-1)W_{n-1}^{*2}\right]\\
	&=\frac{1}{n+\theta -1} \left[ \theta \pi^{*2}+ nW_n^{*2} - Y_n^{*2}\right],
}
and since $\tsfrac{1}{n}\sum_{i=1}^n Y_n^{*2}=\tsfrac{1}{n}\sum_{i=1}^n Y_n^2=W_n^{*2}$, exchangeability implies
\bes{
\IE\bcls{Y_n^{*2}| W_n}=\IE\bcls{Y_n^2| W_n}= W_n^{*2}.
}
Using these last two displays with~\eq{eq:prcrp}, we find
\bes{
\IE\bcls{(W_n'-W_n)^2| W_n} 
	&=\frac{1}{n^2}\IE\bcls{(Y_n'-Y_n)^2| W_n} \\
	&=\frac{1}{n^2}\bclr{\IE[(Y_n')^2| W_n]- \IE[ Y_n'Y_n| W_n]- \IE[ Y_nY_n'| W_n]+\IE[ Y_n^2|W_n]} \\
&
\begin{split}
	&=\frac{1}{n^2(n+\theta-1)}\bbclr{  \theta \pi^{*2}+ (n-1)W_n^{*2} -\bclr{ \theta \pi W_n+ nW_n^2 - W_n^{*2}}\\
	&\hspace{4cm} -\bclr{ \theta  W_n\pi+ nW_n^2 - W_n^{*2}}+(n+\theta-1)W_n^{*2}} 
\end{split}	 \\
	&=\frac{1}{n^2(n+\theta-1)}\bbclr{ 2n W_n^{*2} - 2nW_n^2+ \theta \pi^{*2}-\theta \pi W_n-\theta W_n\pi+\theta  W_n^{*2}} \\
		&=2\lambda\bclr{W_n^{*2} - W_n^2+ \tsfrac{\theta}{2n} \clr{ \pi^{*2}- \pi W_n-W_n\pi+  W_n^{*2}}}.
}
We now easily find $A_2\leq \frac{2\theta}{n}$.
To bound~$A_3$ of Theorem~\ref{THM3}, it is easy to see that
\be{
\IE \bbcls{\norm{W_n'-W_n}^3}= \frac{1}{n^3} \IE \bbcls{\norm{Y_n'-Y_n}^3}\leq  \frac{8}{n^3},
}
and 
\be{
\IE \bbcls{\norm{W_n'-W_n}^2 \norm{W_n'-W_n}}= \frac{1}{n^3} \IE \bbcls{\norm{Y_n'-Y_n}^2\norm{Y_n'-Y_n}}\leq  \frac{8}{n^3},
}
and so $A_3\leq \frac{4(n+\theta-1)}{3n^2}$.
\end{proof}

\subsection{Test functions}\label{sec:testfun}

Here we discuss the implications of 
\ben{\label{eq:convt}
\lim_{N\to\infty}\babs{\IE H(W_N) - \IE H(Z)} =0, \text{ for all } H\in \cH_i,
}
for either $i=1$ or $i=2$. If for all $\phi\in \rC(E)$, 
\be{
\lim_{N\to\infty}\babs{\IE\angle{\phi,W_N} - \IE\angle{\phi,Z}}=0,
} 
then $W_N\stackrel{w}{\longrightarrow} Z$, e.g., see \cite[Chapter~4]{Kallenberg2017}; here weak convergence is with respect to the Prokhorov metric or weak topology. Thus if~\eq{eq:convt} holds for either $\cH_1$ or $\cH_2$, then weak convergence is implied. 

More importantly, we have the following lemma that shows that good approximation of expectations of functions from $\cH_2$ implies good approximation of sampling formulas.
\begin{lemma}\label{lem:smoothsamp}
Let $W$ be a random measure of the form
\be{
W=\sum_{i} Q_i \delta_{\xi_i},
}
where $(Q_1,Q_2,\ldots)\in \nabla_\infty$ is random and independent of $\xi_1,\xi_2, \ldots,$ which are i.i.d.\ $\pi$-distributed, for some diffuse $\pi\in M_1$. Let $\theta>0$ and $Z\sim\DP(\theta,\pi)$. If, 
%as $\eps\to0$,
%\ben{\label{eq:diaglito}
%\IP\bclr{\rho(\xi_1, \xi_2)< \eps} \to 0,
%}
%and 
for all $H\in \cH_2$,
\ben{\label{eq:bdassump}
\babs{\IE H(W) - \IE H(Z)} \leq D_1(H;\theta)B_1+D_2(H;\theta)B_2+D_3(H;\theta)B_3,
}
for some non-negative $B_1,B_2, B_3$, and 
where $D_i(H;\theta)$, $i=1,2,3$, are given in Definition~\ref{DEF:D},
then
\ben{\label{eq:bdconclu}
\dtv\bclr{\cS_n(W), \cS_n(Z)}\leq 
	\wh D_1(n;\theta)B_1+\wh D_2(n;\theta)B_2+ \wh D_3(n;\theta)B_3,
}
where
\ba{
\wh D_1(n;\theta) &:= \frac{4n}{\theta} ,\\
\wh D_2(n;\theta)&:= \frac{4n}{\theta} + \frac{4n(n-1)}{\theta+1},\\
\wh D_3(n;\theta) &:= \frac{4 n}{\theta} + \frac{12 n(n-1)}{\theta+1} +
\frac{16 n(n-1)(n-2)}{3(\theta+2)}. 
}
\end{lemma}
\begin{proof}
Fix $n$, let~$C$ be a subset of set partitions of $\{1,\ldots,n\}$, and let $S(W)\sim\cS_n(W)$ and $S(Z)\sim\cS_n(Z)$. We want to show that 
$\abs{\IP(S(W)\in C) - \IP(S(Z)\in C)}$ is upper bounded by the right hand side of~\eq{eq:bdconclu}.
If the partition $\Pi=\{E_1, E_2,\ldots,E_k\}$, where the $E_i\subseteq\{1,\ldots,n\}$ are disjoint with $\#(E_i)=n_i$ and $\sum_{i=1}^k n_i=n$, then
\be{
\IP(S(W)=\Pi)=\IE\bbbbcls{ \sum_{\{i_1,\ldots,i_k\}} \prod_{j=1}^kQ_{i_j}^{n_j} },
}
where the sum is over all $k$-tuples of distinct  indices. Then also, in obvious notation,
\be{
\IP(S(W)\in C)= \sum_{\Pi\in C} \IE\bbbbcls{ \sum_{\{i_1,\ldots,i_{k(\Pi)}\}} \prod_{j=1}^{k(\Pi)}Q_{i_j}^{n_j(\Pi)} }.
}

To realize this as an expectation of an integral against $W$, define
the diagonal indicator $d: E^2 \to \{0,1\}$ by
\be{
d(x,y)=\Ind\bcls{x=y},
}
and $h_{\Pi}:E^n\to\{0,1\}$, by 
\bes{
h_{\Pi}(y_1,\ldots,y_n)&=\prod_{i=1}^{k(\Pi)}\prod_{j_1,j_2\in E_i(\Pi)} d(y_{j_1}, y_{j_2})\\
	&\hspace{1cm}	\times \prod_{1\leq i_1 < i_2 \leq k(\Pi)}\prod_{j_1\in E_{i_1}(\Pi)}\prod_{j_2\in E_{i_2}(\Pi)}\bclr{1-d\clr{y_{j_1},y_{j_2}}}.
}
Then it is straightforward to see that
\be{
\IP(S(W)=\Pi)=\bangle{ h_{\Pi}, W^n }=:H_{\Pi}(W).%(dy_1,\ldots, dy_n).
}
Unfortunately, $h_{\Pi}\not\in\rC(E^n)$ and so we can't say $H_{\Pi}\in \cH_2$, but we can approximate~$h_{\Pi}$ by an $\eps$-smoothed version for small $\eps>0$. Denoting the metric on $E$ by~$\rho$, define the continuous smoothed diagonal indicator
\be{
d\s{\eps}(x,y):=\max\bbbclc{\bbbclr{1-\frac{\rho(x,y)}{\eps}},0}\in [0,1].
}
%and define the continuous smoothed diagonal indicator
%\be{
%d\s{\eps}(y_1,\ldots, y_m)=\prod_{1\leq i < j \leq m}g\s{\eps}(y_i,y_j).
%}
Now define $\hat h_{\Pi}= h_{\Pi}^{(\eps)}:E^n \to[0,1]$ by 
\ban{
\hat h_{\Pi}(y_1,\ldots,y_n)
	&=\prod_{i=1}^{k(\Pi)}\prod_{j_1,j_2\in E_i(\Pi)} d\s{\eps}(y_{j_1}, y_{j_2}) \notag \\
	&\hspace{1cm}	\times \prod_{1\leq i_1 < i_2 \leq k(\Pi)}\prod_{j_1\in E_{i_1}(\Pi)}\prod_{j_2\in E_{i_2}(\Pi)}\bclr{1-d\s{\eps}\clr{y_{j_1},y_{j_2}}}.\label{eq:cprod} 
	}
Now, since $\hat h_{\Pi} \in \rC(E^n)$, we have 
\be{
\wh H_{C}(W):=\bbangle{ \sum_{\Pi\in C}\hat h_{\Pi}, W^n }\in \cH_2.
}
Moreover, for any fixed $(y_1,\ldots, y_n)$, there is only one $\Pi$ such that $\hat h_{\Pi}(y_1,\ldots, y_n) \not=0$, and so 
 $\bnorm{\sum_{\Pi\in C} \hat h_{\Pi} }_{\infty}\leq 1$,  and we easily find
$D_i\bclr{\wh H_{C}, \theta} \leq \wh D_i(n, \theta)$ for $i=1,2,3$. Therefore,~\eq{eq:bdassump} yields that
\be{
\babs{\IE \wh H_{C}(W) - \IE \wh H_{C}(Z)} \leq \wh D_1(n;\theta)B_1+\wh D_2(n;\theta)B_2+ \wh D_3(n;\theta)B_3.
}
We complete the proof by showing that, as $\eps\to0$, 
\be{
\IE \wh H_C(W) = \IP(S(W)\in C) +  \lito(1),
}
and that the same holds with $W$ replaced by $Z$ (the proof is the same for~$Z$).
Since $n$ is fixed relative to $\eps$, it is enough to show that for each~$\Pi$, 
\be{
\IE \bangle{ \hat h_\Pi, W^n }= \IP(S(W)=\Pi) + \lito(1).
}
Fixing~$\Pi$ and dropping it from the notation, first 
 observe that if $\cI\subseteq \IN^n$ is the subset of indices such that 
 $(i_1,\ldots, i_n)\in \cI$ if and only if
 \be{
 \{i_m\}_{m\in E_j} \cap \{i_m\}_{m\in E_{\ell}}=\varnothing, \,\,\,\, 1\leq j <\ell \leq k,
 }
 then 
 \ba{
\bangle{ \hat h, W^n } 
	&= \sum_{i_1,\ldots, i_n}  \hat h(\xi_{i_1},\ldots, \xi_{i_n}) \prod_{j=1}^n Q_{i_j} \\
	&=  \sum_{(i_1,\ldots, i_n)\in\cI}\hat h(\xi_{i_1},\ldots, \xi_{i_n}) \prod_{j=1}^n Q_{i_j};
} 
this is because the other terms in the sum are zero, coming from a zero factor in~\eq{eq:cprod}. The second key observation is that if $(i_1, \ldots, i_n) \in \cI$ is such that for some $j\in\{1,\ldots,k\}$
and $\ell_1\not=\ell_2$ with $\ell_1,\ell_2\in E_j$, there is $i_{\ell_1}\not=i_{\ell_2}$, 
then 
\bes{
\IE\bbcls{\hat h(\xi_{i_1},\ldots, \xi_{i_n}) \prod_{j=1}^n Q_{i_j}} 
	&\leq \IE\bbcls{d\s\eps\bclr{\xi_{\ell_1},\xi_{\ell_2} } \prod_{j=1}^n Q_{i_j}}  \\
	&\leq \IP\bclr{\rho(\xi_1, \xi_2)< \eps} \IE\bbcls{\prod_{j=1}^n Q_{i_j}}.
}
Thus, using Fubini's theorem (the sum is bounded by one), we have
\bes{
 \IP(S(W)=\Pi) = \IE\bbbbcls{ \sum_{\{i_1,\ldots,i_k\}} \prod_{j=1}^kQ_{i_j}^{n_j} } 
 	&\leq \IE \bbcls{\bangle{ \hat h, W^n }}   \\
	&= \IE\bbbcls{\sum_{\cI}\hat h(\xi_{i_1},\ldots, \xi_{i_n}) \prod_{j=1}^n Q_{i_j}} \\
	&\leq   \IP\bclr{\rho(\xi_1, \xi_2)< \eps}
		+ \IE\bbbbcls{ \sum_{\{i_1,\ldots,i_k\}} \prod_{j=1}^kQ_{i_j}^{n_j} } \\
		&=\IP\bclr{\rho(\xi_1, \xi_2)< \eps} +  \IP(S(W)=\Pi) .
}
That $\IP\bclr{\rho(\xi_1, \xi_2)< \eps}=\lito(1)$ follows since $\pi$ is diffuse:
\be{
\IP\bclr{\rho(\xi_1, \xi_2)< \eps}=\int \IP\bclr{\rho(\xi_1, x_2)< \eps} \pi(dx_2) \to 0,
}
where we have used dominated convergence and that $\IP\bclr{\rho(\xi_1, x_2)< \eps}\to \IP(\xi_1=x_2)=0$.
\end{proof}

\section{Stein's method for the Dirichlet process}\label{sec:smdp}
The first step of Stein's method is to define a characterizing operator for~$\DP(\theta,\pi)$, and here we use the generator of a FV Markov process with unique stationary distribution~$\DP(\theta,\pi)$; see \cite{Fleming1979} and \cite{Ethier1993}. We consider the generator acting on the two domains
\ba{
\cD_1:=\bclc{F(\mu):=f(\angle{\phi_1,\mu},\ldots, \angle{\phi_k,\mu}):
k \in \IN, f\in \rC^{2}(\IR^k), \phi_i\in \rC(E), i=1,\ldots,k},
}
where $\rC^{2}(\IR^k)$ is the set of functions from $\IR^k\to \IR$ which have two continuous derivatives, 
and $\cD_2:=\cH_2$; note $\cH_1 \subseteq \cD_1$. The generator is a differential operator for the ``derivatives" defined, for $F\in\cD_1\cup \cD_2$, by 
\ban{\label{derivatives}
\frac{\partial F(\mu)}{\partial \nu}
	&:=\lim_{\eps \to 0^+} \frac{F((1-\eps)\mu+ \eps \nu) - F(\mu)}{\eps},
%\frac{\partial^2 F(\mu)}{\partial \nu_1 \partial \nu_2}
%	&:=\lim_{\eps_1, \eps_2 \to 0^+} \frac{F((1-\eps_1-\eps_2)\mu+ \eps_1 \nu_1+\eps_2\nu_2)  - F((1-\eps_1)\mu + \eps_1\nu_1) - F((1-\eps_2)\mu + \eps_2\nu_2) + F(\mu)}{\eps_1\eps_2},
}
with higher order derivatives defined analogously; noting that order of differentiation will typically matter. We shorten
formulas by writing $\partial_x$ for $\frac{\partial}{\partial \d_x}$, $\partial_{xy}$ for
$\frac{\partial}{\partial \d_y}\frac{\partial}{\partial\d_x}$ (noting the reversal of order), and so on, when there is no danger of confusion. 
Lemma~\ref{lem:derivsform} below collects formulas for derivatives of functions in $\cD_1\cup\cD_2$; note that most questions about
integrability, differentiation, and formulas stated below are easily resolved by appealing directly to these expressions.

Now, given $\pi \in M_1$, define the parent independent mutation operator~$A:\rC(E)\to \rC(E)$ 
by
\begin{align}
A \phi(x) = \frac\theta2 \int_E \bclr{\phi(y) - \phi(x)}\pi(dy), \label{A}
\end{align}
and then define the generator~$\cA$ of our FV Markov process by 
\ben{\label{eq:dpgen}
\cA F (\mu)=\int_{E} A \partial_{x} F(\mu) \mu(dx) 
+ 
	\frac{1}{2} \int_{E^2} \bclr{\mu(dx)\delta_x (dy) - \mu(dx) \mu(dy)} \partial_{xy} F(\mu),
}
where we can take the domain of $\cA$ to be either $\cD_1$ or $\cD_2$, and to clarify notation,
\be{
A \partial_{x} F(\mu)= \frac\theta2 \int_E \bclr{\partial_{y} F(\mu) - \partial_{x} F(\mu)}\pi(dy),
}
and so, in particular,
\ben{\label{eq:gddrift}
\int_{E} A \partial_{x} F(\mu) \mu(dx) =\frac\theta2 \int_E \partial_{x} F(\mu) \bclr{\pi-\mu}(dx).
}

\begin{remark}
The more standard definition of derivative used in the generator~\eq{eq:dpgen} is 
\ban{\label{eq:derivbad}
\lim_{\eps \to 0^+} \frac{F(\mu+ \eps \nu) - F(\mu)}{\eps}.
}
The action of~$\cA$ on $F\in\cD_1\cup\cD_2$ is the same regardless of choice of derivative definition, but using~\eq{eq:derivbad} requires taking limits from outside of~$M_1$ (though there is no problem in extending the domain of  functions in $\cD_1\cup\cD_2$), which is not amenable to our probabilistic coupling arguments.
\end{remark}

It was shown by \cite{Ethier1990} (see also \cite[Theorems~5.3 and~5.4]{Feng2010}) that
the Markov process on~$M_1$ with generator~\eq{eq:dpgen} is reversible with respect to its unique stationary distribution~$\DP(\theta, \pi)$. Therefore we have the following lemma that gives the characterising operator we use to develop Stein's method.

\begin{lemma}
Fix $i\in 1,2$. Let $\theta>0$ and $W$ be a random probability measure on~$E$.  Then $W \sim \DP(\theta, \pi)$ if and only if for all functions $F \in \cD_i$,
\be{
\E \cA F(W) = 0.
}
\end{lemma}

\begin{remark}
An alternative approach to Stein's method for Poisson-Dirichlet approximation is to work on the infinite dimensional 
simplex directly, characterising the distribution as the stationary distribution of the diffusion given in  
\cite{Ethier1981}. The generator for this process is defined on the core given by the sub-algebra generated by 
$\{1,\sum_{i=1}^\infty x_i^2, \sum_{i=1}^\infty x_i^3, \ldots\}$ as
\ba{
Bf(\tx) = \frac 12 \sum_{i,j=1}^\infty x_i(\delta_{ij} - x_j) \frac{\partial^2 f(\tx)}{\partial x_i \partial x_j} - \frac 12 \theta \sum_{i=1}^\infty x_i \frac{\partial f(\tx)}{\partial x_i}.
}
The difficulty with using this formula is that it may not apply to functions~$f$ outside of the core; see~\cite[Remark~5.4]{Petrov2009}.
\end{remark}

Before going further, for convenience we write
expressions for the derivatives of $F\in\cD_1\cup\cD_2$; cf., \cite[(3.14)]{Ethier1993}.
For a function $f\in \rC^2(\IR^k)$ let $\partial_i f, \partial_{ij}f$ denote 
the first partial derivative in coordinate~$i$ and the 
second in coordinates~$i,j$ (there is no danger of confusion with $\partial_x$ for $x\in E$).

\begin{lemma}\label{lem:derivsform}
If $F\in\cD_1$ is of the form $F(\mu):=f(\angle{\psi_1,\mu},\ldots, \angle{\psi_k,\mu})$ with $f\in\rC^2(\IR^k)$ and $\phi_i\in\rC(E)$, then
\ban{
%\frac{\partial F(\mu)}{\partial \delta_x}
\partial_x F(\mu)&=\sum_{i=1}^k \partial_i f\bclr{\angle{\psi_1,\mu},\ldots, \angle{\psi_k,\mu}} \bclr{\psi_i(x)-\angle{ \psi_i,\mu}}, \label{eq:d1deriv1} \\
%\frac{\partial^2 F(\mu)}{\partial \delta_x \partial \delta_y}
\begin{split}\label{eq:d1deriv2}
\partial_{xy} F(\mu)&=\sum_{i,j=1}^k 
\partial_{ij} f\bclr{\angle{\psi_1,\mu},\ldots, \angle{\psi_k,\mu}} \bclr{\psi_i(x)-\angle{ \psi_i,\mu}}\bclr{\psi_j(y)-\angle{ \psi_j,\mu}} \\
	&\hspace{1cm}-\sum_{i=1}^k \partial_i f\bclr{\angle{\psi_1,\mu},\ldots, \angle{\psi_k,\mu}} \bclr{\psi_i(y)-\angle{ \psi_i,\mu}}.
\end{split}
}
If $F\in\cD_2$ is of the form $F(\mu)=\angle{\psi, \mu^k}$ with $\psi\in\rB(\IR^k)$,
then
\ban{
\partial_x F(\mu)
	&=\sum_{i=1}^k \bclr{\angle{\psi_x\s{i},\mu^{k-1}}-\angle{\psi,\mu^k}},\label{eq:d2deriv1}  \\
\partial_{xy} F(\mu)
	&=\sum_{i\not=j}^k \bcls{\angle{\psi_{xy}\s{i,j}, \mu^{k-2}}-\angle{\psi_x\s{i},\mu^{k-1}}}-k\sum_{i=1}^k \bclr{\angle{\psi_y\s{i},\mu^{k-1}}-\angle{\psi,\mu^k}}, \label{eq:d2deriv2}\\
\begin{split}	\label{eq:d2deriv3}
\partial_{xyz} F(\mu)
	&=\mathop{\sum_{i,j,\ell=1}}_{\mathrm{distinct}}^k\bcls{\angle{\psi_{xyz}\s{i,j,\ell}, \mu^{k-3}}-\angle{\psi_{xy}\s{i,j}, \mu^{k-2}}} 
	 -k\sum_{i\not=\ell}^k \bcls{\angle{\psi_{yz}\s{i,\ell}, \mu^{k-2}}-\angle{\psi_y\s{i},\mu^{k-1}}} \\
	&\quad-(k-1)\sum_{i\not=\ell}^k \bcls{\angle{\psi_{xz}\s{i,\ell}, \mu^{k-2}}-\angle{\psi_x\s{i},\mu^{k-1}}}+k^2\sum_{i=1}^k \bclr{\angle{\psi_z\s{i},\mu^{k-1}}-\angle{\psi,\mu^k}},
	\end{split}
}
where $\psi_x\s{i}\in \rC(E^{k-1})$, $\psi_{xy}\s{i,j} \in \rC(E^{k-2})$, $\psi_{xyz}\s{i,j,\ell} \in \rC(E^{k-3})$ are defined by, for $i<j<\ell$, 
\ba{
\psi_x\s{i}(z_1,\ldots,z_{k-1})&=\psi(z_1,\ldots,z_{i-1},x,z_{i+1},\ldots, z_{k-1}), \\
\psi_{xy}\s{i,j}(z_1,\ldots,z_{k-2})&=\psi(z_1,\ldots,z_{i-1},x,z_{i+1},\ldots,z_{j-1},y,z_{j+1},\ldots, z_{k-2}),  \\
\psi_{xyz}\s{i,j,\ell}(z_1,\ldots,z_{k-3})&=\psi(z_1,\ldots,z_{i-1},x,z_{i+1},\ldots,z_{j-1},y,z_{j+1},\ldots,z_{\ell-1},z,z_{\ell+1},\ldots, z_{k-3}), 
}
with analogous definitions for other orderings of $i,j,\ell$.
\end{lemma}

\subsection{Bounds on the solution of the Stein equation}

To further develop Stein's method, for $H\in \cH_1\cup \cH_2$ and $Z\sim\PD(\theta, \pi)$, we need to solve the Stein equation. That is, we 
solve for $F_H$ satisfying 
\ben{\label{eq:pdstneq1}
\cA F_H(\mu)=H(\mu)-\IE H(Z) =: \wt H(\mu),
} 
and derive properties of the solution $F:=F_H$.
Following the generator approach of \cite{Barbour1988,Barbour1990}, \cite{Gotze1991},
for $\mu\in M_1$, let $(Z_\mu(t))_{t\geq0}$ be distributed as the FV process having generator~\eq{eq:dpgen} with $Z_\mu(0)=\mu$.
For $H:M_1\to\IR$ bounded, we define $F_H: M_1\to \IR$ by 
\ben{\label{eq:pdstnsol}
F_H(\mu)=-\int_0^\infty \IE \bcls{ \wt H(Z_\mu(t))} dt.
}
In the following theorem, we show that~\eq{eq:pdstnsol} is well defined, is the solution to~\eq{eq:pdstneq1}, and calculate bounds on the derivatives of $F_H$. Before stating the theorem, we define the constants that appear in our bounds, and state and prove a technical lemma. For $h\in \BC^{2,1}(\IR^k)$, denote
\bes{
\abs{h}_1&:= \sup_{1\leq i \leq k} \norm{\partial_{i}h}_{\infty}, \\
\abs{h}_2&:= \sup_{1\leq i,j \leq k} \norm{\partial_{i j} h}_{\infty}, \\
\abs{h}_{2,1}&:=\sup_{1\leq i, j \leq k} \sup_{r\not=s} \frac{\abs{h_{ij}(r)-h_{ij}(s)}}{\norm{r-s}_1}.
}

\begin{definition}\label{DEF:D}
For $H\in \cH_1$ with $H(\mu)=h(\angle{\phi_1,\mu},\ldots,\angle{\phi_k,\mu} )$, where $h\in \BC^{2,1}(\IR^k)$ and $\phi_i\in\rC(E)$, $i=1,\ldots, k$, denote
\ba{
\Lip_m(H)&:=\abs{h}_m \bbclr{\sum_{i=1}^{k} \norm{\phi_i}_\infty }^m, \,\, m=1,2, \\
\Lip_3(H)&:=\abs{h}_{2,1}\bbclr{\sum_{i=1}^{k} \norm{\phi_i}_\infty }^3.
} 
\smallskip
\noindent For $H\in \cH_2$ with $H(\mu)=\angle{\phi, \mu^k}$, where $\phi\in\rC(E^k)$, denote
\ba{
\Lip_m(H)&:=k (k-1)\cdots (k-m+1) \norm{\phi}_\infty, \,\, m=1,2,3.
} 

\medskip 
\noindent For $H \in \cH_1 \cup \cH_2$ and $\theta>0$, denote
\ba{
D_1(H; \theta) &:= \frac{4\Lip_1(H)}{ \theta} ,\\
D_2(H; \theta) &:= \frac{4\Lip_1(H)} \theta + \frac{4\Lip_2(H)}{\theta+1},\\
D_3(H; \theta) &:= \begin{cases}
 \frac{4 \Lip_1(H)}{\theta}+  \frac{16  \Lip_2(H)}{\theta+1}+\frac{16 \Lip_3(H)}{3(\theta+2)}, &\ \ H \in \cH_1,\\
 \frac{4 \Lip_1(H)}{\theta}+  \frac{12  \Lip_2(H)}{\theta+1}+\frac{16 \Lip_3(H)}{3(\theta+2)}, &\ \ H \in \cH_2.
\end{cases}
}
\end{definition}

We now state the  technical lemma. Denote the difference operator for signed measure~$\nu$ by
\be{
\Delta_\nu F(\mu):=F(\mu+\nu)-F(\mu).
}

\begin{lemma}\label{lem:domainhbds}
Let $\mu\in M_1$ and $\nu_1,\nu_2,\nu_3$ be bounded signed measures on $E$.
If $H\in \cH_1$ 
and $H(\mu)=h(\angle{\phi_1,\mu},\ldots,\angle{\phi_k,\mu})$ with $h\in \BC^{2,1}(\IR^k)$ and $\phi_i\in\rC(E)$, 
then for  $m=1,2,3$, and $L_m(H)$ given by Definition~\ref{DEF:D}, we have
\ba{ 
\bbabs{\prod_{i=1}^m \Delta_{\nu_i} H(\mu)} &\leq \Lip_m(H) \prod_{i=1}^m \norm{\nu_i}.
}
\end{lemma}
\begin{proof}
For $u,v\in\IR^k$, denote $\bar\Delta_v h(u)=h(u+v)-h(u)$ and denote
\be{
\angle{\bm\phi,\mu}=\bclr{\angle{\phi_1,\mu}, \ldots, \angle{\phi_k, \mu}}.
}
If $m=1,2$, apply \cite[Lemma~3]{Gan2017} to find
\ba{
\bbabs{\prod_{i=1}^m \Delta_{\nu_i} H(\mu)}
	&=\bbbabs{ \bbclr{\prod_{i=1}^m \bar\Delta_{\angle{\bm{\phi},\nu_i}}}h(\angle{\bm{\phi},\mu})}\\
	&\leq \abs{h}_{m} \prod_{i=1}^m \bnorm{\angle{\bm{\phi},\nu_i}}_1,
}
from which the claimed bound follows after noting that for $\phi\in\rC(E)$, $\abs{\angle{\phi,\nu}}\leq \norm{\phi}_\infty \norm{\nu}$. The argument is the same for $m=3$, but replacing~$\abs{h}_m$ by~$\abs{h}_{2,1}$.
\end{proof}

We can now state and prove the relevant facts about the Stein solution~$F_H$.
\begin{theorem}\label{thm:FVstnbds}
If $H:M_1\to\IR$ is bounded, then $F_H$ given at~\eq{eq:pdstnsol}  is well defined, and
\ben{\label{eq:solbd}
\norm{F_H}_{\infty}\leq \frac{2(\theta+1)}{\theta}\norm{\wt H}_{\infty}.
}
If $H\in\cH_1$ ($\cH_2$), then $F_H\in\cD_1$ ($\cD_2$) and satisfies 
$\cA F_H(\mu)=\wt H(\mu)$. Recalling the notation~$D_m(H;\theta)$, $m=1,2,3$, given by Definition~\ref{DEF:D}, if $x,y\in E$, then $\partial_x F_H, \partial_{xy} F_H$ exist, and 
\ban{
\bnorm{\partial_{x} F_H}_{\infty}&\leq D_1(H;\theta), \label{eq:1stderiv} \\
\bnorm{\partial_{xy} F_H}_{\infty}&\leq D_2(H;\theta).
\label{eq:2ndderiv}}
If~$H\in\cH_1$, then for $\mu,\nu\in M_1$
\besn{\label{eq:2ndderivwass}
\babs{\partial_{xy} F_H(\mu)-\partial_{xy} F_H(\nu)}&\leq \norm{\nu-\mu}D_3(H;\theta),
}
and if~$H\in\cH_2$, then for all $x,y,z\in E$, $\partial_{xyz}F_H$ exists and satisfies
\ben{\label{eq:3rdderiv}
\bnorm{\partial_{xyz} F_H}_{\infty}\leq D_3(H;\theta).
}
\end{theorem} 

\begin{proof}
We first give the probabilistic description of the transition semigroup of $(Z_\mu(t))_{t\geq0}$, a FV process with generator~\eq{eq:dpgen}, which is given in~\cite{Ethier1993a}; see also \cite[Theorem~5.5 and Remark following]{Feng2010}.
\begin{itemize}
\item Let~$(L_t)_{t\geq0}$ be a pure death process on~$\{0, 1, \ldots\}\cup\{\infty\}$ started at~$\infty$ with death rates 
\be{
q_{i, i-1}= \frac{1}{2} i(i-1+\theta).
}
\item Let $(X_i)_{i\geq1}$ be i.i.d.\ samples from $\mu$ and independent of $(L_t)_{t\geq0}$.
\item Given $(X_i)_{i\geq1}$ define the measures for $n=1,2,\ldots$,
\ben{\label{eq:nusamp}
\nu\t{1}_n:=\nu\t{1}_n(X_1,\ldots,X_n)=\frac{1}{n+\theta} \sum_{i=1}^n \delta_{X_i} + \frac{\theta}{n+\theta} \pi.
}
\end{itemize}
With this notation, we have $\law(Z_\mu(t)| L_t) =\DP(L_t+\theta, \nu\t{1}_{L_t})$. Otherwise put, 
if for each $n=0,1,2,\ldots$ we set $(P_{nj})_{j\geq1}\sim\PD(n+\theta)$ independent of $(\xi_{nj})_{j\geq1}$, which are conditionally i.i.d.\ given $(X_i)_{i\geq1}$ and $\nu\t{1}_n$-distributed, and these random objects are independent of $(L_t)_{t\geq0}$, we can set 
\be{
Z_{\mu}(t)=\sum_{j=1}^\infty P_{L_t j} \delta_{\xi_{L_t j}}.
}
For the rest of the proof, relabel $\wt H$ as $H$ (equivalently, assume without loss of generality that $\IE H(Z)=0$, noting this preserves domains $\cD_1$ and $\cD_2$), fix $H$, and relabel $F_H$ as $F$.
Following~\cite[Proof of Theorem~5]{Gan2017},
for~$n\geq 1$, let~$Y_n$ be the time the process~$L_t$ spends in state~$n$ and note that~$Y_n$ is exponentially distributed with rate~$n(n-1+\theta)/2$. Since
\be{
	\sum_{n\geq 1} \IE Y_n = \sum_{n\geq 1}\frac{2}{n(n+\theta-1)} \leq  \frac{2(\theta+1)}{\theta},
}
the random variable~$T = \inf\{t>0\,:\,L_t=0\} = \sum_{n\geq 1}Y_n$  satisfies $\IE[T]\leq 2(\theta+1)/\theta$. 
Since $\law(Z_\mu(t) | L_t=0)\sim \DP(\theta, \pi)$, we have~$\IE\bcls{{H}(Z_\mu(t))\big| L_t=0}=0$, and it follows that
\besn{\label{457}
\int_0^\infty \babs{\IE\bcls{ H(Z_\mu(t)) }} dt &\leq \int_0^\infty \norm{H}_\infty \IP(L_t>0) dt \\
	&= \norm{H}_\infty \int_0^\infty \IP(T>t) dt =\norm{H}_\infty\IE T < \infty.
}
Thus, if $\norm{H}_{\infty}<\infty$, then $F$ is well-defined and~\eq{eq:solbd} follows. Let $H\in\cH_1$ $(\cH_2)$; it is clear that from the form of these functions that $\norm{H}_\infty<\infty$, so $F$ is well-defined. 
To show that $F$ is in the relevant domain of~$\cal{A}$ and satisfies $\mathcal{A}F= H$,
we follow the argument of \cite[Pages~301-2]{Barbour1990} also used in \cite[Appendix~B]{Gorham2019}. 
First, it has been shown that both $(\cA, \cD_1)$ \cite{Fleming1979} and $(\cA, \cD_2)$ \cite{Ethier1993} generate Feller semigroups, 
so \cite[Proposition~1.5, Page~9]{Ethier1986} implies $F^{(s)}(\mu) := -\frac12 \int_0^s \IE H(Z_\mu(t)) dt$ is in the relevant  domain of $\cA$ and satisfies
\ben{\label{eq:FHs}
\cA F^{(s)}(\mu) = H(\mu)-\IE H(Z_\mu(s)).
}
Moreover, \cite[Corollary~1.6, Page~10]{Ethier1986} implies that $\mathcal{A}$ is a closed operator, 
so it is enough to show that as $s\to\infty$,
\ben{\label{458}
\norm{F^{(s)}-F}_\infty\to0 \,\,\, \mbox{ and } \,\,\, \norm{\mathcal{A}F^{(s)}-H}_{\infty} \to 0.
}
The first limit follows from~\eq{457}, which implies that independently of any $\mu$, as $s \to \infty$,
$\babs{\int_s^\infty \IE(H(Z_\mu(t))) dt}\leq \norm{H}_\infty \int_{s}^\infty \IP(T>t) dt  \to 0$. The second follows from~\eq{eq:FHs} and the convergence result~\cite[Corollary 1.2]{Ethier1993a}:
\be{
\norm{\mathcal{A}F^{(s)}-H}_{\infty}=\sup_{\mu \in M_1} \abs{\IE H(Z_\mu(s))} \leq \norm{H}_\infty \dtv( \law(Z_\mu(s), \DP(\theta, \pi))) \to 0.
}

Since $F\in \cD_1$ ($\cD_2$), Lemma~\ref{lem:derivsform} implies that $\partial_x F$ and $\partial_{xy} F$ exist. 
For the bounds on the derivative, we use the following key facts, proved last.
If $H\in\cH_1\cup \cH_2$, and $\{\nu_i\}_{i=1}^3, \{\mu_i\}_{i=1}^4\subseteq M_1$, and $\{\alpha_i\}_{i=1}^3$ are positive numbers with $\alpha_1+\alpha_2+\alpha_3<1$, then
\ban{
&\babs{\Delta_{\alpha_1(\nu_1-\mu_1)} F\bclr{\alpha_1\mu_1+(1-\alpha_1)\mu_4}} 
	\leq \alpha_1 \bbbclr{\frac{4\Lip_1(H)}{\theta}+\lito(1)\Ind[H\in \cH_2]}, \label{eq:gddiffbd1}\\
\begin{split} \label{eq:gddiffbd2}
	&\babs{\Delta_{\alpha_1(\nu_1-\mu_1)} \Delta_{\alpha_2(\nu_2-\mu_2)} F\bclr{\alpha_1\mu_1+\alpha_2\mu_2+(1-\alpha_1-\alpha_2)\mu_4}} \\
	&\hspace{6cm}\leq\alpha_1\alpha_2 \bbbclr{\frac{4  \Lip_2(H)}{\theta+1}+\lito(1)\Ind[H\in \cH_2]}, 
%	&\hspace{6cm}\leq
%	\alpha_1\alpha_2 \times 
%	\begin{cases}
%	\frac{4  \Lip_2(H)}{\theta+1}, & H\in \cH_1, \\
%	\frac{4  \Lip_2(H)}{\theta+1}+\lito(1) & H\in \cH_2,
%	\end{cases}
\end{split} 
\\
\begin{split}\label{eq:gddiffbd3}
&\bbabs{\Delta_{\alpha_1(\nu_1-\mu_1)} \Delta_{\alpha_2(\nu_2-\mu_2)} \Delta_{\alpha_3(\nu_3-\mu_3)}F\bbclr{\sum_{i=1}^3 \alpha_i\mu_i+(1-\sum_{i=1}^3 \alpha_i)\mu_4}}  \\
&\hspace{6cm}\leq\alpha_1\alpha_2\alpha_3 \bbbclr{ \frac{16 \Lip_3(H)}{3(\theta+2)} +\lito(1)\Ind[H\in \cH_2]},
%&\hspace{6cm}\leq
%\alpha_1\alpha_2\alpha_3 
%\times
%\begin{cases}
% \frac{16 \Lip_3(H)}{3(\theta+2)}, & H\in \cH_1, \\
%\frac{16 \Lip_3(H)}{3(\theta+2)} +\lito(1), & H\in \cD_2, 
%\end{cases}
\end{split}
}
where~$\lito(1)$ in~\eq{eq:gddiffbd1} is as $\alpha_1 \to 0$, in~\eq{eq:gddiffbd2} is as $\alpha_1,\alpha_2\to 0$ and in~\eq{eq:gddiffbd3} is as $\alpha_1,\alpha_2,\alpha_3\to 0$.  These inequalities are written with this parameterisation because our proof of them requires that all arguments appearing when expanding out the $\Delta$'s are non-negative measures. 

Assuming~\eq{eq:gddiffbd1},~\eq{eq:gddiffbd2}, and~\eq{eq:gddiffbd3}, we bound the derivatives.  
For the first derivative bound~\eq{eq:1stderiv}, compute
\bes{
\partial_x F(\mu) &= \lim_{\eps \to 0^+}\frac{1}{\eps}\left[F((1-\eps)\mu+\eps \delta_x)-F(\mu)\right] = \lim_{\eps \to 0^+}\frac{1}{\eps}\Delta_{\eps(\delta_x-\mu)}F(\mu).
%\\
%	&= -\lim_{\eps \to 0} \frac{1}{\eps}\int_0^\infty \IE \bcls{H\bclr{Z_{(1-\eps)\mu+\eps \delta_x}(t)}}-\IE\bcls{H\bclr{Z_\mu(t)}}dt .
}
Now taking absolute value and using~\eq{eq:gddiffbd1} with $\alpha_1=\eps$, $\nu_1=\delta_x$ and $\mu_1=\mu_4=\mu$, we easily find the desired bound.
%\be{
%\abs{\partial_{x} F(\mu)}\leq 4 \theta^{-1} \Lip_1(H).
%}

For the second derivative bound~\eq{eq:2ndderiv}, by direct calculation, we have 
\ba{
&\partial_{xy} F(\mu)% \notag \\
	=\lim_{\eps_1,\eps_2\to 0^+}\bbbcls{\frac{
%	F\bclr{(1-\eps_2)(1-\eps_1) \mu + \eps_1 \d_x + \eps_2(1-\eps_1) \d_y}-
\Delta_{\eps_1\eps_2(\mu-\delta_y)}F\bclr{(1-\eps_1-\eps_2)\mu + \eps_1 \d_x + \eps_2 \d_y}}{\eps_1\eps_2}  \notag 
%		&\hspace{6mm}+
%		+\lim_{\eps_1,\eps_2\to 0}
		+\frac{ \Delta_{\eps_1( \delta_x-\mu)}\Delta_{\eps_2 ( \delta_y-\mu)} F(\mu) }{\eps_1\eps_2}}. 
%	&=\lim_{\eps_1,\eps_2\to 0}\frac{\Delta_{\eps_1\eps_2(\mu-\delta_y)} \Delta_{\eps_1 \delta_x}\Delta_{\eps_1\delta_x} F(\mu)- \Delta_{\eps_1 \delta_x}\Delta_{\eps_1\delta_x} F(\mu)}{\eps_1\eps_2} + \frac{\Delta_{\eps_1 \delta_x}\Delta_{\eps_1\delta_x} F(\mu)-\Delta_{\eps_1\delta_x} F(\mu)- \Delta_{\eps_2\delta_y} F(\mu)+F(\mu)}{\eps_1\eps_2}.
}
Now taking absolute value, using the triangle inequality, and applying~\eq{eq:gddiffbd1} 
with $\alpha_1=\eps_1\eps_2$ $\nu_1=\delta_y$, $\mu_1=\mu$ and 
\be{
\mu_4=(1-\eps_1\eps_2)^{-1} \bcls{(1-\eps_1-\eps_2)\mu+\eps_1\delta_x+\eps_2(1-\eps_1)\delta_y}
}
to the first term, and~\eq{eq:gddiffbd2} with $\alpha_i=\eps_i, i=1,2$ and $\nu_1=\delta_x, \nu_2=\delta_y$ and $\mu_4=\mu_1=\mu_2=\mu$,
 gives the desired bound.

The Lipschitz second derivative bound~\eq{eq:2ndderivwass} follows the same idea, but is a bit more complicated due to being a higher order, but also because the term to bound is not the same form as a third derivative. Assume now $H\in\cH_1$, let $\mu,\nu\in M_1$, and define $\tilde \mu, \mu', \nu'\in M_1$ by the decomposition \ba{
\nu&=(1-\eps_3)\tilde \mu+\eps_3\nu' \,\, \mbox{ and } \,\, \mu=(1-\eps_3)\tilde \mu+\eps_3\mu',
}
where $\eps_3=\dtv(\mu,\nu)=(1/2) \norm{\mu-\nu}$. 
Now,  for measures $\hat \nu_i$
with $\hat\eps_i:=\norm{\hat\nu_i}\leq 1$, $i=1,2,3$, and function $F:M_1\to\IR$, define the modified ``third difference" operator
\ba{
\Delsthree{\hat\nu_1}{\hat\nu_2}{\hat\nu_3}F(\hat \mu)&=
F\bclr{(1-\hat\eps_1)(1-\hat\eps_2)(1-\hat\eps_3)\hat\mu + (1-\hat\eps_1)(1-\hat\eps_2) \hat\nu_3 +  (1-\hat\eps_1) \hat\nu_2 +  \hat\nu_1} \\
	&\qquad-F\bclr{(1-\hat\eps_1)(1-\hat\eps_2)\hat\mu + (1-\hat\eps_1)\hat\nu_2 +   \hat\nu_1}\\
	&\qquad- F\bclr{(1-\hat\eps_1)(1-\hat\eps_3)\hat\mu + (1-\hat\eps_1)\hat\nu_3 +   \hat\nu_1}\\
	&\qquad-  F\bclr{(1-\hat\eps_2)(1-\hat\eps_3)\hat\mu + (1-\hat\eps_2)\hat\nu_3 +  \hat\nu_2}\\
	&\qquad+ F\bclr{(1-\hat\eps_1)\hat\mu + \hat\nu_1}+ F\bclr{(1-\hat\eps_2)\hat\mu + \hat\nu_2}\\
	&\qquad+ F\bclr{(1-\hat\eps_3)\hat\mu + \hat\nu_3}-F(\hat\mu).
}
Then it is straightforward to see that  
\ba{
\partial_{xy} F(\nu)-\partial_{xy} F(\mu)&=\lim_{\eps_1,\eps_2\to 0}\bbbcls{\frac{\Delsthree{\eps_1 \delta_x}{\eps_2 \delta_y}{\eps_3\nu'} F(\tilde \mu) }{\eps_1\eps_2}-\frac{\Delsthree{\eps_1 \delta_x}{\eps_2 \delta_y}{\eps_3\mu'}  F(\tilde \mu) }{\eps_1\eps_2}},
}
and therefore
\ban{
\babs{\partial_{xy} F(\nu)-\partial_{xy} F(\mu)}&\leq \limsup_{\eps_1,\eps_2\to 0}\bbbabs{\frac{\Delsthree{\eps_1 \delta_x}{\eps_2 \delta_y}{\eps_3\nu'} F(\tilde \mu) }{\eps_1\eps_2}}+ \limsup_{\eps_1,\eps_2\to 0}\bbbabs{\frac{\Delsthree{\eps_1 \delta_x}{\eps_2 \delta_y}{\eps_3\mu'}  F(\tilde \mu) }{\eps_1\eps_2}}.\label{eq:lipdiv}
}
By a straightforward (though tedious) calculation, we have, for small $\eps_1,\eps_2$,
\ban{
&\Delsthree{\eps_1 \delta_x}{\eps_2 \delta_y}{\eps_3\mu'}  F(\tilde \mu) \notag  \\
	&=\Delta_{\eps_1(\delta_x-\tilde \mu)}\Delta_{\eps_2(\delta_y-\tilde \mu)}\Delta_{\eps_3(\mu'-\tilde \mu)}F(\tilde \mu) \notag \\
		&\qquad+\Delta_{\eps_3(1-\eps_1-\eps_2)(\tilde \mu-\mu')}\Delta_{\eps_1\eps_2(\delta_y-\tilde \mu)}F\bbclr{(1-\eps_1-\eps_2-\eps_3+\eps_1\eps_2+\eps_1\eps_3+\eps_2\eps_3)\tilde \mu \notag\\
			&\hspace{7cm} + \eps_3(1-\eps_1-\eps_2) \mu'+\eps_2(1-\eps_1)\delta_y+\eps_1\delta_x}\notag \\
\begin{split}	\label{eq:3rddiffs}		
		&\qquad+\Delta_{\eps_1(\tilde \mu-\delta_x)}\Delta_{\eps_2\eps_3(\mu'-\tilde \mu)}F\bbclr{(1-\eps_1-\eps_2-\eps_3+\eps_1\eps_3+\eps_2\eps_3)\tilde \mu  \\
			&\hspace{7cm} + \eps_3(1-\eps_1-\eps_2) \mu'+\eps_2\delta_y+\eps_1\delta_x} 
\end{split}	 \\
			&\qquad+\Delta_{\eps_1\eps_3(\mu'-\tilde \mu)}\Delta_{\eps_2\eps_3(\mu'-\tilde \mu)}F\bclr{(1-\eps_2-\eps_3+\eps_1\eps_3+\eps_2\eps_3)\tilde \mu + \eps_3(1-\eps_1-\eps_2) \mu'+\eps_2\delta_y} \notag\\
			&\qquad+\Delta_{\eps_2(\tilde \mu-\delta_y)}\Delta_{\eps_1\eps_3(\mu'-\tilde \mu)}F\bclr{(1-\eps_1-\eps_2-\eps_3+\eps_1\eps_3)\tilde \mu + \eps_3(1-\eps_1) \mu'+\eps_2\delta_y+\eps_1\delta_x}\notag \\
			&\qquad-\Delta_{\eps_1\eps_2\eps_3(\tilde\mu-\mu')} F\bclr{(1-\eps_1)(1-\eps_2)(1-\eps_3)\tilde \mu+ \eps_3(1-\eps_1)(1-\eps_2) \mu' + \eps_2(1-\eps_1)\delta_y+\eps_1\delta_x},\notag
}
 and a similar decomposition holds for $\Delsthree{\eps_1 \delta_x}{\eps_2 \delta_y}{\eps_3\nu'}  F(\tilde \mu)$ by replacing~$\mu'$ with~$\nu'$.
Taking the absolute value, using the triangle inequality, and~\eq{eq:gddiffbd1},~\eq{eq:gddiffbd2}, and~\eq{eq:gddiffbd3} (noting $H\in\cH_1$) similar to above, we find that for $\hat\mu=\mu',\nu'$, 
\be{
\babs{\Delsthree{\eps_1 \delta_x}{\eps_2 \delta_y}{\eps_3\hat \mu}  F(\tilde \mu)}\leq \eps_1\eps_2\eps_3\bbbclr{  \frac{16 \Lip_3(H)}{3(\theta+2)}+ (3+\eps_3-\eps_1-\eps_2)\frac{4  \Lip_2(H)}{\theta+1}+ \frac{4 \Lip_1(H)}{\theta}}.
}
Applying this inequality in~\eq{eq:lipdiv}, noting that $2\eps_3=\norm{\nu-\mu}$, we find
\be{
\babs{\partial_{xy} F(\nu)-\partial_{xy} F(\mu)}\leq \norm{\nu-\mu}
\bbbclr{  \frac{16 \Lip_3(H)}{3(\theta+2)}+ \bclr{3+\tsfrac{1}{2}\norm{\nu-\mu}} \frac{4  \Lip_2(H)}{\theta+1}+ \frac{4 \Lip_1(H)}{\theta}},
}
and the result follows after noting $\norm{\nu-\mu}\leq 2$. (Note that it is important here not to have the $\lito(1)$ terms in~\eq{eq:gddiffbd2} and~\eq{eq:gddiffbd3} as we are not sending $\eps_3\to0$.) 

If $H\in\cH_2$, then Lemma~\ref{lem:derivsform} shows the existence of the limit  in the definition of derivative:
\be{
\partial_{xyz}F(\mu)=\lim_{\eps_1,\eps_2,\eps_3\to 0} \frac{\Delsthree{\eps_1\delta_x}{\eps_2\delta_y}{\eps_3\delta_z} F( \mu)}{\eps_1\eps_2\eps_3}. 
}
Applying a decomposition analogous to~\eq{eq:3rddiffs}, leads to
\be{
\babs{\Delsthree{\eps_1\delta_x}{\eps_2\delta_y}{\eps_3\delta_z} F( \mu)}\leq \eps_1\eps_2\eps_3\bbbclr{  \frac{16 \Lip_3(H)}{3(\theta+2)}+ (3+\eps_3-\eps_1-\eps_2)\frac{4  \Lip_2(H)}{\theta+1}+ \frac{4 \Lip_1(H)}{\theta}+\lito(1)},
}
%\note{
%\be{
%\babs{\Deltwo_{\eps_1, \delta_x}\Deltwo_{\eps_2, \delta_y}\Deltwo_{\eps_3, \delta_z} F( \mu)}\leq \eps_1\eps_2\eps_3\bbbclr{  \frac{16 \Lip_3(H)}{3(\theta+2)}+ (3+\eps_1-\eps_2-\eps_3)\frac{4  \Lip_2(H)}{\theta+1}+ \frac{4 \Lip_1(H)}{\theta}+\lito(1)},
%}
%}
where now, as per~\eq{eq:gddiffbd3}, the~$\lito(1)$ is as $\eps_1,\eps_2,\eps_3\to 0$.
The bound~\eq{eq:3rdderiv} on the third derivative easily follows.

We now turn to the proofs of~\eq{eq:gddiffbd1},~\eq{eq:gddiffbd2}, and~\eq{eq:gddiffbd3}, which are different 
for $H\in\cH_1$ and $H\in \cH_2$, so we separate the cases.

\noindent{\bf Case 1:  $H\in \cH_1$.}
From the definition of~$F$, we can write
\be{
\Delta_{\alpha_1(\nu_1-\mu_1)} F\bclr{\alpha_1\mu_1+(1-\alpha_1)\mu_4}=\int_0^\infty \IE\bcls{H\bclr{Z_{\alpha_1\mu_1+(1-\alpha_1)\mu_4}(t)}}-\IE \bcls{H\bclr{Z_{\alpha_1\nu_1+(1-\alpha_1)\mu_4}(t)}}dt.
}
To upper bound the absolute value of this quantity, we use the 
probabilistic description of the FV process to define a coupling
\ba{
(Z(t),Z\s{1}(t)):=\bclr{Z_{\alpha_1\mu_1+(1-\alpha_1)\mu_4}(t),Z_{\alpha_{1}\nu_1+(1-\alpha_1)\mu_4}(t)}.
}
First let $(U_i)_{i\geq1}, (U_i\s{1})_{i\geq1}, (V_i\s{1})_{i\geq1}$ be independent i.i.d.\ sequences distributed as $\mu_4, \mu_1,\nu_1$. Let $(X_i,X_i\s{1})_{i\geq1}$ be an i.i.d.\ sequence of the maximal coupling between $\alpha_1\mu_1+(1-\alpha_1)\mu_4$ and~$\alpha_{1}\nu_1+(1-\alpha_1)\mu_4$:
\besn{\label{eq:dtv1}
&\IP\bclr{\clr{X_i, X_i\s{1}}= (U_i, U_i)| (U_i, U_i\s{1}, V_i\s{1}) }=1-\alpha_1, \\
&\IP\bclr{\clr{X_i, X_i\s{1}}= (U_i\s{1}, V_i\s{1})|(U_i, U_i\s{1}, V_i\s{1}) }=\alpha_1.
} 
Given the i.i.d.\ sequence
$(X_i, X_i\s{1})_{i\geq1}$, define the measures for $n\geq 1$,
\ben{\label{eq:numeas}
\nu_{n}\t{2}:=\nu_{n}\t{2}((X_1, X_1\s{1}),\ldots,(X_{n}, X_{n}\s{1}))=\frac{1}{n+\theta} \sum_{i=1}^{n} \delta_{(X_i,X_i\s{1})} + \frac{\theta}{n+\theta} \pi^{*2},
}
where $\pi^{*2}(B)=\pi(\{x: (x,x)\in B\})$.
Let also $(\xi_{n j}, \xi_{n j}\s{1})_{j\geq 1}$ be conditionally i.i.d.\ $\nu_{n}\t{2}$-distributed, and 
$(P_{n j})_{j\geq1}\sim\PD(n+\theta)$ independent of the variables above. Finally, define
\ba{
Z(t)%:=Z_{\alpha_1\mu_1+(1-\alpha_1)\mu_4}(t)
=\sum_{j=1}^\infty P_{L_t j} \delta_{\xi_{L_t j}}, \mbox{\, and \,\,}
Z\s{1}(t)%:=Z_{\alpha_1\nu_1+(1-\alpha_1)\mu_4}(t)
=\sum_{j=1}^\infty P_{L_t j} \delta_{\xi_{L_t j}\s{1}},
}
where the notation means, for example, that given $L_t=n$,
\be{
Z(t)=\sum_{j=1}^\infty P_{n j} \delta_{\xi_{n j}}.
}   
That the marginal distributions of $Z(t)$ and $Z\s{1}(t)$ are correct easily comes from checking the marginal distributions of $\nu_{n}\t{2}$ match the appropriate~$\nu\t{1}_n$ given at~\eq{eq:nusamp}.

Now, for each $n\geq1$, define $B_{n1}=\{j:\xi_{n j}\s{1}\not= \xi_{n j}\}$. We can then write
\ban{\label{eq:gddecomp1}
Z\s{1}(t)&:=Z(t)+\sum_{j\in B_{L_t 1}} P_{L_t j} (\delta_{\xi\s{1}_{L_t j}}-\delta_{\xi_{L_t j}}),
}
which implies
\besn{\label{eq:deriv1y}
H(Z\s{1}(t))-H(Z(t)) &=\Delta_{\sum_{j\in B_{L_t 1}} P_{L_t j} (\delta_{\xi\s{1}_{L_t j}}-\delta_{\xi_{L_t j}})} H(Z(t)).
}
Thus, for $H\in \cH_1$, Lemma~\ref{lem:domainhbds}
implies
\besn{\label{eq:deriv1z}
\babs{H(Z\s{1}(t))-H(Z(t))} 
	&\leq \Lip_1(H)\bbnorm{ \sum_{j\in B_{L_t 1}} P_{L_t j} (\delta_{\xi\s{1}_{L_t j}}-\delta_{\xi_{L_t j}})} \\
	&\leq 2 \Lip_1(H)\sum_{j\in B_{L_t 1}} P_{L_t j}.
}
Using~\eq{eq:deriv1z} and the definition of~$F$, we find
\besn{\label{eq:deriv1a}
\babs{\Delta_{\alpha_1(\nu_1-\mu_1)} F\bclr{\alpha_1\mu_1+(1-\alpha_1)\mu_4}} 
	&\leq \int_0^\infty\IE \babs{H(Z\s{1}(t))-H(Z(t))} dt \\
	&\leq  2 \Lip_1(H)\int_0^\infty\IE \bbcls{ \sum_{j\in B_{L_t 1}} P_{L_t j}} dt \\
	&= 2 \Lip_1(H) \int_0^\infty\IE \bbcls{\sum_{j=1}^\infty \Ind[j\in B_{L_t 1}]P_{L_t j} }dt \\
	&= 2 \Lip_1(H) \int_0^\infty\IE \bbcls{ \sum_{j=1}^\infty \IP(j\in B_{L_t 1} |L_t) \IE [P_{L_t j}| L_t] } dt,
	}
where we use conditional independence in the last line.
Now, we have
\bes{
\IP\bclr{ j\in B_{L_t 1}|L_t,(X_i, X_i\s{1}))_{i\geq1}}  
	&=\frac{1}{L_t+\theta}\sum_{\ell=1}^{L_t} \Ind[X_\ell\s{1}\not= X_\ell],
}
and averaging out $(X_i, X_i\s{1})_{i\geq1}$, we have
\besn{\label{eq:deriv1b}
\IP\bclr{ j\in B_{L_t 1} |L_t } &=\frac{1}{L_t+\theta}\sum_{\ell=1}^{L_t} \IP(X_\ell\s{1}\not= X_\ell) 
		\leq \alpha_1 \frac{L_t}{L_t+\theta},
}
where the inequality is because of~\eq{eq:dtv1}.
Applying the last two displays to~\eq{eq:deriv1a} and recalling the definition of $(Y_n)_{n\geq1}$ above, we have 
\besn{\label{eq:deriv1c}
\babs{\Delta_{\alpha_1(\nu_1-\mu_1)} F\bclr{\alpha_1\mu_1+(1-\alpha_1)\mu_4}} 
		&\leq 2\alpha_1\Lip_1(H)  \int_0^\infty\IE \bbbcls{ \frac{L_t}{L_t+\theta}  \IE\bbcls{\sum_{j=1}^\infty P_{L_t j} | L_t} } dt \\
	&= 2\alpha_1\Lip_1(H) \int_0^\infty\IE \bbbcls{ \frac{L_t}{L_t+\theta} }  dt	 \\
	&=2\alpha_1 \Lip_1(H) \sum_{n\geq1}\bbbcls{\frac{n}{n+\theta}} \IE Y_n \\
	&=2 \Lip_1(H)\alpha_1  \sum_{n=1}^\infty \frac{n }{n+\theta} \frac{2}{n(n+\theta-1)} \\
	&=  \frac{4\alpha_1 \Lip_1(H)}{\theta},
}	
which is~\eq{eq:gddiffbd1}.

For~\eq{eq:gddiffbd2}, we follow a similar strategy and define a coupling of 
\bes{
(Z(t), Z\s{1}(t),Z\s{2}(t), Z\s{1,2}(t)):=
	&\bclr{Z_{\alpha_{1}\mu_1+\alpha_{2}\mu_2+(1-\alpha_1-\alpha_2)\mu_4}(t),
	Z_{\alpha_{1}\nu_1+\alpha_{2}\mu_2+(1-\alpha_1-\alpha_2)\mu_4}(t), \\
	&\hspace{7mm} Z_{\alpha_{1}\mu_1+\alpha_{2}\nu_2+(1-\alpha_1-\alpha_2)\mu_4}(t),
	Z_{\alpha_{1}\nu_1+\alpha_{2}\nu_2+(1-\alpha_1-\alpha_2)\mu_4}(t)},
}
where the reuse of the notation on the left hand side will not cause a problem.
Building from the ideas of the previous coupling, let 
 $(U_i)_{i\geq1}, (U_i\s{j})_{i\geq1}, (V_i\s{j})_{i\geq1}$, $j=1,2$ be independent i.i.d.\ $\mu_4, \mu_j,\nu_j$-distributed sequences. Now define an i.i.d.\ sequence
$(X_i, X_i\s{1}, X_i\s{2}, X_i\s{1,2})_{i\geq1}$
by
\ba{
&\IP\bclr{\clr{X_i, X_i\s{1}, X_i\s{2}, X_i\s{1,2}}= (U_i, U_i,U_i,U_i)| U_i, (U_i\s{j},V_i\s{j})_{j=1}^2}=1-\alpha_1-\alpha_2,  \\
&\IP\bclr{\clr{X_i, X_i\s{1}, X_i\s{2}, X_i\s{1,2}}=  (U_i\s{1}, V_i\s{1} ,U_i\s{1}, V_i\s{1})| U_i, (U_i\s{j},V_i\s{j})_{j=1}^2 }=\alpha_1,  \\
&\IP\bclr{\clr{X_i, X_i\s{1}, X_i\s{2}, X_i\s{1,2}}= (U_i\s{2}, U_i\s{2},V_i\s{2},V_i\s{2})| U_i, (U_i\s{j},V_i\s{j})_{j=1}^2 }=\alpha_2.
} 
Given $(X_i, X_i\s{1}, X_i\s{2}, X_i\s{1,2})_{i\geq1}$, define the measures for $n\geq1$,
\besn{\label{eq:nun4}
\nu_{n}\t{4}:&=\nu_{n}\t{4}((X_1, X_1\s{1}, X_1\s{2}, X_1\s{1,2}),\ldots,(X_{n}, X_{n}\s{1}, X_{n}\s{2}, X_{n}\s{1,2}))  \\
	&=\frac{1}{n+\theta} \sum_{i=1}^{n} \delta_{(X_i, X_i\s{1}, X_i\s{2}, X_i\s{1,2})} + \frac{\theta}{n+\theta} \pi^{*4},
}
where $\pi^{*4}(B)=\pi(\{x: (x,x,x,x)\in B\})$.
Now, given the above, let $(\xi_{n j},  \xi_{n j}\s{1},\xi_{n j}\s{2},\xi_{n j}\s{1,2})_{j\geq 1}$ be conditionally i.i.d.\ $\nu_{n}\t{4}$-distributed, and 
$(P_{n j})_{j\geq1}\sim\PD(n+\theta)$ independent of the variables above. Finally, define
\ba{
Z(t)&=\sum_{j=1}^\infty P_{L_t j} \delta_{\xi_{L_t j}}, &Z\s{1}(t)=\sum_{j=1}^\infty P_{L_t j} \delta_{\xi_{L_t j}\s{1}}, \\
Z\s{2}(t)&=\sum_{j=1}^\infty P_{L_t j} \delta_{\xi_{L_t j}\s{2}}, &Z\s{1,2}(t)=\sum_{j=1}^\infty P_{L_t j} \delta_{\xi_{L_t j}\s{1,2}}.
}
Now, for each $n\geq1$ and $i=1,2$, define $B_{ni}=\{j:\xi_{n j}\s{i}\not= \xi_{n j}\}$, so that we can write
\ba{
Z\s{1}(t)&=Z(t)+\sum_{j\in B_{L_t 1}} P_{L_t j} (\delta_{\xi_{L_t j}\s{1}}-\delta_{\xi_{L_t j}}), \\
Z\s{2}(t)&=Z(t)+\sum_{j\in B_{L_t 2}} P_{L_t j} (\delta_{\xi_{L_t j}\s{2}}-\delta_{\xi_{L_t j}}), \\
Z\s{1,2}(t)&=Z(t)+\sum_{j\in B_{L_t 1}} P_{L_t j} (\delta_{\xi_{L_t j}\s{1}}-\delta_{\xi_{L_t j}})+\sum_{j\in B_{L_t 2}} P_{L_t j} (\delta_{\xi_{L_t j}\s{2}}-\delta_{\xi_{L_t j}}),
}
where in the last expression note that if $i\in\{1,2\}$ and $j\in B_{nj}$, then $\xi_{nj}\s{1,2}=\xi_{nj}\s{i}$.
Thus
\bes{
H(Z\s{1,2}(t))&-H(Z\s{1}(t))-H(Z\s{2}(t))+H(Z(t)) \\
	&=\Delta_{\sum_{j\in B_{L_t 1}} P_{L_t j} (\delta_{\xi_{L_t j}\s{1}}-\delta_{\xi_{L_t j}})}\Delta_{\sum_{j\in B_{L_t 2}} P_{L_t j} (\delta_{\xi_{L_t j}\s{2}}-\delta_{\xi_{L_t j}})}H(Z(t)),
%	&=h\bbclr{\angle{\vect{\phi}, Z_{\mu}(t)}+ \sum_{j\in B_{L_t 1}} P_{L_t j} \bclr{\vect{\phi}(x)-\vect{\phi}(\xi_{L_t j})}+ \sum_{j\in B_{L_t 2}} P_{L_t j} \bclr{\vect{\phi}(y)-\vect{\phi}(\xi_{L_t j})}} \\
%	&\hspace{7mm}-h\bbclr{\angle{\vect{\phi}, Z_{\mu}(t)}+ \sum_{j\in B_{L_t 1}} P_{L_t j} \bclr{\vect{\phi}(x)-\vect{\phi}(\xi_{L_t j})}} \\
%	&\hspace{7mm}-h\bbclr{\angle{\vect{\phi}, Z_{\mu}(t)}+ \sum_{j\in B_{L_t 2}} P_{L_t j} \bclr{\vect{\phi}(y)-\vect{\phi}(\xi_{L_t j})}} + h\bclr{\angle{\vect{\phi}, Z_{\mu}(t)}}.
}
so that for $H\in \cH_1$, Lemma~\ref{lem:domainhbds} implies
%Thus, according to Lemma~3 of 
%\cite{Gan2017}, we have
\bes{
\babs{H(Z\s{1,2}(t))&-H(Z\s{1}(t))-H(Z\s{2}(t))+H(Z(t))} \\
	&\leq \Lip_2(H) \bbnorm{\sum_{j\in B_{L_t 1}} P_{L_t j} (\delta_{\xi\s{1}_{L_{t j}}}-\delta_{\xi_{L_t j}})}\bbnorm{\sum_{j\in B_{L_t 2}} P_{L_t j} (\delta_{\xi\s{2}_{L_{t j}}}-\delta_{\xi_{L_t j}})}\\
%	&\leq \abs{h}_2 \bbcls{\sum_{j\in B_{L_t 1}} P_{L_t j} \sum_{i=1}^k \babs{\phi_i(x)-\phi_i(\xi_{L_t j})}}\bbcls{\sum_{k\in B_{L_t 2}} P_{L_t k}\sum_{i=1}^k \babs{\phi_i(y)-\phi_i(\xi_{L_t j})}} \\
	&\leq4 \Lip_2(H)\bbcls{\sum_{j\in B_{L_t 1}} P_{L_t j}}\bbcls{\sum_{k\in B_{L_t 2}} P_{L_t k}}.
}
Using this inequality and the definition of~$F$, we find
\besn{\label{eq:deriv2a}
\babs{\Delta_{\alpha_1( \nu_1-\mu_1)}&\Delta_{\alpha_2 ( \nu_2-\mu_2)} F(\alpha_1\mu_1+\alpha_2\mu_2+(1-\alpha_1-\alpha_2)\mu_4)} \\
	&\leq \int_0^\infty\IE \babs{H(Z\s{1,2}(t))-H(Z\s{1}(t))-H(Z\s{2}(t))+H(Z(t))} dt \\
	&\leq  4 \Lip_2(H)\int_0^\infty\IE \bbcls{\bclr{\sum_{j\in B_{L_t 1}} P_{L_t j}}\bclr{\sum_{k\in B_{L_t 2}} P_{L_t k}}} dt \\
	&= 4 \Lip_2(H) \int_0^\infty\IE \bbcls{\sum_{j,k=1}^\infty \Ind[j\in B_{L_t 1}]P_{L_t j} \Ind[k\in B_{L_t 2}] P_{L_t k}} dt \\
	&= 4 \Lip_2(H) \int_0^\infty\IE \bbcls{ \sum_{j,k=1}^\infty \IP(j\in B_{L_t 1}, k\in B_{L_t 2} |L_t) \IE [P_{L_t j} P_{L_t k}| L_t] } dt,
	}
where we use conditional independence in the last line.
Now, we have
\bes{
\IP\bclr{ j\in B_{L_t 1},& k\in B_{L_t 2} |L_t,(X_i, X_i\s{1}, X_i\s{2}, X_i\s{1,2})_{i\geq1}}  \\
	&=\Ind[j\not=k](L_t+\theta)^{-2}\sum_{\ell, m=1}^{L_t} \Ind[X_\ell\s{1}\not=X_\ell, X_m\s{2}\not=X_m],
}
and averaging out $(X_i, X_i\s{1}, X_i\s{2}, X_i\s{1,2})_{i\geq1}$, we have
\bes{
\IP\bclr{ j\in B_{L_t 1}, k\in B_{L_t 2} |L_t }  
	&=\Ind[j\not=k](L_t+\theta)^{-2}\sum_{\ell, m=1}^{L_t} \IP(X_\ell\s{1}\not=X_\ell, X_m\s{2}\not=X_m) \\
		&\leq \Ind[j\not=k](L_t+\theta)^{-2}\sum_{\ell, m=1}^{L_t} \alpha_1\alpha_2 \Ind[\ell\not=m]\\
		&=\alpha_1\alpha_2 \Ind[j\not=k]\frac{L_t(L_t-1)}{(L_t+\theta)^{2}}.
}
Applying the last two displays in~\eq{eq:deriv2a}, we have 
\bes{
	\babs{
	\Delta_{\alpha_1( \nu_1-\mu_1)}&\Delta_{\alpha_2 ( \nu_2-\mu_2)} F(\alpha_1\mu_1+\alpha_2\mu_2+(1-\alpha_1-\alpha_2)\mu_4)} \\
	&\leq 4\alpha_1\alpha_2 \Lip_2(H)  \int_0^\infty\IE \bbbcls{ \frac{L_t(L_t-1)}{(L_t+\theta)^2}  \IE\bbcls{\sum_{j\not=k} P_{L_t j} P_{L_t k}| L_t} } dt \\
	&= 4\alpha_1\alpha_2\Lip_2(H) \int_0^\infty\IE\bbbcls{\frac{L_t(L_t-1)}{(L_t+\theta)^2} \frac{L_t+\theta}{L_t+\theta+1}}dt 	 \\
	&=4\alpha_1\alpha_2 \Lip_2(H) \sum_{n\geq1}\bbbcls{\frac{n(n-1)}{(n+\theta)(n+\theta+1)}} \IE Y_n \\
		&=8\alpha_1\alpha_2 \Lip_2(H) \sum_{n\geq1}\bbbcls{\frac{(n-1)}{(n+\theta)(n+\theta+1)(n+\theta-1)}} \\
	&= \frac{4}{1+\theta}\alpha_1\alpha_2 \Lip_2(H),
}	
where in the first equality, we use that if $(\tilde P_j)_{j\geq1}\sim \PD(\tilde \theta)$, then
\be{
\IE \sum_{j\not=k} \tilde P_j \tilde P_k= \frac{\tilde \theta}{\tilde \theta+1},
}
which can be seen most easily from Kingman's paintbox representation, which says the sum on the left hand side is the chance of getting the partition~$(1,1)$ when sampling twice from~$\PD(\tilde \theta)$; this is the same as starting a table on the second step of the Chinese Restaurant Process; see \cite[Section~3.1]{Pitman2006}.
Thus we have shown~\eq{eq:gddiffbd2}.

To show~\eq{eq:gddiffbd3}, we again define a coupling
\bes{
\bclr{Z(t), Z\s{1}(t)&, Z\s{2}(t), Z\s{3}(t), Z\s{1,2}(t),Z\s{1,3}(t),Z\s{2,3}(t),Z\s{1,2,3}(t)} \\
&:=
\bbclr{
	Z_{\sum_{i=1}^3 \alpha_i \mu_i +\bclr{1-\sum_{i=1}^3 \alpha_i \mu_i}\mu_4},
	Z_{\alpha_1\nu_1+\sum_{i=2,3} \alpha_i \mu_i +\bclr{1-\sum_{i=1}^3 \alpha_i \mu_i}\mu_4}, \\
&\hspace{10mm}	Z_{\alpha_2\nu_2+\sum_{i=1,3} \alpha_i \mu_i +\bclr{1-\sum_{i=1}^3 \alpha_i \mu_i}\mu_4},
	Z_{\alpha_3\nu_3+\sum_{i=1,2} \alpha_i \mu_i +\bclr{1-\sum_{i=1}^3 \alpha_i \mu_i}\mu_4},  \\
&\hspace{10mm}		Z_{\sum_{i=1,2} \alpha_i \nu_i+ \alpha_3 \mu_3 +\bclr{1-\sum_{i=1}^3 \alpha_i \mu_i}\mu_4},
	Z_{\sum_{i=1,3} \alpha_i \nu_i+ \alpha_2 \mu_2 +\bclr{1-\sum_{i=1}^3 \alpha_i \mu_i}\mu_4}, \\
&\hspace{10mm}	Z_{\sum_{i=2,3} \alpha_i \nu_i+ \alpha_1 \mu_1 +\bclr{1-\sum_{i=1}^3 \alpha_i \mu_i}\mu_4},
Z_{\sum_{i=1}^{3} \alpha_i \nu_i +\bclr{1-\sum_{i=1}^3 \alpha_i \mu_i}\mu_4} },
}
where again the reuse of notation will not cause a problem.
Similar to before, let 
 $(U_i)_{i\geq1}$, $(U_i\s{j})_{i\geq1}, (V_i\s{j})_{i\geq1}$, $j=1,2,3$ be independent i.i.d.\ sequences distributed as $\mu_4, \mu_j,\nu_j$, and define an i.i.d.\ sequence 
\be{
(X_i, X_i\s{1}, X_i\s{2}, X_i\s{1,2}, X_i\s{3}, X_i\s{1,3}, X_i\s{2,3}, X_i\s{1,2,3})_{i\geq1}
} 
with distribution
\ba{
&\IP\bclr{X_i= X_i\s{1}= X_i\s{2}= X_i\s{1,2} = X_i\s{3} \\
	&\hspace{20mm}=X_i\s{1,3}= X_i\s{2,3}= X_i\s{1,2,3}=U_i | U_i, (U_i\s{j},V_i\s{j})_{j=1}^3} 
 =1-\alpha_1-\alpha_2-\alpha_3,  \\
&\IP\bclr{X_i= X_i\s{2}=X_i\s{3}= X_i\s{2,3}=U_i\s{1}, \\
& \hspace{20mm}X_i\s{1}= X_i\s{1,2}=X_i\s{1,3}= X_i\s{1,2,3}=V_i\s{1} | U_i, (U_i\s{j},V_i\s{j})_{j=1}^3}=\alpha_1,  \\
&\IP\bbclr{X_i= X_i\s{1}= X_i\s{3}= X_i\s{1,3}=U_i\s{2}, \\
	&\hspace{20mm} X_i\s{2}= X_i\s{1,2}= X_i\s{2,3}= X_i\s{1,2,3}=V_i\s{2} |U_i, (U_i\s{j},V_i\s{j})_{j=1}^3}  =\alpha_2, \\
&\IP\bclr{X_i= X_i\s{1}= X_i\s{2}= X_i\s{1,2}=U_i\s{3}, \\
	&\hspace{20mm} X_i\s{3}= X_i\s{1,3}= X_i\s{2,3}= X_i\s{1,2,3}=V_i\s{3} |U_i, (U_i\s{j},V_i\s{j})_{j=1}^3}  =\alpha_3.
} 
Given $(X_i, X_i\s{1}, X_i\s{2}, X_i\s{1,2}, X_i\s{3}, X_i\s{1,3}, X_i\s{2,3}, X_i\s{1,2,3})_{i\geq1}$, define the measures for $n\geq1$,
\bes{
\nu_{n}\t{8} 
	&=\frac{1}{n+\theta} \sum_{i=1}^{n} \delta_{(X_i, X_i\s{1}, X_i\s{2}, X_i\s{1,2}, X_i\s{3}, X_i\s{1,3}, X_i\s{2,3}, X_i\s{1,2,3})} + \frac{\theta}{n+\theta} \pi^{*8},
}
where $\pi^{*8}(B)=\pi(\{x: (x,\ldots, x)\in B\})$.
Now, given the above, let 
\be{
(\xi_{n j},  \xi_{n j}\s{1},\xi_{n j}\s{2},\xi_{n j}\s{1,2},\xi_{n j}\s{3},  \xi_{n j}\s{1,3},\xi_{n j}\s{2,3},\xi_{n j}\s{1,2,3})_{j\geq 1}
}
be conditionally i.i.d.\ $\nu_{n}\t{8}$-distributed, and 
$(P_{n j})_{j\geq1}\sim\PD(n+\theta)$ independent of the variables above. Finally, for $i\in\{1,2,3\}$ and $\ell\in \{1,2,3\}$ with $i<\ell$, define 
\ba{
Z(t)&=\sum_{j=1}^\infty P_{L_t j} \delta_{\xi_{L_t j}}, &Z\s{i}(t)=\sum_{j=1}^\infty P_{L_t j} \delta_{\xi_{L_t j}\s{i}}, \\
Z\s{i,\ell}(t)&=\sum_{j=1}^\infty P_{L_t j} \delta_{\xi_{L_t j}\s{i,\ell}}, &Z\s{1,2,3}(t)=\sum_{j=1}^\infty P_{L_t j} \delta_{\xi_{L_t j}\s{1,2,3}}.
}
Now, for each $n\geq1$ and $i=1,2,3$, define $B_{ni}=\{j:\xi_{n j}\s{i}\not= \xi_{n j}\}$, so that we can write, for $i\in\{1,2,3\}$ and $\ell\in \{1,2,3\}$ with $i<\ell$,
\ba{
Z\s{i}(t)&=Z(t)+\sum_{j\in B_{L_t i}} P_{L_t j} (\delta_{\xi_{L_t j}\s{i}}-\delta_{\xi_{L_t j}}), \\
Z\s{i,\ell}(t)&=Z(t)+\sum_{j\in B_{L_t i}} P_{L_t j} (\delta_{\xi_{L_t j}\s{i}}-\delta_{\xi_{L_t j}})+\sum_{j\in B_{L_t \ell}} P_{L_t j} (\delta_{\xi_{L_t j}\s{\ell}}-\delta_{\xi_{L_t j}}),\\
Z\s{1,2,3}(t)&=Z(t)+\sum_{i=1}^3\sum_{j\in B_{L_t i}} P_{L_t j} (\delta_{\xi_{L_t j}\s{i}}-\delta_{\xi_{L_t j}}).
}
Using this decomposition, we have
\bes{
&\bcls{H(Z\s{1,2,3}(t))-H(Z\s{1,3}(t))-H(Z\s{2,3}(t))+H(Z\s{3}(t))} \\
&\hspace{12mm}-\bcls{H(Z\s{1,2}(t))-H(Z\s{1}(t))-H(Z\s{2}(t))+H(Z(t))} \\
		&\qquad=\bbclr{\prod_{i=1}^3 \Delta_{\sum_{j\in B_{L_t i}} P_{L_t j} (\delta_{\xi_{L_t j}\s{i}}-\delta_{\xi_{L_t j}})}}H(Z(t)),
}
and so Lemma~\ref{lem:domainhbds} implies
\bes{
&\Big|\bcls{H(Z\s{1,2,3}(t))-H(Z\s{1,3}(t))-H(Z\s{2,3}(t))+H(Z\s{3}(t))} \\
&\hspace{12mm}-\bcls{H(Z\s{1,2}(t))-H(Z\s{1}(t))-H(Z\s{2}(t))+H(Z(t))}\Big| \\
&\hspace{4cm}\leq 8 L_3(H) \prod_{i=1}^3 \bbclr{ \sum_{j\in B_{L_t i}} P_{L_t j}  }.
}
%Thus, according to Lemma~3 of 
%\cite{Gan2017},~\eq{eq:888} is bounded by 
%\bes{
%\abs{h}_{2,1} \bbbcls{\sum_{j\in B_{L_t 0}} &P_{L_t j} \sum_{i=1}^k \babs{\phi_i(\zeta_{L_t j})-\phi_i(\xi_{L_t j})}} \\
%	&\times \bbbcls{\sum_{j\in B_{L_t 1}} P_{L_t j} \sum_{i=1}^k \babs{\phi_i(x)-\phi_i(\xi_{L_t j})}}\bbbcls{\sum_{k\in B_{L_t 2}} P_{L_t k}\sum_{i=1}^k \babs{\phi_i(y)-\phi_i(\xi_{L_t j})}} \\
%	&\hspace{8mm}\leq 8 \Lip_3(H)\prod_{k=0}^2 \bbcls{\sum_{j\in B_{L_t k}} P_{L_t k}}.
%}
Using this inequality and the definition of~$F$, we find
\ban{
\bbabs{ \bbclr{\prod_{i=1}^3&  \Delta_{\alpha_i( \nu_i-\mu_i)}}F\bbclr{\sum_{i=1}^3 \alpha_i\mu_i+(1-\sum_{i=1}^3 \alpha_i)\mu_4}}  \notag \\
%	&\leq \int_0^\infty\IE \babs{H(Z\s{1,2}(t))-H(Z\s{1}(t))-H(Z\s{2}(t))+H(Z_\mu(t))} dt \\
	&\leq  8 \Lip_3(H)\int_0^\infty\IE \bbbcls{\prod_{i=1}^3 \bbclr{\sum_{j\in B_{L_t i}} P_{L_t j}}} dt \label{eq:deriv2wass} \\
%	&= 2 \Lip_2(H) \int_0^\infty\IE \bbcls{\sum_{j,k=1}^\infty \Ind[j\in B_{L_t 1}]P_{L_t j} \Ind[k\in B_{L_t 2}] P_{L_t k}} dt \\
	&= 8 \Lip_3(H) \int_0^\infty\IE \bbcls{ \sum_{i,j,k=1}^\infty \IP(i\in B_{L_t 1}, j\in B_{L_t 2}, k\in B_{L_t 3} |L_t) \IE [P_{L_t i} P_{L_t j} P_{L_t k}| L_t] } dt, \notag
	}
where we use conditional independence in the last line.
Now, we have
\bes{
\IP&\bclr{i\in B_{L_t 1}, j\in B_{L_t 2}, k\in B_{L_t 3} |L_t,(X_i, X_i\s{1}, X_i\s{2}, X_i\s{1,2}, X_i\s{3}, X_i\s{1,3}, X_i\s{2,3}, X_i\s{1,2,3})_{i\geq1}}  \\
	&=\Ind[i,j,k \,\mbox{ distinct}](L_t+\theta)^{-3}\sum_{\ell, m, r=1}^{L_t} \Ind[X_\ell\s{1}\not=X_\ell, X_m\s{2}\not= X_m, X_r\s{3}\not=X_r],
}
and averaging out $(X_i, X_i\s{1}, X_i\s{2}, X_i\s{1,2}, X_i\s{3}, X_i\s{1,3}, X_i\s{2,3}, X_i\s{1,2,3})_{i\geq1}$, we have
\bes{
\IP&\bclr{i\in B_{L_t 1}, j\in B_{L_t 2}, k\in B_{L_t 3} |L_t }  \\
	&=\Ind[i,j,k \,\mbox{ distinct}](L_t+\theta)^{-3}\sum_{\ell, m, r=1}^{L_t} 
	\IP\bclr{X_\ell\s{1}\not=X_\ell, X_m\s{2}\not= X_m, X_r\s{3}\not=X_r} \\
		&\leq\Ind[i,j,k \,\mbox{ distinct}](L_t+\theta)^{-3}\sum_{\ell, m,r =1}^{L_t} \alpha_1\alpha_2 \alpha_3 \Ind[\ell,m,r \, \mbox{ distinct}]\\
		&=\alpha_1\alpha_2\alpha_3 \Ind[i,j,k \,\mbox{ distinct}]\frac{L_t(L_t-1)(L_t-2)}{(L_t+\theta)^{3}}.
}
 Applying the last two displays in~\eq{eq:deriv2wass}, we have 
\bes{
\bbabs{ \bbclr{\prod_{i=1}^3&  \Delta_{\alpha_i( \nu_i-\mu_i)}}F\bbclr{\sum_{i=1}^3 \alpha_i\mu_i+(1-\sum_{i=1}^3 \alpha_i)\mu_4}}  \\
	&\leq 8\alpha_1\alpha_2\alpha_3 \Lip_3(H)  \int_0^\infty\IE \bbbcls{ \frac{L_t(L_t-1)(L_t-2)}{(L_t+\theta)^3}  \IE\bbcls{\sum_{i,j,k \mbox{ \tiny{distinct}}} P_{L_t i} P_{L_t j} P_{L_t k}| L_t} } dt \\
	&= 8\alpha_1\alpha_2\alpha_3 \Lip_3(H) \int_0^\infty\IE\bbbcls{\frac{L_t(L_t-1)(L_t-2)}{(L_t+\theta)^3} \frac{(L_t+\theta)^2}{(L_t+\theta+1)(L_t+\theta+2)}}dt 	 \\
	&=8\alpha_1\alpha_2\alpha_3 \Lip_3(H) \sum_{n\geq1}\bbbcls{\frac{n(n-1)(n-2)}{(n+\theta)(n+\theta+1)(n+\theta+2)}} \IE Y_n \\
		&=16\alpha_1\alpha_2\alpha_3 \Lip_3(H) \sum_{n\geq1}\bbbcls{\frac{(n-1)(n-2)}{(n+\theta)(n+\theta+1)(n+\theta+2)(n+\theta-1)}} \\
	&= \frac{16}{3(2+\theta)}\alpha_1\alpha_2\alpha_3 \Lip_3(H),
}	
where in the first equality we use that if $(\tilde P_j)_{j\geq1}\sim \PD(\tilde \theta)$, then
\be{
\IE \sum_{i,j,k \mbox{ \tiny{distinct}}}\tilde P_i \tilde P_j \tilde P_k= \frac{\tilde \theta^2}{(\tilde \theta+1)(\tilde \theta +2)},
}
which again follows from Kingman's paintbox representation, which says the sum on the left hand side is the chance of getting the partition~$(1,1,1)$ when sampling three times from~$\PD(\tilde \theta)$; this is the same as starting a table on the second and third step of the Chinese Restaurant Process; see \cite[Section~3.1]{Pitman2006}.
This establishes~\eq{eq:gddiffbd3}.

\smallskip
\noindent{\bf Case 2:  $H\in \cH_2$.}
If $H \in \cH_2=\cD_2$, then it can be written in the form 
$H=\angle{\phi,\mu^k}$ with $\phi\in B(E^k)$. Now, using Lemma~\ref{lem:derivsform}, the action of the generator~\eq{eq:dpgen} on~$\cH_2$ can be rewritten as 
\ban{
\cA H(\mu) = \sum_{1 \leq i < j \leq k} [\angle{\Phi_{ij}^{(k)}\phi, \mu^{k-1}} - \angle{\phi, \mu^k}] + \frac{\theta}{2} \sum_{1 \leq i \leq k}[\angle{\phi, \mu^{i-1} \pi \mu^{k-i}} - \angle{\phi,\mu^k}],\label{eq:dualgen}
}
where $\Phi_{ij}^{(k)}\phi(x_1, \ldots, x_{k-1}) = \phi(x_1, \ldots, x_{j-1}, x_i, x_j, \ldots, x_{k-1})$. Using this form of the generator, 
we follow \cite[Section~3]{Ethier1993}; in turn following \cite{Dawson1982}; and define a ``dual" 
Markov jump process $(Y_\phi(t))_{t\geq0}$ on $\cup_{j\geq0} \rC(E^j)$ with
$Y(0)=\phi$.
Informally, given $Y(t)=\psi \in \rC(E^j)$ for some $j\geq 1$, then for each $\ell=1,\ldots,j$,  
the process transitions at rate $\theta/2$ to functions in $\rC(E^{j-1})$ of the form
\be{
(x_1,\ldots,x_{j-1}) \mapsto \int_E \psi(x_1,\ldots,x_{\ell-1},x,x_{\ell},\ldots, x_{j-1}) \pi(dx),
}
and for each $1\leq i<\ell\leq j$, it transitions at rate~$1$ to functions in $\rC(E^{j-1})$ of the form
\be{
(x_1,\ldots,x_{j-1}) \mapsto \psi(x_1,\ldots,x_{\ell-1},x_i,x_{\ell},\ldots, x_{j-1}).
}
The process absorbs once it reaches a constant function (in $\rC(E^0)$).
Then, with $H$ as above, \cite[Theorem~3.1]{Ethier1993} implies
\ben{\label{eq:dualform}
\IE \bcls{H(Z_\mu(t))}= \IE \bcls{\angle{Y_\phi(t), \mu^{k-M_k(t)}}}\Ind[M_k(t)<k],
}
 where $M_k(t)$ is the number of transitions of the process $(Y_\phi(\cdot))$ up to time $t$,
 and the indicator is because $H=\wt H$ is centered.
Now, to show~\eq{eq:gddiffbd1}, use~\eq{eq:dualform} to write
\besn{\label{eq:333}
\babs{\Delta_{\alpha_1(\nu_1-\mu_1)} &F\bclr{\alpha_1\mu_1+(1-\alpha_1)\mu_4}} \\
	&=\bbabs{\int_0^\infty \IE \Delta_{\alpha_1(\nu_1-\mu_1)}
	\bangle{Y_\phi(t), (\alpha_1\mu_1+(1-\alpha_1)\mu_4)^{k-M_k(t)}}dt}.
}
Let $0<\tau_1<\cdots<\tau_k$ be the times of the transitions of the process $(Y_\phi(\cdot))$ and $\tau_0:=0$. Then continuing from~\eq{eq:333}, using that for $t\geq \tau_k$, 
\be{
 \Delta_{\alpha_1(\nu_1-\mu_1)}
	\bangle{Y_\phi(t), (\alpha_1\mu_1+(1-\alpha_1)\mu_4)^{0}}=0, 
}
we find
\besn{\label{eq:gdexp11}
\bbabs{&\Delta_{\alpha_1(\nu_1-\mu_1)} F\bclr{\alpha_1\mu_1+(1-\alpha_1)\mu_4}} \\
	&=\bbabs{ \sum_{i=1}^k \IE\bbcls{ \int_{\tau_{i-1}}^{\tau_i}  \Delta_{\alpha_1(\nu_1-\mu_1)}
	\bangle{Y_\phi(\tau_{i-1}), (\alpha_1\mu_1+(1-\alpha_1)\mu_4)^{k-i+1}}dt}} \\
	&\leq  \norm{\phi}_\infty  \sum_{i=1}^k  \bnorm{(\alpha_1\nu_1+(1-\alpha_1)\mu_4)^{k-i+1}-(\alpha_1\mu_1+(1-\alpha_1)\mu_4)^{k-i+1}}\IE\cls{\tau_i-\tau_{i-1}}, 
}
where we have used the triangle inequality and that $\norm{Y_\phi(t)}_\infty\leq \norm{\phi}_{\infty}$.
Since the inter-jump time $(\tau_i-\tau_{i-1})$ is exponential with rate $(k-i+1)(k-i+\theta)/2$, and, from 
\be{
\bnorm{(\alpha_1\nu_1+(1-\alpha_1)\mu_4)^{k-i+1}-(\alpha_1\mu_1+(1-\alpha_1)\mu_4)^{k-i+1}}= (k-i+1) \alpha_1 \norm{\nu_1-\mu_1} + \lito(\alpha_1),
}
and $\norm{\nu_1-\mu_1}\leq 2$, we find from~\eq{eq:gdexp11} that
\ba{
\bbabs{\Delta_{\alpha_1(\nu_1-\mu_1)} &F\bclr{\alpha_1\mu_1+(1-\alpha_1)\mu_4}} \leq 4 \alpha_1 \norm{\phi}_\infty  \sum_{i=1}^k \frac{1}{k-i+\theta} + \lito(\alpha_1)
	\leq 4 \alpha_1 \norm{\phi}_\infty k\theta^{-1} + \lito(\alpha_1),
}
as desired.

For~\eq{eq:gddiffbd2}, the same arguments leading to~\eq{eq:gdexp11} give
\ban{
&\babs{\Delta_{\alpha_1(\nu_1-\mu_1)} \Delta_{\alpha_2(\nu_2-\mu_2)} F\bclr{\alpha_1\mu_1+\alpha_2\mu_2+(1-\alpha_1-\alpha_2)\mu_4}}\label{eq:gdexp22} \\
	&\leq  \norm{\phi}_\infty  \sum_{i=1}^k  
	\bnorm{
	\Delta_{\alpha_1(\nu_1-\mu_1)}\Delta_{\alpha_2(\nu_2-\mu_2)}G_{k-i+1}(\alpha_1\mu_1+\alpha_2\mu_2+(1-\alpha_1-\alpha_2)\mu_4)} \IE\cls{\tau_i-\tau_{i-1}}, 
	\notag
}
where $G_j:M_1\to M_1(E^j)$ is defined by $G_j(\mu)=\mu^j$.
Since 
\bes{
\bnorm{
\Delta_{\alpha_1(\nu_1-\mu_1)}\Delta_{\alpha_2(\nu_2-\mu_2)}G_{k-i+1}&(\alpha_1\mu_1+\alpha_2\mu_2+(1-\alpha_1-\alpha_2)\mu_4)} \\
	&\leq\alpha_1\alpha_2\bbclr{(k-i+1)(k-i)\bnorm{(\nu_1-\mu_1)\times(\nu_2-\mu_2)}+\lito(1)}, 
}
and $\bnorm{(\nu_1-\mu_1)\times(\nu_2-\mu_2)}\leq 4$, we find the desired inequality
\bes{
\babs{\Delta_{\alpha_1(\nu_1-\mu_1)} \Delta_{\alpha_2(\nu_2-\mu_2)} F\bclr{\alpha_1\mu_1&+\alpha_2\mu_2+(1-\alpha_1-\alpha_2)\mu_4}} \\
	&\leq4 \alpha_1\alpha_2 \bbbclr{2\norm{\phi}_\infty \sum_{i=1}^{k-1} \frac{k-i}{k-i+\theta} +\lito(1)}\\
	&\leq 4\alpha_1\alpha_2\bbbclr{  \norm{\phi}_\infty \frac{k(k-1)}{\theta+1}+\lito(1)}.
}
For~\eq{eq:gddiffbd3}, following the same arguments leading to~\eq{eq:gdexp11} and~\eq{eq:gdexp22}, noting that
\bes{
\bbnorm{
&\Delta_{\alpha_1(\nu_1-\mu_1)}\Delta_{\alpha_2(\nu_2-\mu_2)}\Delta_{\alpha_3(\nu_3-\mu_3)}G_{k-i+1}\bbclr{\sum_{i=1}^3 \alpha_i\mu_i+(1-\sum_{i=1}^3 \alpha_i)\mu_4} }\\
	&\quad\leq\alpha_1\alpha_2\alpha_3\bbclr{(k-i+1)(k-i)(k-i-1)\bnorm{(\nu_1-\mu_1)\times(\nu_2-\mu_2)\times(\nu_3-\mu_3)}+\lito(1)}, 
}
and $\bnorm{(\nu_1-\mu_1)\times(\nu_2-\mu_2)\times(\nu_3-\mu_3)}\leq 8$,
we have
\bes{
\bbabs{\Delta_{\alpha_1(\nu_1-\mu_1)} &\Delta_{\alpha_2(\nu_2-\mu_2)} \Delta_{\alpha_3(\nu_3-\mu_3)}F\bbclr{\sum_{i=1}^3 \alpha_i\mu_i+(1-\sum_{i=1}^3 \alpha_i)\mu_4}}  \\
	&\leq16\alpha_1\alpha_2\alpha_3\bbbclr{\norm{\phi}_\infty \sum_{i=1}^{k-2} \frac{(k-i)(k-i-1)}{k-i+\theta} +\lito(1)} \\
		&\leq 16\alpha_1\alpha_2\alpha_3\bbbclr{\norm{\phi}_\infty \frac{k(k-1)(k-2)}{3(\theta+2)} +\lito(1)},
}
which is~\eq{eq:gddiffbd3}.
\end{proof}
\subsection{Proof of Theorem~\ref{THM3}}

Before we prove Theorem~\ref{THM3} we record a lemma for convenience.

\begin{lemma}\label{exlem}
Fix $H \in \cH_1\cup \cH_2$ and let $F := F_H$ denote the solution to the Stein equation~\eq{eq:pdstnsol} associated with $H$. Set $g(t) := F(W + t(W'-W))$. Then
\ba{
g'(t) &=\int \partial_x F(W + t(W' - W)) (W'-W)(dx),\\
g''(t) &= \int \partial_{xy} F(W + t(W' - W)) (W'-W)^2(dx,dy).
}
If $H\in\cH_2$, then
\be{
g'''(t) =\int \partial_{xyz} F(W + t(W' - W)) (W'-W)^3(dx,dy,dz).
}
\end{lemma}
\begin{proof}
Suppose $H \in \cH_1$, then by Theorem~\ref{thm:FVstnbds}, $F\in\cD_1$, so $F(\mu)=f(\angle{\bm\psi,\mu})$ for some $f\in\rC^2(\IR)$ and $\bm\psi=(\psi_1,\ldots,\psi_k)\in\rC(E^k)$, where we write $\angle{\bm\psi,\mu}=(\angle{\psi_1,\mu},\ldots,\angle{\psi_k,\mu})$.  Using
the chain rule, the fact that $\int (W'-W)(dx)=0$, and~\eq{eq:d1deriv1} from Lemma~\ref{lem:derivsform}, we find
\ban{
g'(t) &=\sum_{i=1}^k f_i\bclr{\angle{\bm\psi,W + t(W' - W)}}\angle{\psi_i,  W'-W} \label{eq:gp1} \\
%	&= \int \sum_{i=1}^k f_i\bclr{\angle{\psi_1,W + t(W' - W)},\ldots,\angle{\psi_k,W + t(W' - W)}}(\psi_i(x)(W'-W)(dx)\\
	&= \int \sum_{i=1}^k\bbclr{ f_i\bclr{\angle{\bm\psi,W + t(W' - W)}}\angle{\psi_i, \delta_x}}(W'-W)(dx)\notag\\
	&= \int \sum_{i=1}^k\bbklr{f_i\bclr{\angle{\bm\psi,W + t(W' - W)}}\angle{\psi_i, \delta_x - (W + t(W'-W))}}(W'-W)(dx)\notag\\
	&=\int \partial_x F(W + t(W' - W)) (W'-W)(dx).
}
Similarly, using~\eq{eq:d1deriv2} from Lemma~\ref{lem:derivsform}, and that $\int (W'-W)(dx)=0$, it is straightforward to see
\ba{
\int \partial_{xy} F(W + t(&W' - W)) (W'-W)^2(dx,dy) \\
&= \int \sum_{i,j=1}^k  f_{ij}\bclr{\angle{\bm\psi,W + t(W' - W)}} \psi_i(x)\psi_j(y)(W'-W)^2(dx,dy),
}
which is clearly the same as $g''(t)$ obtained by differentiating~\eq{eq:gp1}.

Similar to above, if $H \in \cH_2$ then Theorem~\ref{thm:FVstnbds} implies that $F\in\cD_2$, so there exists a $\psi \in \rC(\IR^k)$ such that $F(\mu) = \angle{\psi,\mu^k}$. Then, using that $\int (W'-W)(dx)=0$ and the expression~\eq{eq:d2deriv1}, we find
\ban{
&g'(t) \notag \\
&= \sum_{i=0}^{k-1} \angle{\psi, (W + t(W'-W))^i (W'-W)(W + t(W'-W))^{k-i-1}} \label{eq:gp2}\\
%	&=\int \sum_{i=1}^{k-1} \angle{ \psi, (W + t(W' - W))^i (W'-W) (W + t(W' - W))^{k-i-1}}\\
	&=\int \sum_{i=1}^{k-1} \bangle{ \psi, (W + t(W' - W))^i \delta_x (W + t(W' - W))^{k-i-1}} (W' - W) (dx)\notag\\
	&=\int \sum_{i=1}^{k-1} \bangle{ \psi, (W + t(W' - W))^i \bclr{\delta_x - (W + t(W'-W)) (W + t(W' - W)}^{k-i-1}} (W' - W) (dx) \notag\\
	&=\int \partial_x F(W + t(W' - W)) (W'-W)(dx). \notag
}
For the second derivative, use~\eq{eq:d2deriv2} and that $\int (W'-W)(dx)=0$ to find
\ba{
\int \partial_{xy} F(W + t(&W' - W)) (W'-W)^2(dx,dy) \\
	&=\sum_{i\not=j}^k \int \angle{\psi_{xy}\s{i,j}, \bclr{W + t(W' - W)}^{k-2}}(W'-W)^2(dx,dy),
}
which is the same as differentiating~\eq{eq:gp2}.

Finally, the expression for the third derivative follows similarly from~\eq{eq:d2deriv3} and 
 that $\int (W'-W)(dx)=0$; in particular, note
 \bes{
\int \partial_{xyz} F(W + t(W' - &W)) (W'-W)^3(dx,dy) \\
	&=\mathop{\sum_{i,j,\ell}}_{\mathrm{distinct}}^k \int \angle{\psi_{xyz}\s{i,j,\ell}, 
		\bclr{W + t(W' - W)}^{k-2}}(W'-W)^3(dx,dy,dz). \qedhere
}
\end{proof}

We are now ready to prove Theorem~\ref{THM3}.
\begin{proof}[Proof of Theorem~\ref{THM3}]
Writing  $F:=F_H$, set  $g(t) := F( W + t(W' - W))$ so that $g(1) = F(W')$ and $g(0) = F(W)$. 

By Lemma~\ref{exlem}, $g$ is twice differentiable, so the usual elementary calculation gives
\ba{
g(1)-g(0)=
\begin{cases}
 g'(0) + \frac{1}{2}g''(0) + \int_0^1 (1-t) (g''(t)-g''(0))dt, & H\in \cH_1, \\
 g'(0) + \frac{1}{2}g''(0) + \int_0^1\int_0^t (1-t) g'''(s) ds dt, & H\in \cH_2,
\end{cases}
}
and thus from Lemma~\ref{exlem},
\besn{\label{part1}
F(W')- F(W) &=\int \partial_x F(W) (W'-W)(dx) \\ %\int_0^1\rmd F(W)[\delta_x](W'(dx) - W(dx)) \\
	&\qquad+\frac{1}{2}\int\partial_{xy} F(W) (W'-W)^2(dx,dy) %\frac{1}{2} \int_0^1\int_0^1D^2_GF(W)[\delta_x, \delta_y](W'(dx) - W(dx))(W'(dy) - W(dy)) 
	+ R_3, 
}
where we define
\ba{
R_3=
\begin{cases}
\int_0^1 (1-t) \int \bclr{\partial_{xy} F(W + t(W' - W))- \partial_{xy}F(W)}(W'-W)^2(dx,dy) dt,
	& H\in\cH_1 \\
\int_0^1\int_0^t  (1-t) \int \partial_{xyz} F(W + s(W' - W))(W'-W)^3(dx,dy,dz) ds dt,
	& H\in\cH_2.
	\end{cases}
}
Using basic facts about random measures;
% Campbell's formula for measures; see, for example, \cite[(13.1)]{Last2018}
 the linearity condition~\eq{7} and direct calculation (noting $\int R(dx)=0$ by integrating~\eq{7}), we find
\besn{\label{eq:lincondeq}
\IE\bbbcls{\int \partial_x F(W) (W'-W)(dx)\Big| W} 
	&=\int \partial_x F(W) \IE\bcls{(W'-W) | W}(dx)\\
	&=-\lambda \theta \int \partial_x F(W) (W-\pi)(dx) + \int \partial_x F(W) R(dx).
% \lambda \theta(\ind_{x = U} - W(dx)) + R.
}
Now, taking expectation in~\eq{part1}, noting~\eq{eq:lincondeq} and that $\law(W')=\law(W)$, we find
\bes{
 \frac{\theta}{2} \IE \int \partial_x F(W) (W-\pi)(dx) &=\frac{1}{4\lambda} \IE\bbbclc{\IE\bbbcls{ \int \partial_{xy} F(W) (W'-W)^2(dx,dy) \Big| W } } \\
 &\hspace{10mm} +  \frac{1}{2\lambda}\int \partial_x F(W) R(dx)+\frac{\IE R_3}{2\lambda}.
 }
Plugging into~\eq{eq:dpgen}, noting~\eq{eq:gddrift}, we have
\ba{
\IE \cA F(W) 
	&= \frac{1}{2}\IE \iint \partial_{xy} F(W) \bbclr{ W(dx)\bclr{\delta_x(dy)-W(dy)}-\frac{1}{2\lambda} \IE\bcls{(W'-W)^2|W}(dx,dy)}  \\
	& \hspace{10mm}- \frac{1}{2\lambda} \int \partial_x F(W) R(dx)-\frac{\IE R_3}{2\lambda}.
}
The result easily follows by taking absolute value, using the triangle inequality, and applying the Stein bounds of Theorem~\ref{thm:FVstnbds}; for example,
\ba{
\abs{R_3} \leq \begin{cases} 
\frac{D_3(H;\theta)}{6} \|W' - W \| \|(W'-W)^2\|, & H \in \cH_1,\\
\frac{D_3(H;\theta)}{6} \| (W' - W)^3\|, & H \in \cH_2, 
\end{cases}
}
and $\babs{ \int \partial_x F(W) R(dx)}\leq \norm{\partial_x F}_{\infty} \norm{R}\leq D_1(H;\theta) \norm{R}$.
\end{proof}

\section{Proofs of Wright-Fisher results}\label{sec:wf}

Here we prove Theorems~\ref{THM:WF} and~\ref{thm:KN2mombd}.

\subsection{Proof of Wright-Fisher approximation Theorem~\ref{THM:WF}}
Write the stationary random variable $W=W_N= \frac{1}{N} \sum_{i=1}^{K} N_i \delta_{x_i}$,
where the $x_i$ are the random type labels, $N_i$ is the number of type $x_i$, and $K=K_N$ is the number of types. 
We define $W'=W'_N$ to be one step in the Markov chain from $W$, so $\law(W')=\law(W)$ and Theorem~\ref{THM3} can be applied. To more precisely define~$W'$, we need some intermediate quantities. 
First, given~$W$, let $(M_i)_{i=1}^K\sim\MN\bclr{N, N^{-1}(N_1,\ldots, N_K)}$ represent the number of offspring of each type in one step of the chain from~$W$. Now, given $W$ and $(M_i)_{i=1}^K$, let $B_1,\ldots, B_K$ be conditionally
independent with $B_i\sim\Bi\bclr{M_i, p(x_i)}$ represent the number of mutations in the offspring of each type. Finally, given the variables above, for $1\leq i\leq K$ and $1\leq j \leq B_i$, let $x_{ij}\sim \kappa_{x_i}$ be independent mutation types.
Now we can define
\ban{
W'
	&=\frac{1}{N} \sum_{i=1}^K (M_i-B_i) \delta_{x_i} + \frac{1}{N} \sum_{i=1}^K\sum_{j=1}^{B_i} \delta_{x_{ij}}  \label{eq:wfwo} \\
  	&=\frac{1}{N} \sum_{i=1}^K M_i \delta_{x_i} + \frac{1}{N} \sum_{i=1}^K\sum_{j=1}^{B_i}\clr{\delta_{x_{ij}} -\delta_{x_i}}, \label{eq:wfwp}
}
and it is clear that $W'$ is distributed as one step in the chain from $W$. The next lemma shows we can apply Theorem~\ref{THM3} with~$\lambda=1/(2N)$ and~$R$ read from below.
\begin{lemma}\label{lem:linwf}
For $(W,W')$ defined above and $\lambda=1/(2N)$, 
\bes{
\IE\bcls{W' - W | W} 
		&= -\frac{\theta}{2N} (W-\pi) 
			+ \frac{1}{N} \sum_{i=1}^K N_i \bclr{\tsfrac{\theta}{2N}-p(x_i)} \clr{\delta_{x_i} - \kappa_{x_i}}	
			+\frac{1}{N}\sum_{i=1}^K N_i \tsfrac{\theta}{2N} \bclr{\kappa_{x_i}-\pi}\\
		&=: -\theta \lambda (W-\pi) + R.
	}
\end{lemma}
\begin{proof}
Starting from~\eq{eq:wfwp}, intermediately conditioning on $(M_i)_{i=1}^K$, and  using formulas for the mean of  binomial distributions, we easily find
\bes{
\IE\bcls{W'  | W} &= \frac{1}{N} \sum_{i=1}^K N_i \delta_{x_i} + \frac{1}{N} \sum_{i=1}^K N_i p(x_i)\clr{\kappa_{x_i} -\delta_{x_i}}, \\
%	&=W-\frac{\theta}{2N} (W-\pi) + \frac{1}{N} \sum_{i=1}^K N_i \bclr{\tsfrac{\theta}{2N}-p(x_i)} \delta_{x_i} 
%			+\frac{1}{N}\sum_{i=1}^K N_i \bclr{\kappa_{x_i}-\pi},
}
and rearranging gives the lemma.
\end{proof}

We now write some lemmas used to bound~$A_1,A_2,A_3$ from Theorem~\ref{THM3}.
\begin{lemma}\label{lem:A1wf} 
With the definitions above, 
\bes{
\frac{1}{\lambda} \IE \norm{R} \leq   4N \sup_{x\in  E} \babs{ p(x) - \tsfrac{\theta}{2N}} + \theta \sup_{x\in E} \norm{\kappa_x - \pi}. %% WF A1
}			
\end{lemma}
\begin{proof}
Use the triangle inequality to find
\ba{
\frac{1}{\lambda} \IE \norm{R}
	&= 2N \, \IE \bbbnorm{ \frac{1}{N} \sum_{i=1}^K N_i \bclr{\tsfrac{\theta}{2N}-p(x_i)} \clr{\delta_{x_i} - \kappa_{x_i}}	
		+\frac{1}{N}\sum_{i=1}^K N_i \tsfrac{\theta}{2N} \bclr{\kappa_{x_i}-\pi}} \\
	&\leq 2N \, \IE \bbbcls{ \frac{1}{N} \sum_{i=1}^K N_i \babs{\tsfrac{\theta}{2N}-p(x_i)} \norm{\delta_{x_i} - \kappa_{x_i}}	
		+\frac{1}{N}\sum_{i=1}^K N_i \tsfrac{\theta}{2N} \norm{\kappa_{x_i}-\pi}},
}
and the result now easily follows.
\end{proof}

The next lemma is used to bound~$A_2$ from Theorem~\ref{THM3}.
\begin{lemma}\label{lem:A2wf}
With the definitions above,
\ba{
 \IE \bnorm{W^{*2}-W^2-\tsfrac{1}{2\lambda}\IE[(W'-W)^2|W]} 
  	\leq 4 \norm{p}_\infty \bclr{N \norm{p}_\infty + 3  } %%%%WF A2
}
\end{lemma}
\begin{proof}
First note $\lambda=1/(2N)$, and from~\eq{eq:wfwp},
\ban{
\frac{(W'-W)^2 }{2\lambda} &= \frac{1}{N} \sum_{i,j=1}^K (M_i-N_i) (M_j-N_j) \delta_{x_i}  \delta_{x_j} 
			+  \frac{1}{N} \sum_{i,j=1}^K\sum_{k=1}^{B_i}\sum_{l=1}^{B_j}\clr{\delta_{x_{ik}} -\delta_{x_i}}\clr{\delta_{x_{jl}} -\delta_{x_j}}  \label{eq:tee1}\\
		&	 \qquad \quad		+ \frac{1}{N}\sum_{i,j=1}^K \sum_{k=1}^{B_i} (M_j-N_j) \bclr{  \delta_{x_j}\clr{\delta_{x_{ik}} -\delta_{x_i}}+\clr{\delta_{x_{ik}} -\delta_{x_i}}\delta_{x_j}}. \label{eq:tee2}
}
Using the formula for the covariance of a multinomial distribution, the conditional expectation of the first term of~\eq{eq:tee1}
becomes
\ba{
\frac{1}{N} \sum_{i,j=1}^K \IE\bcls{(M_i-N_i) (M_j-N_j)| W} \delta_{x_i}  \delta_{x_j}
	&=\frac{1}{N} \sum_{i,j=1}^K N_i \bclr{\Ind[i=j]- \tsfrac{N_j}{N}}  \delta_{x_i}  \delta_{x_j} \\
	&=\frac{1}{N} \sum_{i=1}^K N_i \delta_{x_i}^2 - \frac{1}{N^2} \sum_{i,j=1}^K N_i N_j\delta_{x_i}\delta_{x_j} \\
	&= W^{*2}  -W^2.
}
For the conditional expectation of the second term of~\eq{eq:tee1}, basic definitions imply
\ba{
 \frac{1}{N} \sum_{i,j=1}^K \IE \bbcls{\sum_{k=1}^{B_i}\sum_{l=1}^{B_j}\clr{\delta_{x_{ik}} -\delta_{x_i}}\clr{\delta_{x_{jl}} -\delta_{x_j}} \big|W} 
 	&=\frac{1}{N} \sum_{i,j=1}^K \IE \bcls{B_iB_j|W} \clr{\kappa_{x_i} -\delta_{x_i}}\clr{\kappa_{x_{j}} -\delta_{x_j}}.
}
Finally, for the conditional expectation of~\eq{eq:tee2}, we again use the multinomial covariance formula, yielding
\ba{
\frac{1}{N}\sum_{i,j=1}^K \IE\bbbcls{ \sum_{k=1}^{B_i} (M_j-N_j)& \bclr{  \delta_{x_j}\clr{\delta_{x_{ik}} -\delta_{x_i}}+\clr{\delta_{x_{ik}} -\delta_{x_i}}\delta_{x_j}}\big| W} \\
	&=\frac{1}{N}\sum_{i,j=1}^K p(x_i)  N_i\bclr{\Ind[i=j]- \tsfrac{N_j}{N}} \bclr{  \delta_{x_j}\clr{\kappa_{x_{i}} -\delta_{x_i}}+\clr{\kappa_{x_{i}}  -\delta_{x_i}}\delta_{x_j}}.
}
Combining the previous four displays and using the triangle inequality, we have
\ba{
\bnorm{W^{*2}-W^2-\tsfrac{1}{2\lambda}\IE[(W'-W)^2|W]} 
	&\leq \frac{4}{N} \sum_{i,j=1}^K \IE \bcls{B_iB_j|W} + \frac{4 \norm{p}_{\infty}}{N}\sum_{i,j=1}^K  N_i\bclr{\Ind[i=j]+\tsfrac{N_j}{N}} \\
	&=\frac{4}{N}\bbbclr{\IE \bbcls{\bbclr{\sum_{i=1}^K B_i}^2 \big| W}+ 2 N \norm{p}_{\infty} }.
}
Since for $i=1,\ldots,K$, $\law(B_i)$ is stochastically dominated by $\Bi(M_i, \norm{p}_\infty)$, we have that
$\law\bclr{\sum_{i=1}^K B_i|W}$ is stochastically dominated by $\Bi(N, \norm{p}_\infty)$,
and the result follows from the second moment formula for the binomial distribution and easy simplification.
\end{proof}

We next bound $A_3$ from Theorem~\ref{THM3}.

\begin{lemma}\label{lem:A3wf}
With the definitions above,
\ba{
\max\bbbclc{\IE\bbcls{\bnorm{(W'-W)^3}},\IE\bbcls{ \norm{W'-W}  \bnorm{(W'-W)^2}}}  \leq \frac{\IE\bcls{K^{3/2}}}{N^{3/2}}\bbclr{\sqrt{2}+2 C_{N,p}^{1/3}}^{3}, %%%WF A3
}
where $C_{N,p}:=12 \bcls{\clr{N\norm{p}_\infty}^3 + N \norm{p}_{\infty}}$.
\end{lemma}
\begin{proof}
First, from~\eq{eq:wfwo}, we have
\be{
W'-W= \frac{1}{N} \sum_{i=1}^K (M_i-N_i) \delta_{x_i} + \frac{1}{N} \sum_{i=1}^K\sum_{j=1}^{B_i}\bclr{ \delta_{x_{ij}} -\delta_{x_i}}=: X+Y.
}
Thus, for $k=1,2,3$, we have
\ben{\label{eq:wwpkn}
\bnorm{(W'-W)^k}=\sum_{j=0}^k \binom{k}{j} \norm{X^j Y^{k-j}},
}
and we can easily bound
\ben{\label{eq:wwkpknexp}
\norm{X^j Y^{k-j}}\leq  \frac{1}{N^k} \sum_{i_1,\ldots,i_k=1}^{K} \prod_{\ell=1}^j \abs{M_{i_\ell}-N_{i_\ell}}
\prod_{m=j+1}^k (2 B_{i_m}).
}
Applying~\eq{eq:wwpkn} and~\eq{eq:wwkpknexp} and taking expectation, we have 
\besn{\label{eq:wfmom3bd}
\max\bbbclc{\IE\bbcls{\bnorm{(W'-&W)^3}},\IE\bbcls{ \norm{W'-W}  \bnorm{(W'-W)^2}}} \\
	&\leq  \frac{1}{N^3} \sum_{j=0}^3  \binom{3}{j} 2^{3-j} 
			 \IE\bbbclc{ \sum_{i_1,i_2,i_3=1}^{K} \IE\bbcls{\prod_{\ell=1}^j \abs{M_{i_\ell}-N_{i_\ell}}
			\prod_{m=j+1}^3  B_{i_m} \, \big| \, W}}.
}
Applying H\"older's inequality, we bound the conditional expectations of the summands by
\be{
 \IE\bbcls{\prod_{\ell=1}^j \abs{M_{i_\ell}-N_{i_\ell}} \prod_{m=j+1}^3  B_{i_m} \, \big| \, W}
 	 \leq \bbbclr{  \prod_{\ell=1}^j \IE\bcls{ \abs{M_{i_\ell}-N_{i_\ell}}^3 | W} \prod_{m=j+1}^3 \IE\bcls{ B_{i_m}^{3} |  W} }^{1/3}.
}
To bound these expectations, we use Jensen's inequality and standard calculations for binomial moments from Lemma~\ref{lem:binmoms} below, to find that, 
\be{
\IE\bcls{  \abs{M_{i_\ell}-N_{i_\ell}}^3 | W}
	 \leq\bbclr{ \IE\bcls{  \clr{M_{i_\ell}-N_{i_\ell}}^4 | W} }^{3/4} 
	\leq \bclr{3 N_i^2 + N_i}^{3/4} \leq \bclr{2 N_i}^{3/2}.
}
Similarly, we have,
\ba{
\IE\bcls{ B_{i}^{3} |  W} 
	&\leq \IE\bcls{ M_i^3 p(x_{i})^3 +3 M_i^2 p(x_{i})^2 +M_i p(x_{i})   |  W}  \\
	&\leq 5 N_i^3 p(x_{i})^3 +6 N_i^2 p(x_{i})^2 + N_i p(x_{i})  \\
	&\leq 5 N_i^3 \norm{p}_\infty^3 +6 N_i^2 \norm{p}_\infty^2 +N_i \norm{p}_\infty  \\
	&\leq 12 \bcls{\clr{N\norm{p}_\infty}^3 + N \norm{p}_{\infty}}\frac{N_i}{N}\\
	& =: C_{N,p} \frac{N_i}{N}.
}
Combining the last four displays implies~\eq{eq:wfmom3bd} is bounded above by
\be{
 \sum_{j=0}^3  \frac{1}{N^{3-\frac{j}{2}}}  \binom{3}{j} 2^{3-\frac{j}{2}} \bclr{C_{N,p}}^{\frac{3-j}{3}}\IE \bbbcls{\bbclr{\sum_{i=1}^K \bbclr{\frac{N_i}{N}}^{\frac{1}{2}}}^j 
		\bbclr{ \sum_{i=1}^K\bbclr{\frac{N_i}{N}}^{\frac{1}{3}}}^{3-j}}.
}
We further bound this quantity using that for $q\geq1$,
\be{
\sum_{i=1}^{K} \bbclr{\frac{N_i}{N}}^{1/q} \leq K^{(q-1)/q},
}
which follows from H\"older's inequality.
Therefore, using the fact that $K \leq N$, we can upper bound~\eq{eq:wfmom3bd} by
\ba{
\sum_{j=0}^3 \frac{1}{N^{3-\frac{j}{2}}}  \binom{3}{j} 2^{3-\frac{j}{2}} \bclr{ C_{N,p}}^{\frac{3-j}{3}} \IE\bcls{K^{2-\frac{j}{6}} }
		& \leq \frac{\IE\bcls{K^{3/2}}}{N^{3/2}}  \sum_{j=0}^3 \binom{3}{j} 2^{3-\frac{j}{2}} \bclr{ C_{N,p}}^{\frac{3-j}{3}} \\
		&=\frac{\IE\bcls{K^{3/2}}}{N^{3/2}}\bbclr{\sqrt{2}+2 C_{N,p}^{1/3}}^{3}, 
}
as desired.
\end{proof}

The proof of Theorem~\ref{THM:WF} now easily follows by plugging Lemmas~\ref{lem:linwf},~\ref{lem:A1wf},~\ref{lem:A2wf}, and~\ref{lem:A3wf} into Theorem~\ref{THM3}.

\begin{lemma}[Binomial Moments]\label{lem:binmoms}
If $X\sim\Bi(n,p)$, then
\ba{
\IE\bcls{(X-np)^4}
	&=np(1-p)\bclr{3(n-2)p(1-p)+1} \\
	&\leq 3 n^2 p^2 + np,\\
\IE \cls{X^3} 
	&=n^3 p^3 + 3n^2 p^2 (1 -  p) + n p (1 - 3 p + 2 p^2) \\
	&\leq n^3 p^3 + 3n^2 p^2 + n p,\\
\IE \cls{X^2}&\leq n^2 p^2 + np.
}
\end{lemma}

\subsection{Proof of number of types moment bound Theorem~\ref{thm:KN2mombd}}

The primary idea of proof of Theorem~\ref{thm:KN2mombd} is to look at the total number of ancestors  in the entire genealogy going backward in time until there are~$2$ ancestors; call this quantity~$E$, for now. Since every type in the current population had to come from one of these ancestors, the total number of types is upper bounded by $2+X(E,q)$, where $X(E,q)\sim\Bi(E,q)$ is the total number of mutations in $E$. Thus, 
\be{
\IE[K_N^2]\leq 4 +5\IE[E] q+ \IE[E^2] q^2,
}
and so it is enough to bound $\IE[E^2] \leq (6\times 10^6) \bclr{N \log(N)}^2$. In fact, we are able to
show a bit more than this, as we now explain. The sequence of numbers of individuals in the Wright-Fisher genealogy  going backward in time
can be represented as a Markov chain $(X_n)_{n\geq0}$ with transitions probabilities
\be{
\IP(X_{n+1}=i | X_n=j) = N^{-j} S(j,i) \binom{N}{i} i!,
}
where $S(j,i)$ denotes the number of ways to partition $j$ labelled items into $i$ non-empty parts, i.e., Stirling numbers of the second kind. For integers $1\leq x < y \leq N$, let
\be{
\tau_{x,y}:= \sum_{n\geq0} \Ind\bcls{X_n\in(i,j]}
} 
be the number of generations while the number of individuals in the genealogy is in the interval $(x,y]$. 
A key to our approach is the following result, bounding the second moment of $\tau_{x,y}$.

\begin{lemma}\label{lem:durint}
Fix integers $x,y$ with $2 \leq x  <y \leq N$. Then
\be{
\IE\bcls{\tau_{x,y}^2|X_0=N}\leq 21+ 85  \bbbbclr{ \frac{6 N(y-x+1)}{x(x-1)}}^2.
}
\end{lemma}
Considering the number of generations that a genealogy spends in various intervals has figured
into a number of studies of the expected time back to the most recent common ancestor in the Wright-Fisher and related models, such as
\cite{Dalal2002},
\cite{Fill2002},
\cite{McSweeney2008}, and 
\cite{Mohle2004}. However, Lemma~\ref{lem:durint} 
is more general (applying to generic intervals $(x,y]$ rather than specific choices) and detailed (e.g., explicit constants) than these references. 
Our approach, also taken in \cite{Hitczenko2005}, is to 
view the increments $\law(X_{n}-X_{n+1}|X_n)$ as an inhomogeneous renewal process, and the lemma then follows
by comparing to a renewal process with increments stochastically dominated by the smallest increment of $\law(X_{n}-X_{n+1}|X_n=z)$ for $z\in(x,y]$ (which is $z=x$). Now considering 
\be{
E_{x,y}:=\sum_{n\geq0} X_n \Ind\bcls{X_n\in(x,y]},
}
which is the number of individuals in the genealogy from generations where the number of individuals in the genealogy is in the interval $(x,y]$, we can bound
\ben{\label{eq:edgebd}
E_{x,y}\leq y \tau_{x,y},
}
and Lemma~\ref{lem:durint} yields bounds on the second moment of $E_{x,y}$. 
Using this idea gives the following lemma.

\begin{lemma}\label{lem:numedges}
Set $N\geq 3$ and  $2 \leq L< N$. %:=\bfloor{N^{1/2}/\log(N)}\geq 2$. 
Then 
\be{
\IE\bcls{E_{L,N}^2| X_0=N} \leq \bclr{ 6\times 10^6 }\bclr{N \log(N)}^2.
}
\end{lemma}
The key argument for this result 
is to break the interval $(L,N]$ into smaller intervals,
and bound the number of individuals while the genealogy is in the smaller intervals by
the time spent times the upper value of the interval, as per~\eq{eq:edgebd}, and then apply Lemma~\ref{lem:durint}.
Theorem~\ref{thm:KN2mombd} then easily follows as described above, with $E=E_{L,N}$ 
for $L=2$. 

Finally, we remark that the approach would also work for higher moments of $\tau_{x,y}$ and~$E_{x,y}$. We now prove the  results rigorously.

\begin{proof}[Proof of Lemma~\ref{lem:durint}]
Our approach is to couple the increments $X_{n}-X_{n+1}$ that fall in $(x,y]$ to a stochastically smaller i.i.d.\ sequence, so that  $\tau_{x,y}$ is dominated by the number of renewals in $[0,y-x)$ of a renewal process driven by the i.i.d.\ sequence. The key observation is that
$\law(X_{n}-X_{n+1}| X_n=z)$ is stochastically increasing in $z$, which follows easily after noting that it is the same distribution as the number of occupied bins when allocating $z$ balls into $N$ bins. Thus $\tau_{x,y}$ is stochastically dominated from above by the number of renewals in $[0,y-x)$ (inclusive of the initial renewal at zero) in a renewal process with inter-arrival distribution $x-Y$, where $\law(Y)=\law(X_1 | X_0=x)$; also note that we are using that $\law(\tau_{x,y}|X_0=N)\stleq \law(\tau_{x,y}|X_0=y)$, where $\stleq$ denotes stochastic domination, corresponding to starting the renewal process at zero, rather than $y$ minus where it enters (if at all) in the process $(X_n | X_0=N)$.
Now, since $Y$ has the same distribution as the number of occupied bins when allocating $x$ balls into $N$ bins, a moment's thought shows that $\law(x-Y)\geq_{\mathrm{st}}\law(Z)$, where $Z$ is the number of bins with at least two balls when allocating the $x$ balls into $N$ bins. Therefore, the number of renewals in $[0,y-x)$ of  a renewal process driven by $\law(x-Y)$ is stochastically dominated by that of a renewal process driven by $\law(Z)$. If $Z_1,Z_2, \ldots$ are i.i.d.\ distributed as $Z$, and we define $S_n=\sum_{i=1}^n Z_i$, then we have the bound
\ban{
\IE\bcls{\tau_{x,y}^2}&\leq \IE\bbcls{\bbclr{1 + \sum_{n=1}^\infty \Ind[ S_n < y-x]}^2}  \notag \\
	&= 1 + 3 \sum_{n=1}^\infty \IP(S_n < y-x) + 2 \IE \bbclc{ \sum_{1\leq m < n } \Ind[ S_m < y-x]\Ind[ S_n < y-x] } \notag \\  
	& = 1 + 3 \sum_{n=1}^\infty \IP(S_n < y-x) + 2 \sum_{n=2 }^{\infty} (n-1)\IP( S_n < y-x)  \notag \\
	& \leq 1 + 5 \sum_{n=1 }^{\infty} n \IP( S_n < y-x) \notag  \\
	& \leq 1 + 5 \bclr{\tsfrac{y-x+1}{\mu}}^2 + 5 \sum_{n\geq\tsfrac{y-x+1}{\mu} }^{\infty} n \IP\bclr{ S_n-n \mu < -(n \mu - (y-x))},\label{eq:2mombd}
}
where $\mu:=\IE[Z]$.
To bound further, we need some concentration properties of~$\law(Z)$.
From \cite[Appendix and Lemma~3.4]{Bartroff2018}, for any $\theta<0$,
\be{
\IE\bcls{e^{\theta(Z-\mu)}}\leq e^{\frac{\mu \theta^2}{2}},
}
so then for $\theta<0$,
\be{
\IE\bcls{e^{\theta(S_n-n\mu)}} \leq e^{\frac{n \mu \theta^2}{2}}.
}
Therefore, for $t>0$, and $\theta=-t/(n\mu)  <0$, 
\besn{\label{eq:concincz}
\IP\bclr{(S_n-n\mu) \leq  - t} & = \IP\bclr{\theta (S_n-n\mu) \geq -\theta t} \\
	& \leq e^{\theta t+ \frac{n \mu \theta^2}{2} }  \\
	&= e^{-t^2/(2n\mu)}.
}
We use this fact\footnote{It is worth mentioning here why we have shifted the problem to $Z$ rather than $x-Y$. \cite{Bartroff2018} show that the number of occupied bins $Y$ also satisfies a concentration inequality. However, the analog of~\eq{eq:concincz} would be an upper tail, leading to a worse reliance on $t$ in the exponent, and, more importantly,~$\IE[Y]$ is of order $x$ rather than the smaller order $x^2/N$ of $\mu$.} to bound the sum in~\eq{eq:2mombd} as
\ba{
\sum_{n\geq\tsfrac{y-x+1}{\mu}  }^{\infty} &n  \exp\bbbclc{-\frac{(n \mu - (y-x))^2}{2n \mu}} \\
	&=\sum_{n\geq\tsfrac{y-x+1}{\mu}}^{\infty} n  \exp\bbbclc{-\frac{n\mu-(y-x)}{2} \bbbclr{1 - \frac{y-x}{n \mu}}}\\
	&\leq \sum_{n\geq\tsfrac{y-x+1}{\mu} }^{\infty} n  \exp\bbbclc{-\frac{n\mu-(y-x)}{2} \frac{ 1 }{y-x+1}} \\
	&=\exp\bbbclc{\frac{y-x}{2(y-x+1)} } \sum_{n\geq\tsfrac{y-x+1}{\mu} }^{\infty} n  \exp\bbbclc{-\frac{n\mu}{2\clr{y-x+1}}} \\
	&=  \frac{\exp\bbbclc{\frac{ y-x-\mu \ceil{\frac{y-x+1}{\mu}}}{2(y-x+1)}}}{\bbclr{
			1-e^{-\frac{\mu}{2\clr{y-x+1}}}}^{2}}
	\bbbclr{1+
		\bbclr{
			1-e^{-\frac{\mu}{2\clr{y-x+1}}}}\bbclr{\bbceil{\frac{y-x+1}{\mu}}-1}} \\
	&\leq \frac{3}{2} \exp\bbbclc{-\frac{1}{2(y-x+1)}}\bbbclr{
			1-e^{-\frac{\mu}{2\clr{y-x+1}}}}^{-2} \\
	&\leq 4 \bbbclr{1 + 4\bbclr{ \frac{y-x+1}{\mu} }^2},
}
where in the second to last inequality we have used that $1-e^{-t}\leq t$, and in the last inequality that for $t> 0$, 
\be{
\bclr{1-e^{-t}}^{-2}  \leq\bclr{1-e^{-1}}^{-2}\bbclr{ \Ind[t\geq 1]+ t^{-2} \Ind[t<1]}\leq \tsfrac{8}{3}\bclr{1+ t^{-2}}.
}
Plugging this bound into~\eq{eq:2mombd}, we have
\be{
\IE\bcls{\tau_{x,y}^2}\leq 21 + 85 \bclr{\tsfrac{y-x+1}{\mu}}^2. 
}
To complete the proof, we need a lower bound for $\mu$. Write
\be{
Z=\sum_{i=1}^N I_i, 
}
where $I_i$ is the indicator that the $i$th bin has at least two balls in it. The $I_i$ have the same distribution and
\be{
\IP(I_i=1)=1-\tsfrac{x}{N} \bclr{1-\tsfrac{1}{N}}^{x-1} -\bclr{1-\tsfrac{1}{N}}^{x},
}
so that
\be{
\mu=N\bclr{1-\tsfrac{x}{N} \bclr{1-\tsfrac{1}{N}}^{x-1} -\bclr{1-\tsfrac{1}{N}}^{x}}.
}
Expanding this expression as a power series in $(1/N)$ and noting that $2 \leq x \leq N$ so that the alternating series terms are decreasing in absolute value, we find
\be{
\mu \geq \frac{x(x-1)}{2N} - \frac{x(x-1)(x-2  )}{3 N^2},% \leq \mu \leq \frac{x(x-1)}{2N} - \frac{5x(x-1)(x-2)}{24 N^2},
}
and the result now easily follows after noting that, since $2 \leq x \leq N$, 
\be{
\frac{x(x-1)}{2N}\bbbclr{1 - \frac{2 (x-2)}{3 N}}\geq \frac{x(x-1)}{6N}. \qedhere
}
\end{proof}

\begin{proof}[Proof of Lemma~\ref{lem:numedges}]
We first break up the interval $(L,N]$ into three subintervals: $(L,L_1]$, $(L_1,L_2]$, $(L_2,N]$,
where $L_1=\bceil{(N/\log(N))^{1/2}}$ and $L_2=\bfloor{(3N)^{1/2}}$. The first interval is the easiest to handle, since in this range, the genealogy is well-approximated by the coalescent.  
More precisely, if $p_{kk}=\prod_{i=1}^{k-1}\bclr{1-i/N}$ is the probability the Markov chain holds in state $k$,
then the amount of time $T_k$ that the chain spends in state $k$ is stochastically dominated by
a positive geometric distribution with parameter  $(1-p_{kk})$ (the domination because the chain may not visit the state). 
Thus 
\besn{\label{eq:ell1bd}
\IE\bcls{E_{L,L_1}^2}
	&=\IE\bbbbcls{\bbbclr{\sum_{k=L}^{L_1} kT_k}^2} \\
	&=\sum_{k=L}^{L_1} k^2 \IE\bcls{T_k^2} + \sum_{k\not=\ell}^{L_1} k\ell \IE[T_k]\IE[T_\ell] \\
%	&\leq \sum_{k=L}^{L_1} \frac{k^2}{(1-p_{kk})^2} \\
	&\leq 2 \bbbclr{\sum_{k=L}^{L_1} \frac{k}{1- p_{kk}}}^2.
}
Now,  for $k\leq L_1$, a Bonferroni inequality implies that
\be{
p_{kk} \leq 1- \frac{\binom{k}{2}}{N} + \frac{\binom{k}{2}\binom{k-2}{2}}{2N^2}+\frac{k\binom{k-1}{2}}{N^2} 
	\leq 1- \frac{\binom{k}{2}}{N} \bbbclr{ 1 - \frac{1}{4\log(N)}- \frac{1}{\sqrt{N\log(N)}}} \leq 1- \frac{1}{4}\frac{\binom{k}{2}}{N},
}
where we used that $k\leq(N/\log(N))^{1/2}+1$ and $N\geq 3$. 
Plugging this bound into~\eq{eq:ell1bd} gives 
\ben{\label{eq:245}
\IE\bcls{E_{L,L_1}^2} \leq 128 N^2\bbbclr{\sum_{k=L}^{L_1} \frac{1}{k-1} }^2\leq  128 \bclr{N \log(N)}^2.
}

The second interval requires a different argument--also used in the third interval--which is to bound 
the total number of individuals in the genealogy living while the population size is in the interval, by the maximum of the interval times the time spent in the interval.
That is, we can bound
\be{
E_{L_1,L_2} \leq L_2 \,  \tau_{L_1,L_2},
}
and so, using Lemma~\ref{lem:durint} and that $L_2\leq \sqrt{3N}$, $L_1\geq (N/\log(N))^{1/2}$ and $N\geq3$, we find
\besn{\label{eq:245a}
\IE\bcls{E_{L_1,L_2}^2}
			&\leq L_2^2 \IE\bcls{\tau_{L_1,L_2}^2}   \\
			 &\leq 3N \bbbclr{21 + (85)(36) N^2 \frac{ L_2^2}{L_1^2(L_1-1)^2}} \\
			 &\leq 3N  \bbbclr{21 + (85)(36) (3)N \log(N)^2 \bbclr{1-\sqrt{\log(N)/N}}^{-2} } \\
			 &\leq 3N \bbclr{21 + (85)(18) (39)N \log(N)^2 } \\
			 &\leq (86)(64) (39)N^2\log(N)^2.
}

For the third interval, the basic idea is to break the interval $(L_2,N]$  into smaller intervals, and apply the argument of the second interval. To determine the intervals, note that the jump size of the genealogy Markov chain at level~$z$ is of order $z^2/N$, (cf.\ Lemma~\ref{lem:durint}), and we choose the smaller intervals to each capture about one jump, so the two terms from Lemma~\ref{lem:durint} are of comparable order. That is, if $\ell_i$ is the cutoff for the $i$th interval ($\ell_0=L_2$),
then we choose $\ell_{i+1}-\ell_i \approx \ell_i^2/N$. More concretely, set  
\be{
\ell_i=\bbbceil{\frac{N L_2}{N- L_2 i}}, \,\,\, i=0,\ldots, M-1, 
}
where 
\be{
M-1=\max\bbclc{i: \bbceil{\tsfrac{N L_2}{N- L_2 i}} < N},
}
and $\ell_M=N$.
We have
\be{
E_{L_2, N}\leq \sum_{i=0}^{M-1} E_{\ell_i,\ell_{i+1}} \leq \sum_{i=0}^{M-1} \ell_{i+1} \tau_{\ell_i,\ell_{i+1}},
}
so that
\be{
E_{L_2,N}^2 \leq %\sum_{i=0}^{M-1} \ell_{i+1}^2 \tau_{\ell_i,\ell_{i+1}}^2 + 
\sum_{i,j=0}^{M-1} \ell_{i+1}\ell_{j+1}\tau_{\ell_{i},\ell_{i+1}}\tau_{\ell_{j},\ell_{j+1}}.
}
Taking expectation and using the Cauchy-Schwarz inequality, we find
\ben{\label{eq:E2mom}
\IE\bcls{E_{L_2,N}^2}\leq
 \bbbbclr{ \sum_{i=0}^{M-1} \ell_{i+1}\sqrt{ \IE \bcls{\tau_{\ell_i,\ell_{i+1}}^2 }}}^2. 
}
Using Lemma~\ref{lem:durint} to bound  $\IE\bcls{\tau_{\ell_i,\ell_{i+1}}^2}$, noting that $\sqrt{a^2 + b^2} \leq a + b$ for $a,b\geq 0$, we have
\ben{\label{eq:E2moma}
 \sum_{i=0}^{M-1} \ell_{i+1}\sqrt{ \IE \bcls{\tau_{\ell_i,\ell_{i+1}}^2 }} 
 	\leq  	\sum_{i=0}^{M-1} \ell_{i+1}\bbbclr{\sqrt{21}+ 6 \sqrt{85} N \bbbclr{\frac{\ell_{i+1}-\ell_i+1}{\ell_i(\ell_i-1)}} }.
	}
Using the bounds for $j=1,\ldots, M-1$, 
\ben{\label{eq:ljbds}
\frac{N L_2}{N- L_2 j} \leq \ell_j < \frac{N L_2}{N- L_2 j}+1,
}
and, in the second inequality, that $2\leq \frac{NL_2}{N-L_2(i+1)}-\frac{NL_2}{N-L_2i}$ for $L_2 \geq \sqrt{3N}-1$, and, in the penultimate  inequality that $M<\frac{N}{L_2}$, 
we find that for $i=0,\ldots, M-2$, 
\besn{\label{eq:E2momb}
\ell_{i+1}\bbbclr{\sqrt{21}+ 6& \sqrt{85} N \bbbclr{\frac{\ell_{i+1}-\ell_i+1}{\ell_i(\ell_i-1)}} } \\
	&\leq\ell_{i+1}%\bbbclr{ \frac{N L_2}{N- L_2 (i+1)}+1}
	\bbbclr{\sqrt{21}+ 6 \sqrt{85} N \frac{\frac{NL_2}{N-L_2(i+1)}-\frac{NL_2}{N-L_2i}+2}{\bclr{\frac{NL_2}{N-L_2i}}^2\bclr{1-\frac{N-L_2i}{NL_2}}}} \\
	&\leq\ell_{i+1}%\bbbclr{ \frac{N L_2}{N- L_2 (i+1)}+1}
	\bbbclr{\sqrt{21}+ 12 \sqrt{85}\bbbclr{\frac{ L_2}{L_2-1}} \bbbclr{\frac{N-L_2i}{N-L_2(i+1)}}} \\
	&\leq\ell_{i+1}%\bbbclr{ \frac{N L_2}{N- L_2 (i+1)}+1}
	\bbbclr{\sqrt{21}+ 24 \sqrt{85}\bbbclr{\frac{N-L_2(M-2)}{N-L_2(M-1)}}} \\
	&\leq\ell_{i+1}%\bbbclr{ \frac{N L_2}{N- L_2 (i+1)}+1}
	\bbbclr{\sqrt{21}+ 48\sqrt{85} }\\
	&\leq49\sqrt{85} \, \ell_{i+1}.% \bbbclr{ \frac{N L_2}{N- L_2 (i+1)}+1}
}
Also note that, in particular, since $\ell_{M-1}<N$, 
\besn{\label{eq:E2mombb}
\ell_{M-1}\bbbclr{\sqrt{21}+ 6& \sqrt{85} N \bbbclr{\frac{\ell_{M-1}-\ell_{M-2}+1}{\ell_{M-2}(\ell_{M-2}-1)}} }\leq 49\sqrt{85} N.
}
For the final term in the sum~\eq{eq:E2moma}, note that $\ell_M=N$ and $NL_2/(N-ML_2)\geq N$ implies $M\geq (N-L_2)/L_2$, which, using~\eq{eq:ljbds}, implies $\ell_{M-1}\geq N/2$, and so
\besn{\label{eq:E2momc}
\ell_{M}\bbbclr{\sqrt{21}+ 6 \sqrt{85} N \bbbclr{\frac{\ell_{M}-\ell_{M-1}+1}{\ell_{M-1}(\ell_{M-1}-1)}} }
	&\leq N \bbbclr{\sqrt{21}+ 6 \sqrt{85} N \frac{N/2+1}{(N/2)(N/2-1)}} \\
	&\leq N \bbbclr{\sqrt{21}+ 60\sqrt{85} }\\
		&\leq 61\sqrt{85} N.
}
Combining~\eq{eq:E2mom},~\eq{eq:E2moma},~\eq{eq:E2momb} with~\eq{eq:ljbds},~\eq{eq:E2mombb}, and~\eq{eq:E2momc}, and using again that $M < N/L_2$, $N\geq 3$, and $L_2\geq \sqrt{3N}-1$, 
we have
\besn{\label{eq:245b}
\sqrt{\IE\bcls{E_{L_2,N}^2}}
	&\leq 110 \sqrt{85} N
		+49\sqrt{85} \sum_{i=0}^{M-3}\bbbclr{ \frac{N L_2}{N- L_2 (i+1)}+1} \\
	&\leq 110 \sqrt{85} N
		+49\sqrt{85}  \bbbclr{ N  \int_{0}^{M-2} \frac{1}{\frac{N}{L_2} -(x+1)} dx + (M-2) } \\
	&\leq  110 \sqrt{85} N
		+49\sqrt{85}  \bclr{ N \log(N/L_2) + N/L_2 -2} \\
	&\leq 135 \sqrt{85} N \log(N).
}
Finally, we use the Cauchy-Schwarz inequality (for vectors on $\IR^3$) to find
\be{
\IE\bcls{(E_{L,L_1}+E_{L_1,L_2} + E_{L_2,N})^2}\leq 3 \bclr{\IE\bcls{E_{L,L_1}^2} +\IE\bcls{E_{L_1,L_2}^2} + \IE\bcls{E_{L_2,N}^2}},
}
and combining this with~\eq{eq:245},~\eq{eq:245a} and~\eq{eq:245b} implies 
\be{
\IE\bcls{ E_{L,N}^2} \leq 3\bclr{128+ (86)(64) (39)+(85)(135)^2} \bclr{N \log(N)}^2 \leq \bclr{ 6\times 10^6 }\bclr{N \log(N)}^2. \qedhere
}
\end{proof}

\begin{proof}[Proof of Theorem~\ref{thm:KN2mombd}]
For any $2\leq L <N$, $K_N$ is stochastically dominated by $L+X_{L,N}$, where $X_{L,N}\sim\Bi(E_{L,N}, q)$.
Setting $L=2$, applying Lemma~\ref{lem:numedges}, using Jensen's inequality, and that $N\geq3$, we find
\ba{
\IE \bcls{K_N^2} &\leq L^2 + 2L q \IE[E_{L,N}]+ q \IE[E_{L,N}]+q^2 \IE[E_{L,N}^2] \\
	&\leq  L^2 + \bclr{ \sqrt{6}\times 10^3} (2 L+1) q   N \log(N) 
	+ \bclr{ 6\times 10^6}q^2  \bclr{N \log(N)}^2 \\
	&\leq  \log(N)^2  \bclr{4+\bclr{12\times 10^3} Nq+\bclr{ 6\times 10^6} (Nq)^2}. \qedhere
}
\end{proof}

\section{Acknowledgments}
HG would like to thank the School of Mathematics and Statistics at the University of Melbourne for their hospitality while some of this work was done.  NR was supported by Australian Research Council grant DP150101459.
We thank Bob Griffiths for comments regarding the number of types in the finite Wright-Fisher model in stationary, in particular for the suggestion to view the distribution of $K_N$ through the mutations in the genealogy back to the most recent common ancestor, and two referees for their helpful comments. 

%\bibliographystyle{mynatbib}
%\bibliography{/Users/rossn1/Desktop/Dropbox/Bib}

\begin{thebibliography}{58}
\providecommand{\natexlab}[1]{#1}
\providecommand{\url}[1]{\texttt{#1}}
\expandafter\ifx\csname urlstyle\endcsname\relax
  \providecommand{\doi}[1]{doi: #1}\else
  \providecommand{\doi}{doi: \begingroup \urlstyle{rm}\Url}\fi

\bibitem[Aldous(1985)]{Aldous1985}
D.~J. Aldous, 1985.
\newblock Exchangeability and related topics.
\newblock In \emph{\'{E}cole d'\'et\'e de probabilit\'es de {S}aint-{F}lour,
  {XIII}---1983}, volume 1117 of \emph{Lecture Notes in Math.}, pages 1--198.
  Springer, Berlin.

\bibitem[Arratia et~al.(2003)Arratia, Barbour, and Tavar\'{e}]{Arratia2003}
R.~Arratia, A.~D. Barbour, and S.~Tavar\'{e}, 2003.
\newblock \emph{Logarithmic combinatorial structures: a probabilistic
  approach}.
\newblock EMS Monographs in Mathematics. European Mathematical Society (EMS),
  Z\"{u}rich.

\bibitem[Barbour(1988)]{Barbour1988}
A.~D. Barbour, 1988.
\newblock Stein's method and {P}oisson process convergence.
\newblock \emph{J. Appl. Probab.}, \penalty0 (Special Vol. 25A):\penalty0
  175--184.
\newblock A celebration of applied probability.

\bibitem[Barbour(1990)]{Barbour1990}
A.~D. Barbour, 1990.
\newblock Stein's method for diffusion approximations.
\newblock \emph{Probab. Theory Related Fields}, 84\penalty0 (3):\penalty0
  297--322.

\bibitem[Barbour et~al.(1992)Barbour, Holst, and Janson]{Barbour1992}
A.~D. Barbour, L.~Holst, and S.~Janson, 1992.
\newblock \emph{Poisson approximation}, volume~2 of \emph{Oxford Studies in
  Probability}.
\newblock The Clarendon Press Oxford University Press, New York.
\newblock Oxford Science Publications.

\bibitem[Bartroff et~al.(2018)Bartroff, Goldstein, and I\c{s}lak]{Bartroff2018}
J.~Bartroff, L.~Goldstein, and U.~I\c{s}lak, 2018.
\newblock Bounded size biased couplings, log concave distributions and
  concentration of measure for occupancy models.
\newblock \emph{Bernoulli}, 24\penalty0 (4B):\penalty0 3283--3317.

\bibitem[Baxendale(2011)]{Baxendale2011}
P.~Baxendale, 2011.
\newblock T. {E}. {H}arris's contributions to recurrent {M}arkov processes and
  stochastic flows.
\newblock \emph{Ann. Probab.}, 39\penalty0 (2):\penalty0 417--428.

\bibitem[Bhaskar et~al.(2014)Bhaskar, Clark, and Song]{Bhaskar2014}
A.~Bhaskar, A.~G. Clark, and Y.~S. Song, 2014.
\newblock Distortion of genealogical properties when the sample is very large.
\newblock \emph{Proc. Natl. Acad. Sci. USA}, 111\penalty0 (6):\penalty0
  2385--2390.

\bibitem[Bourguin and Campese(2019)]{Bourguin2019}
S.~Bourguin and S.~C. Campese, 2019.
\newblock Approximation of {H}ilbert-valued {G}aussian measures on {D}irichlet
  structures.
\newblock Preprint \url{https://arxiv.org/abs/1905.05127}.

\bibitem[Chatterjee(2014)]{Chatterjee2014}
S.~Chatterjee, 2014.
\newblock A short survey of {S}tein's method.
\newblock In S.~Y. Jang, Y.~R. Kim, D.-W. Lee, and I.~Yie, editors,
  \emph{Proceedings of the {I}nternational {C}ongress of {M}athematicians,
  {S}eoul 2014, Volume {IV}, Invited Lectures}, pages 1--24, Seoul, Korea.
  KYUNG MOON SA Co. Ltd.

\bibitem[Chatterjee and Meckes(2008)]{Chatterjee2008}
S.~Chatterjee and E.~Meckes, 2008.
\newblock Multivariate normal approximation using exchangeable pairs.
\newblock \emph{ALEA Lat. Am. J. Probab. Math. Stat.}, 4:\penalty0 257--283.

\bibitem[Chatterjee and Shao(2011)]{Chatterjee2011a}
S.~Chatterjee and Q.-M. Shao, 2011.
\newblock Nonnormal approximation by {S}tein's method of exchangeable pairs
  with application to the {C}urie-{W}eiss model.
\newblock \emph{Ann. Appl. Probab.}, 21\penalty0 (2):\penalty0 464--483.

\bibitem[Chatterjee et~al.(2005)Chatterjee, Diaconis, and
  Meckes]{Chatterjee2005}
S.~Chatterjee, P.~Diaconis, and E.~Meckes, 2005.
\newblock Exchangeable pairs and {P}oisson approximation.
\newblock \emph{Probab. Surv.}, 2:\penalty0 64--106.

\bibitem[Chatterjee et~al.(2011)Chatterjee, Fulman, and
  R{\"o}llin]{Chatterjee2011}
S.~Chatterjee, J.~Fulman, and A.~R{\"o}llin, 2011.
\newblock Exponential approximation by {S}tein's method and spectral graph
  theory.
\newblock \emph{ALEA Lat. Am. J. Probab. Math. Stat.}, 8:\penalty0 197--223.

\bibitem[Chen and Xia(2004)]{Chen2004}
L.~H.~Y. Chen and A.~Xia, 2004.
\newblock Stein's method, {P}alm theory and {P}oisson process approximation.
\newblock \emph{Ann. Probab.}, 32\penalty0 (3B):\penalty0 2545--2569.

\bibitem[Chen et~al.(2011)Chen, Goldstein, and Shao]{Chen2011}
L.~H.~Y. Chen, L.~Goldstein, and Q.-M. Shao, 2011.
\newblock \emph{Normal approximation by {S}tein's method}.
\newblock Probability and its Applications (New York). Springer, Heidelberg.

\bibitem[Dalal and Schmutz(2002)]{Dalal2002}
A.~Dalal and E.~Schmutz, 2002.
\newblock Compositions of random functions on a finite set.
\newblock \emph{Electron. J. Combin.}, 9\penalty0 (1):\penalty0 Research Paper
  26, 7.

\bibitem[Dawson and Hochberg(1982)]{Dawson1982}
D.~A. Dawson and K.~J. Hochberg, 1982.
\newblock Wandering random measures in the {F}leming-{V}iot model.
\newblock \emph{Ann. Probab.}, 10\penalty0 (3):\penalty0 554--580.

\bibitem[D{\"o}bler(2015)]{Dobler2015}
C.~D{\"o}bler, 2015.
\newblock Stein's method of exchangeable pairs for the beta distribution and
  generalizations.
\newblock \emph{Electron. J. Probab.}, 20:\penalty0 no. 109, 1--34.

\bibitem[Ethier(1990)]{Ethier1990}
S.~N. Ethier, 1990.
\newblock The infinitely-many-neutral-alleles diffusion model with ages.
\newblock \emph{Adv. in Appl. Probab.}, 22\penalty0 (1):\penalty0 1--24.

\bibitem[Ethier and Griffiths(1993)]{Ethier1993a}
S.~N. Ethier and R.~C. Griffiths, 1993.
\newblock The transition function of a {F}leming-{V}iot process.
\newblock \emph{Ann. Probab.}, 21\penalty0 (3):\penalty0 1571--1590.

\bibitem[Ethier and Kurtz(1981)]{Ethier1981}
S.~N. Ethier and T.~G. Kurtz, 1981.
\newblock The infinitely-many-neutral-alleles diffusion model.
\newblock \emph{Adv. in Appl. Probab.}, 13\penalty0 (3):\penalty0 429--452.

\bibitem[Ethier and Kurtz(1986)]{Ethier1986}
S.~N. Ethier and T.~G. Kurtz, 1986.
\newblock \emph{Markov processes}.
\newblock Wiley Series in Probability and Mathematical Statistics: Probability
  and Mathematical Statistics. John Wiley \& Sons Inc., New York.
\newblock Characterization and convergence.

\bibitem[Ethier and Kurtz(1993)]{Ethier1993}
S.~N. Ethier and T.~G. Kurtz, 1993.
\newblock Fleming-{V}iot processes in population genetics.
\newblock \emph{SIAM J. Control Optim.}, 31\penalty0 (2):\penalty0 345--386.

\bibitem[Ethier and Kurtz(1994)]{Ethier1994}
S.~N. Ethier and T.~G. Kurtz, 1994.
\newblock Convergence to {F}leming-{V}iot processes in the weak atomic
  topology.
\newblock \emph{Stochastic Process. Appl.}, 54\penalty0 (1):\penalty0 1--27.

\bibitem[Ewens(2004)]{Ewens2004}
W.~J. Ewens, 2004.
\newblock \emph{Mathematical population genetics. {I}}, volume~27 of
  \emph{Interdisciplinary Applied Mathematics}.
\newblock Springer-Verlag, New York, second edition.
\newblock Theoretical introduction.

\bibitem[Feng(2010)]{Feng2010}
S.~Feng, 2010.
\newblock \emph{The {P}oisson-{D}irichlet distribution and related topics}.
\newblock Probability and its Applications (New York). Springer, Heidelberg.
\newblock Models and asymptotic behaviors.

\bibitem[Fill(2002)]{Fill2002}
J.~A. Fill, 2002.
\newblock On compositions of random functions on a finite set.
\newblock Preprint.

\bibitem[Fleming and Viot(1979)]{Fleming1979}
W.~H. Fleming and M.~Viot, 1979.
\newblock Some measure-valued {M}arkov processes in population genetics theory.
\newblock \emph{Indiana Univ. Math. J.}, 28\penalty0 (5):\penalty0 817--843.

\bibitem[Fu(2006)]{Fu2006}
Y.-X. Fu, 2006.
\newblock Exact coalescent for the {W}right-{F}isher model.
\newblock \emph{Theoret. Population Biol.}, 69\penalty0 (4):\penalty0 385--394.

\bibitem[Fulman and Ross(2013)]{Fulman2013}
J.~Fulman and N.~Ross, 2013.
\newblock Exponential approximation and {S}tein's method of exchangeable pairs.
\newblock \emph{ALEA Lat. Am. J. Probab. Math. Stat.}, 10\penalty0
  (1):\penalty0 1--13.

\bibitem[Gan et~al.(2017)Gan, R\"ollin, and Ross]{Gan2017}
H.~L. Gan, A.~R\"ollin, and N.~Ross, 2017.
\newblock Dirichlet approximation of equilibrium distributions in {C}annings
  models with mutation.
\newblock \emph{Adv. in Appl. Probab.}, 49\penalty0 (3):\penalty0 927--959.

\bibitem[Ghosal(2010)]{Ghosal2010}
S.~Ghosal, 2010.
\newblock The {D}irichlet process, related priors and posterior asymptotics.
\newblock In \emph{Bayesian nonparametrics}, volume~28 of \emph{Camb. Ser.
  Stat. Probab. Math.}, pages 35--79. Cambridge Univ. Press, Cambridge.

\bibitem[Gorham et~al.(2019)Gorham, Duncan, Vollmer, and Mackey]{Gorham2019}
J.~Gorham, A.~B. Duncan, S.~J. Vollmer, and L.~Mackey, 2019.
\newblock Measuring sample quality with diffusions.
\newblock \emph{Ann. Appl. Probab.}, 29\penalty0 (5):\penalty0 2884--2928.

\bibitem[G{{\"o}}tze(1991)]{Gotze1991}
F.~G{{\"o}}tze, 1991.
\newblock On the rate of convergence in the multivariate {CLT}.
\newblock \emph{Ann. Probab.}, 19\penalty0 (2):\penalty0 724--739.

\bibitem[Hitczenko and Pemantle(2005)]{Hitczenko2005}
P.~Hitczenko and R.~Pemantle, 2005.
\newblock Central limit theorem for the size of the range of a renewal process.
\newblock \emph{Statist. Probab. Lett.}, 72\penalty0 (3):\penalty0 249--264.

\bibitem[Kallenberg(2017)]{Kallenberg2017}
O.~Kallenberg, 2017.
\newblock \emph{Random measures, theory and applications}, volume~77 of
  \emph{Probability Theory and Stochastic Modelling}.
\newblock Springer, Cham.

\bibitem[Karlin and McGregor(1967)]{Karlin1967}
S.~Karlin and J.~McGregor, 1967.
\newblock The number of mutant forms maintained in a population.
\newblock In \emph{Proceedings of the Fifth {B}erkeley Symposium on
  Mathematical Statistics and Probability, Volume 4: Biology and Problems of
  Health}, pages 415--438, Berkeley, Calif. University of California Press.

\bibitem[Kasprzak(2017{\natexlab{a}})]{Kasprzak2017a}
M.~J. Kasprzak, 2017{\natexlab{a}}.
\newblock Diffusion approximations via {S}tein's method and time changes.
\newblock Preprint \url{https://arxiv.org/abs/1701.07633}.

\bibitem[Kasprzak(2017{\natexlab{b}})]{Kasprzak2017c}
M.~J. Kasprzak, 2017{\natexlab{b}}.
\newblock Multivariate functional approximations via {S}tein's method of
  exchangeable pairs.
\newblock Preprint \url{https://arxiv.org/abs/1710.09263}.

\bibitem[Kasprzak(2020)]{Kasprzak2020}
M.~J. Kasprzak, 2020.
\newblock Stein's method for multivariate {B}rownian approximations of sums
  under dependence.
\newblock \emph{Stochastic Process. Appl.}, 130\penalty0 (8):\penalty0
  4927--4967.

\bibitem[Kingman(1975)]{Kingman1975}
J.~F.~C. Kingman, 1975.
\newblock Random discrete distribution.
\newblock \emph{J. Roy. Statist. Soc. Ser. B}, 37:\penalty0 1--22.
\newblock With a discussion by S. J. Taylor, A. G. Hawkes, A. M. Walker, D. R.
  Cox, A. F. M. Smith, B. M. Hill, P. J. Burville, T. Leonard and a reply by
  the author.

\bibitem[Lessard(2007)]{Lessard2007}
S.~Lessard, 2007.
\newblock An exact sampling formula for the {W}right-{F}isher model and a
  solution to a conjecture about the finite-island model.
\newblock \emph{Genetics}, 177\penalty0 (2):\penalty0 1249--1254.

\bibitem[Lessard(2010)]{Lessard2010}
S.~Lessard, 2010.
\newblock Recurrence equations for the probability distribution of sample
  configurations in exact population genetics models.
\newblock \emph{J. Appl. Probab.}, 47\penalty0 (3):\penalty0 732--751.

\bibitem[McSweeney and Pittel(2008)]{McSweeney2008}
J.~K. McSweeney and B.~G. Pittel, 2008.
\newblock Expected coalescence time for a nonuniform allocation process.
\newblock \emph{Adv. in Appl. Probab.}, 40\penalty0 (4):\penalty0 1002--1032.

\bibitem[Meyn and Tweedie(1993)]{Meyn1993}
S.~P. Meyn and R.~L. Tweedie, 1993.
\newblock \emph{Markov chains and stochastic stability}.
\newblock Communications and Control Engineering Series. Springer-Verlag
  London, Ltd., London.

\bibitem[M{\"o}hle(2004)]{Mohle2004}
M.~M{\"o}hle, 2004.
\newblock The time back to the most recent common ancestor in exchangeable
  population models.
\newblock \emph{Adv. in Appl. Probab.}, 36\penalty0 (1):\penalty0 78--97.

\bibitem[Petrov(2009)]{Petrov2009}
L.~A. Petrov, 2009.
\newblock A two-parameter family of infinite-dimensional diffusions on the
  {K}ingman simplex.
\newblock \emph{Funktsional. Anal. i Prilozhen.}, 43\penalty0 (4):\penalty0
  45--66.

\bibitem[Pitman(2006)]{Pitman2006}
J.~Pitman, 2006.
\newblock \emph{Combinatorial stochastic processes}, volume 1875 of
  \emph{Lecture Notes in Mathematics}.
\newblock Springer-Verlag, Berlin.
\newblock Lectures from the 32nd Summer School on Probability Theory held in
  Saint-Flour, July 7--24, 2002, With a foreword by Jean Picard.

\bibitem[Reinert(2005)]{Reinert2005}
G.~Reinert, 2005.
\newblock Three general approaches to {S}tein's method.
\newblock In \emph{An introduction to {S}tein's method}, volume~4 of
  \emph{Lect. Notes Ser. Inst. Math. Sci. Natl. Univ. Singap.}, pages 183--221.
  Singapore Univ. Press, Singapore.

\bibitem[Reinert and R{{\"o}}llin(2009)]{Reinert2009}
G.~Reinert and A.~R{{\"o}}llin, 2009.
\newblock Multivariate normal approximation with {S}tein's method of
  exchangeable pairs under a general linearity condition.
\newblock \emph{Ann. Probab.}, 37\penalty0 (6):\penalty0 2150--2173.

\bibitem[Rinott and Rotar(1997)]{Rinott1997}
Y.~Rinott and V.~Rotar, 1997.
\newblock On coupling constructions and rates in the {CLT} for dependent
  summands with applications to the antivoter model and weighted
  {$U$}-statistics.
\newblock \emph{Ann. Appl. Probab.}, 7\penalty0 (4):\penalty0 1080--1105.

\bibitem[R{\"o}llin(2007)]{Rollin2007}
A.~R{\"o}llin, 2007.
\newblock Translated {P}oisson approximation using exchangeable pair couplings.
\newblock \emph{Ann. Appl. Probab.}, 17\penalty0 (5-6):\penalty0 1596--1614.

\bibitem[R{\"o}llin(2008)]{Rollin2008}
A.~R{\"o}llin, 2008.
\newblock A note on the exchangeability condition in {S}tein's method.
\newblock \emph{Statist. Probab. Lett.}, 78\penalty0 (13):\penalty0 1800--1806.

\bibitem[Ross(2011)]{Ross2011}
N.~Ross, 2011.
\newblock Fundamentals of {S}tein's method.
\newblock \emph{Probab. Surv.}, 8:\penalty0 210--293.

\bibitem[Stein(1972)]{Stein1972}
C.~Stein, 1972.
\newblock A bound for the error in the normal approximation to the distribution
  of a sum of dependent random variables.
\newblock In \emph{Proceedings of the {S}ixth {B}erkeley {S}ymposium on
  {M}athematical {S}tatistics and {P}robability ({U}niv. {C}alifornia,
  {B}erkeley, {C}alif., 1970/1971), {V}ol. {II}: {P}robability theory}, pages
  583--602. Univ. California Press, Berkeley, Calif.

\bibitem[Stein(1986)]{Stein1986}
C.~Stein, 1986.
\newblock \emph{Approximate computation of expectations}.
\newblock Institute of Mathematical Statistics Lecture Notes---Monograph
  Series, 7. Institute of Mathematical Statistics, Hayward, CA.

\bibitem[Wright(1949)]{Wright1949}
S.~Wright, 1949.
\newblock Adaptation and selection.
\newblock \emph{Genetics, paleontology and evolution}, pages 365--389.

\end{thebibliography}

\end{document}